\documentclass[10pt]{conm-p-l}
\usepackage{xypic,amscd,amssymb,latexsym}
\makeatletter
\def\@setcopyright{}
\makeatother
\renewcommand{\-}{\hbox{-}}
\newcommand{\X}{\raisebox{0.43ex}{$\chi$}}
\newcommand{\eins}{\mbox{\rm 1\hspace{-0.24em}l}}
\renewcommand{\mod}{\operatorname{mod}\nolimits}
\newcommand{\Sub}{\operatorname{Sub}\nolimits}
\newcommand{\add}{\operatorname{add}\nolimits}
\newcommand{\Sets}{\operatorname{Sets}\nolimits}
\newcommand{\Hom}{\operatorname{Hom}\nolimits}
\newcommand{\End}{\operatorname{End}\nolimits}
\newcommand{\Aut}{\operatorname{Aut}\nolimits}
\newcommand{\Irr}{\operatorname{Irr}\nolimits}
\newcommand{\Inn}{\operatorname{Inn}\nolimits}
\newcommand{\Out}{\operatorname{Out}\nolimits}
\renewcommand{\Im}{\operatorname{Im}\nolimits}
\newcommand{\Ker}{\operatorname{Ker}\nolimits}
\newcommand{\Coker}{\operatorname{Coker}\nolimits}
\newcommand{\rad}{\operatorname{rad}\nolimits}
\newcommand{\soc}{\operatorname{soc}\nolimits}
\newcommand{\ann}{\operatorname{ann}\nolimits}
\newcommand{\Tor}{\operatorname{Tor}\nolimits}
\newcommand{\Tr}{\operatorname{Tr}\nolimits}
\newcommand{\GL}{\operatorname{GL}\nolimits}
\newcommand{\GF}{\operatorname{GF}\nolimits}
\newcommand{\Mat}{\operatorname{Mat}\nolimits}
\newcommand{\PSL}{\operatorname{PSL}\nolimits}
\newcommand{\U}{\operatorname{U}\nolimits}
\newcommand{\Sp}{\operatorname{Sp}\nolimits}
\newcommand{\SL}{\operatorname{SL}\nolimits}
\newcommand{\Fi}{\operatorname{Fi}\nolimits}
\newcommand{\Co}{\operatorname{Co}\nolimits}
\newcommand{\J}{\operatorname{J}\nolimits}
\newcommand{\kar}{\operatorname{char}\nolimits}
\newcommand{\Ext}{\operatorname{Ext}\nolimits}
\newcommand{\op}{{\operatorname{op}}}
\newcommand{\tr}{{\operatorname{tr}}}
\newcommand{\Ab}{\operatorname{Ab}\nolimits}
\newcommand{\id}{\operatorname{id}\nolimits}
\newcommand{\comp}{\operatorname{\scriptstyle\circ}}
\newcommand{\E}{\operatorname{\mathcal E}\nolimits}
\newcommand{\I}{\operatorname{\mathcal I}\nolimits}
\newcommand{\M}{\operatorname{\mathcal M}\nolimits}
\renewcommand{\P}{\operatorname{\mathcal P}\nolimits}
\newcommand{\W}{\operatorname{\mathfrak{R}}\nolimits}
\newcommand{\Y}{\operatorname{\mathcal Y}\nolimits}
\newcommand{\w}{\operatorname{\mathcal W}\nolimits}
\newcommand{\G}{\Gamma}
\renewcommand{\L}{\Lambda}
\newcommand{\la}{\lambda}
\renewcommand{\r}{\operatorname{\underline{r}}\nolimits}
\renewcommand{\a}{\underline{\mathrm a}}
\renewcommand{\b}{\underline{\mathrm b}}
\newcommand{\m}{\underline{\mathrm m}}
\newcommand{\dwnsimeq}{\Big\downarrow\lefteqn{\wr}}
\newcommand{\upsimeq}{\Big\uparrow\lefteqn{\wr}}
\newcommand{\subneq}{\hspace*{2pt}\mathrel
{{\raisebox{-.3ex}{\hbox{$\scriptscriptstyle\not$}}}\!\!\!\subseteq}}
\newcommand{\semi}{\mathbin{\vcenter{\hbox{$\scriptscriptstyle|$}}\;\!\!\!\times}}
\newcommand{\longequals}{\Relbar\joinrel=}
\newcommand{\A}{\alpha}
\newcommand{\B}{\beta}
\newcommand{\C}{\gamma}
\newcommand{\D}{\delta}
\newcommand{\T}{\tau}
\newcommand{\Z}{\xi}

\newtheorem{lemma}{Lemma}[section]
\newtheorem{proposition}[lemma]{Proposition}
\newtheorem{corollary}[lemma]{Corollary}
\newtheorem{theorem}[lemma]{Theorem}
\newtheorem{definition}[lemma]{Definition}

\newtheorem{cclass}[lemma]{}
\newtheorem{ctab}[lemma]{}

\theoremstyle{definition}
\newtheorem{remark}[lemma]{Remark}

\renewcommand{\arraystretch}{1.5}

\begin{document}

\title{Simultaneous Constructions of the Sporadic Groups~$\Co_2$    
and~$\Fi_{22}$}

\author{Hyun Kyu Kim}
\address{Department of Mathematics, Cornell University, Ithaca,
  N.Y. 14853, USA}

\author{Gerhard O. Michler}
\address{Department of Mathematics, Cornell University, Ithaca,
  N.Y. 14853, USA}

\begin{abstract}
In this article we give 
self-contained existence proofs for the
sporadic simple groups $\Co_2$ and $\Fi_{22}$  using the second
author's algorithm \cite{michler1} constructing finite simple
groups from irreducible subgroups of $\GL_n(2)$. These 
two sporadic groups were originally discovered by J.~Conway
\cite{conway1} and 
B.~Fischer~\cite{fischer}, respectively, by means of
completely different and unrelated methods. In this article $n =
10$ and the irreducible subgroups are the Mathieu group $\M_{22}$
and its automorphism group 
${\rm Aut}(\M_{22})$. We construct their 
five non-isomorphic extensions $E_i$ by the two $10$-dimensional 
non-isomorphic simple modules of $\M_{22}$ and by the two
$10$-dimensional simple modules of $A_{22} = 
{\rm Aut}(\M_{22})$ over $F = 
{\rm GF}(2)$. In 
two cases we construct the centralizer $H_i =
C_{G_i}(z_i)$ of a $2$-central involution $z_i$ of $E_i$ in any
target simple group $\mathfrak G_i$. Then we prove that all the
conditions of Algorithm 7.4.8 of \cite{michler} are satisfied.
This allows us to construct $\mathfrak G_3 \cong \Co_2$ inside
$\GL_{23}(13)$ and $\mathfrak G_2 \cong \Fi_{22}$ inside
$\GL_{78}(13)$. We also calculate their character tables and
presentations.
\end{abstract}


\keywords{construction of simple groups, sporadic simples groups,
  Conway sporadic simple group ${\rm Co}_2$, Fischer sporadic simply
  group ${\rm Fi}_{22}$}
\subjclass[2000]{20D08, 20D05, 20C40}

\thanks{First published in \textit{Computational group theory and the theory of groups} in Contemporary Mathematics vol. 470, 2008, published by the American Mathematical Society.}



\thanks{The first author kindly acknowledges financial support by the
Hunter R. Rawlings III Cornell Presidential Research Scholars Program
for participation in this research project. This collaboration has
also been supported by the grant NSF/SCREMS DMS-0532/06. Both authors
thank the referee for helpful suggestions and corrections.}

\maketitle

\section{Introduction}

In $1969$ J.H. Conway \cite{conway1} discovered 
three sporadic
simple groups which he defined in terms of the automorphism group
$A = {\rm Aut}(\Lambda)$ of the $24$-dimensional Leech lattice
$\Lambda$, see also \cite{conway2}. The center $Z(A)$ of $A$ has order
$2$, and 
$\Co_1 = A/Z(A)$ is the largest of these
three simple groups. He obtained his groups $\Co_2$ and $\Co_3$ as
stabilizers in $A$ of suitable vectors of the Leech lattice.

Two years later B.~Fischer found 
three sporadic simple groups by
characterizing the finite groups $G$ 
that can be generated by a
conjugacy class $D =z^G$ of 
$p$-transpositions, which means that
the product of 
two elements of $D$ has order $1$, 
$2$, or $p$, where
$p$ is an odd prime. He proved that besides the symmetric groups
$S_n$, the symplectic groups $\Sp_n(2)$, the projective unitary
groups $\U_n(2)$ over the field with $4$ elements and certain
orthogonal 
groups, his three sporadic simple groups 
$\Fi_{22}$, $\Fi_{23}$, and $\Fi'_{24}$ 
describe all $3$-transposition groups, see
\cite{fischer}. Whereas the $24$-dimensional Leech lattice
$\Lambda$ provides a natural action of 
$A = {\rm Aut}(\Lambda)$ 
and has
two well defined $23$-dimensional sublattices on which the Conway
groups $\Co_2$ and $\Co_3$ act, Fischer had to construct a graph
$\mathcal G$ and an action on it for each $3$-transposition group
$G = \langle D \rangle$. As its vertices he took the
$3$-transpositions $x$ of $D$. Two distinct elements $x, y \in D$
are called to be {\em connected} and joined by an edge $(x,y)$ in
$\mathcal G$ if they commute in $G$. He showed that each of the
groups considered in his theorem has a natural representation as
an automorphism group of its graph $\mathcal G$. See also
\cite{aschbacher} for a coherent account on Fischer's theorem.

It is the purpose of this paper to provide uniform existence
proofs for the simple sporadic groups $\Co_2$ and $\Fi_{22}$ by
means of Algorithm 2.5 of \cite{michler1} constructing finite
simple groups from irreducible subgroups $T$ of $\GL_n(2)$. Here
we deal with the case $n = 10$ and the irreducible subgroups
$\M_{22}$ and $A_{22} = 
{\rm Aut}(\M_{22})$. Our methods are purely
algebraic and do not require any geometric insight. Furthermore,
they are not restricted to the study of sporadic groups. Since
every finite simple group 
$G$ has finitely many irreducible
representations 
$V$, Algorithm 2.5 of \cite{michler1} can be
applied again to the extensions $E$ of $G$ by $V$. In particular,
$\M_{22}$ has been constructed this way from the irreducible
subgroup $S_5$ of $\GL_4(2)$, see Proposition 8.2.4 of
\cite{michler}. Here the algorithm is applied again to the simple
modules of $\M_{22}$ and its automorphism group $A_{22}$.

In 
Section~$2$ we state some known facts about $\M_{22}$, its
automorphism group~$A_{22}$ and their irreducible representations
over 
$F = {\rm GF}(2)$. 
As shown in~\cite{james}, the groups
$\M_{22}$ and
$A_{22}$ 
each have two simple modules $V_1, V_2$ and $V_3, V_4$ of
dimension $10$, respectively. We show that the second
cohomological dimensions 
$\dim_F[H^2(\M_{22},V_i)]$ and 
$\dim_F[H^2(A_{22},V_i)]$ 
are $0$ for $i = 1,2$ and 
$i = 4$, but 
$\dim_F[H^2(A_{22}, V_3)] = 1$. Therefore there are 
four split
extensions $E_i = V_i \rtimes \M_{22}$, $i = 1,2$, $E_i = V_i
\rtimes A_{22}$, $i = 3,4$ and one 
non-split extension $E_5$ of
$V_3$ by $A_{22}$. For each of these 
five groups we state a
presentation in terms of generators and relations, see
Lemmas~\ref{l. M22-extensions} and~\ref{l. Aut(M22)-extensions}. We also
construct for each of them a faithful permutation representation,
see Lemma \ref{l. classes}. In all 
five cases it has been checked
that the elementary abelian normal subgroup $V_i$ is a maximal
elementary abelian subgroup of a Sylow $2$-subgroup $S_i$ of $E_i$
as 
required in Step $4$ of Algorithm 2.5 of \cite{michler1}.

In 
Section~$3$ we apply 
Step~$5$ of that algorithm to the
extension groups $E_3$ and~$E_2$. Let $E$ be any of these 
two groups. Then $E$ has a unique conjugacy class $z^{G}$ of
$2$-central involutions. Let $D = C_E(z)$. Using a faithful
permutation representation of $E$ we find in both cases a uniquely
determined 
nonabelian normal subgroup $Q$ of $D$ such that $V =
Q/Z(Q)$ is elementary abelian of order $2^8$, where $Z(Q)$ denotes
the center of $Q$. Furthermore, $Q$ has a complement $W$ in $D$.
In both cases $W$ has an elementary abelian normal Fitting
subgroup $B$ which has a complement $L$ in $W$.

In the first 
case, $E=E_3$, it follows that $W$ has a center $Z(W) =
\langle u \rangle$ of order $2$ and that $L \cong S_6$. Using
Yamaki's Theorem 8.6.6 of \cite{michler} we see that $W$ is
isomorphic to the centralizer of a $2$-central involution of the
symplectic group~$\Sp_6(2)$. Applying now Algorithm 7.4.8 of
\cite{michler} we construct a simple subgroup $K$ of $\GL_8(2)$
such that $C_K(u) \cong W$, $|K : W|$ is odd, and $K \cong
\Sp_6(2)$. In order to perform all these steps we start from the
factor group $D/Z(Q)$. Let $U_1$ be the split extension of $K$
by~$V$. Then we construct all central extensions $H_3$ of $U_1$ by
$Z(Q)$ and check which ones have a Sylow $2$-subgroup $S_3$ which
is isomorphic to the ones of $D$. We prove that this happens
exactly once. Thus the group $H_3$ with center $Z(H_3) = \langle z
\rangle$ is uniquely determined, see Propositions \ref{prop.
D(Co_2)}. Furthermore, we prove in Lemma \ref{l. H(Co_2)} that
$S_3$ has a maximal elementary abelian normal subgroup $A$ of
order $2^{10}$ such that $N_{H_3}(A) \cong D$ as it is required by
Algorithm 2.5 of \cite{michler1}.

In the second 
case, $E=E_2$, it follows that $Q$ has an elementary
abelian center $Z(W) = \langle z_1, z_2 \rangle$ of order $4$.
This time the Fitting subgroup $B$ of the complement $W$ of $Q$
has order $16$ and its complement $L$ in $W$ is isomorphic to the
symmetric group $S_5$. Since the center of $W$ is trivial we apply
Algorithm 2.5 of  \cite{michler1} to construct a subgroup $K$ of
$\GL_8(2)$ such that $N_K(A) \cong W$, $|K : W|$ is odd, and $K
\cong \Aut(\U_4(2))$. Again we start from the factor group
$D/Z(Q)$. Let $U_1$ be the split extension of $K$ by $V$. Since
$Z(Q)$ is elementary abelian of order $4$ we have to construct 
two consecutive central extensions of $U_1$. As above it follows that
there is a uniquely determined extension group $H_2$ which
satisfies all the required conditions of Algorithm 2.5 of
\cite{michler1}, see Proposition \ref{prop. D(Fi_22)} and Lemma
\ref{l. H(Fi_22)}.

In 
Section~$4$ we take the constructed presentation of $H_3$ as
the input of Algorithm 7.4.8 of \cite{michler}. It returns a
finite simple group $\mathfrak G_3$ inside $\GL_{23}(13)$ of order
$2^{18}\cdot 3^6\cdot 5^3\cdot 7\cdot 11\cdot 23$ having a
$2$-central involution $\mathfrak z$ such that $C_{\mathfrak
G_3}(\mathfrak z) \cong H_3$, see Theorem \ref{thm.
existenceCo_2}. We also calculate the character table of
$\mathfrak G_3$. It is equivalent to the one of Conway's second
sporadic group $\Co_2$. For further applications we also state a
finite presentation of $\mathfrak G_3$ in Corollary \ref{cor.
presentCo_2}.

In 
Section~$5$ we apply the same methods to the presentation of
the involution centralizer $H_2$ and 
we obtain a finite simple group
$\mathfrak G_2$ inside $\GL_{78}(13)$ 
of order $2^{17}\nobreak\cdot\nobreak 3^9\cdot 5^2 \cdot 7\cdot 11\cdot 13$ 
having a $2$-central
involution $\mathfrak z$ such that $C_{\mathfrak G_2}(\mathfrak z)
\cong H_2$, 
see Theorem \ref{thm. existenceFi_22}. We calculated
the character table of $\mathfrak G_2$ and verified that it is
equivalent to the one of Fischer's sporadic group $\Fi_{22}$. We
also obtained a presentation of $\mathfrak G_2$, see Corollary
\ref{cor. presentFi_22}.

In the final 
section we show that Algorithm 2.5 of \cite{michler1}
constructs the automorphism group of the Fischer group 
$\Fi_{22}$
when it is applied to the extension $E_4$, see Corollary \ref{cor.
D(Aut(Fi_22))}. However, it fails to construct a centralizer $H_i$
from $E_i$ in the cases $i = 1, 5$. The following diagram
summarizes our experiments with Algorithm 2.5 of~\cite{michler1}
described in this article.
$$
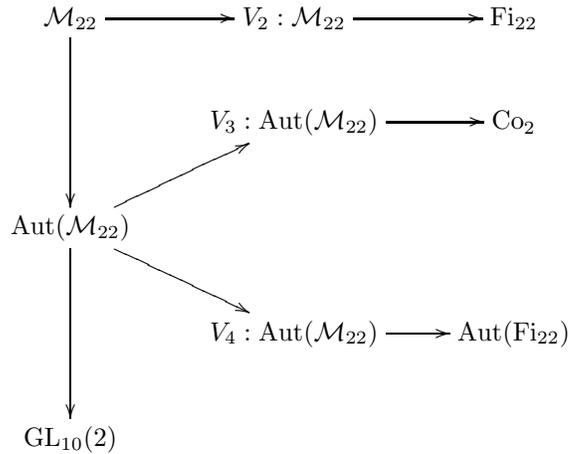
\diagram
\M_{22} \rto \ddto & V_2:\M_{22} \rto & \Fi_{22} & \\
& V_3:{\rm Aut}(\M_{22}) \rto & \Co_2 & \\
{\rm Aut}(\M_{22}) \urto \drto \ddto & & & \\
& V_4:{\rm Aut}(\M_{22}) \rto & {\rm Aut}(\Fi_{22})\\
\GL_{10}(2) & & & \\
\enddiagram
$$

In the 
Appendices we collect all the systems of representatives of
conjugacy classes in terms of the given generators of the local
subgroups $E_i$, $H_i$ and $D_i$ for $i = 3,2$. We also state the
character tables of these groups and the 
four generating matrices
of the matrix group $\mathfrak G_2 \cong \Fi_{22}$.

Concerning our notation and terminology we refer to the books
\cite{carter} and \cite{michler}. The computer algebra system
\textsc{Magma} is described in 
\cite{magma}.

\section{Extensions of Mathieu group $\M_{22}$ and 
${\rm Aut}(\M_{22})$}

The Mathieu group $\M_{22}$ is defined in Definition 8.2.1 of
\cite{michler} by means of generators and relations. This
beautiful presentation is due to J.A. Todd. The irreducible
$2$-modular representations of the Mathieu group $\M_{22}$ were
determined by G. James \cite{james}. Here only the 
two non-isomorphic simple modules $V_i$, $i = 1,2$, of dimension $10$ over
$F = \GF(2)$ will be considered. Todd's permutation representations
of the Mathieu groups are stated in Lemma 8.2.2 of \cite{michler}.
All conditions of Holt's Algorithm \cite{holt}
implemented in \textsc{Magma} are 
satisfied, and it is applied here. Thus we
show that for each simple module $V_i$ there is exactly one
extension $E_i$ of $\M_{22}$ by $V_i$, the split extension. The
presentations of $E_1$ 
and $E_2$ are given in Lemma~\ref{l. M22-extensions}.

Then we show that the automorphism group $A_{22} = 
{\rm Aut}(\M_{22})$ 
has two $10$-dimensional irreducible representations $V_i$, $i =
3,4$. In Lemma \ref{l. Aut(M22)-extensions} we also prove that
$A_{22}$ has a 
non-split extension $E_5$ by $V_3$ besides the 
two split extensions by $V_3$ and $V_4$. All extension groups are
described here by generators and relations. We also give faithful
permutation representations of all 
five extension groups $E_i$, see
Lemma \ref{l. classes}. They yield the basic information about
these groups 
that will be used in the following sections.

In order to avoid any ambiguities we specify our concrete
definition of a 
semidirect product. 
This is necessary to document
the close relation between the irreducible representations $V_i$
of $\M_{22}$ and $A_{22}$ and the documented presentations of the
extension groups $E_i$, $1 \le i \le 5$. Furthermore, these
presentations work for all split extensions of linear groups over
prime fields by their natural vector space.

\begin{definition}\label{def. semi-direct product}
Let $F$ be a prime field of characteristic $p>0$, $G$ be a subgroup of
$\GL_n(F)$, and $V=F^n$ be the canonical $n$-dimensional vector
space. 
Then $V$ is a right $FG$-module of
the group algebra $FG$ of $G$ over $F$ with respect to the
multiplication 
$v\cdot g$ defined by the product of the row
vector $v \in V$ times the matrix $g \in G$. For each matrix $g
\in G$ let $g^* = [g^{-1}]^T$ be the transpose of the inverse
matrix $g^{-1}$ of $g$. Then by Definition 3.4.4 of \cite{michler}
$V$ becomes a right $FG$-module under the multiplication $[v,g] :=
[g^{*}\cdot(v^T)]^T$ where $\cdot$ denotes the product of the
matrix $g^{*}$ times the column vector $v^T$. This right
$FG$-module $V^*$ is called the {\em dual} $FG$-module of~$V$.

The {\em 
semidirect product} $V\rtimes G$ of $V$ and $G$ consists
of all pairs 
$(v,g)$, with $v\in V$ and $g\in G$. Its multiplication $*$ is
defined by:
\[ (v_1,g_1)*(v_2,g_2) := (v_1 + [v_2,g_1], g_1g_2)\ \mbox{for all
  $v_1,v_2\in V$ and all $g_1,g_2\in G$.}\]
If we identify each $v\in V$ with $(v,1)$ and each $g\in G$ with $(0,g)$, 
then we have that $v*g = (v+[0,1],1g) = (v,g)$ for each $v\in V$ and $g\in G$.
\end{definition}
The proof of the following 
well-known result is routine and
therefore omitted.

\begin{lemma}\label{l. presentation-split} Let $G = \langle 
x_i \,|\, 1\le i \le r \rangle$
be a finitely generated subgroup of $\GL_n(F)$ with set of
defining relations ${\mathcal R}(G)$ in the given generators
$x_i$. Let $V = F^n$ be the canonical $n$-dimensional vector space
with standard basis 
$\{e_j \,|\, 1 \le j \le n\}$.
Let ${\mathcal R_1}(V\rtimes G)$ be the set of all relations:
\begin{alignat*}{2}
e_j^p & = 1 & &\mbox{for all $1\leq j\leq n$,}\\
e_k*e_j &= e_j*e_k &\quad &\mbox{for all $1\leq j,k\leq n$.}
\end{alignat*}
Let ${\mathcal R_2}(V\rtimes G)$ be the set of all relations:
\[x_i*e_j*x_i^{-1} = {e_1}^{a(i)_{1,j}} * \cdots * {e_n}^{a(i)_{n,j}}
\quad \mbox{for all $1 \le i \le r$,  $1 \le j \le n$,}\]
where $(a(i)_{j1},a(i)_{j2}, \ldots,a(i)_{jn}) =
[(x_i)^*\cdot((e_j)^T)]^T$. Then
$${\mathcal R}(V\rtimes G) = {\mathcal R_1}(V\rtimes G) \cup {\mathcal R_2}(V\rtimes G) \cup {\mathcal R}(G).$$
\end{lemma}
\begin{remark}(a) For a given finitely presented group $G$ the set
${\mathcal R}(G)$ of defining relations is equal for all its split
extensions $E$ by elementary abelian $p$-groups $V$. Furthermore,
the set ${\mathcal R_1}(V\rtimes G)$ of relations describing the
multiplication in the elementary abelian group is uniquely
determined by the dimension of $V$ over $F$. Therefore this set of
relations is abbreviated by ${\mathcal R_1}(V\rtimes G)$ in the
rest of this article. In particular, it suffices to 
state explicitly only the set of essential relations ${\mathcal
R_2}(V\rtimes G)$ of the split extension $V\rtimes G$.

(b) The 
semidirect products $V\rtimes G$ constructed in this
article by means of Lemma \ref{l. presentation-split} are
isomorphic to the 
semidirect products of $G$ by $V$ calculated by
means of \textsc{Magma}. However, sometimes the sets ${\mathcal R}(V\rtimes
G)$ of defining relations can be different. Furthermore, Lemma
\ref{l. presentation-split} can be applied to large groups and it
is independent of \textsc{Magma}.
\end{remark}

In the following 
$F = \GF(2)$ and $p = 2$.

\begin{lemma}\label{l. M22-extensions}
Let $\M_{22} = \langle a,b,c,d,t,g,h,i \rangle$ be the finitely
presented group of Definition 8.2.1 of \cite{michler}. Then the
following statements hold:
\begin{enumerate}
\item[\rm(a)] A faithful permutation representation of degree $24$
of $\M_{22}$ is stated in Lemma 8.2.2 of \cite{michler}.

\item[\rm(b)] The first irreducible representation $V_1$ of
$\M_{22}$ is described by the following matrices:

{\renewcommand{\arraystretch}{0.5}
\scriptsize
\begin{alignat*}{2}
Ma_1 &= \left( \begin{array}{*{10}{c@{\,}}c}
1 & 0 & 1 & 0 & 0 & 0 & 0 & 0 & 0 & 0\\
0 & 0 & 0 & 1 & 0 & 0 & 0 & 0 & 0 & 0\\
0 & 0 & 1 & 0 & 0 & 0 & 0 & 0 & 0 & 0\\
0 & 1 & 0 & 0 & 0 & 0 & 0 & 0 & 0 & 0\\
0 & 0 & 1 & 0 & 0 & 0 & 0 & 1 & 0 & 0\\
0 & 0 & 0 & 0 & 0 & 0 & 0 & 0 & 1 & 0\\
0 & 0 & 1 & 0 & 0 & 0 & 1 & 0 & 0 & 0\\
0 & 0 & 1 & 0 & 1 & 0 & 0 & 0 & 0 & 0\\
0 & 0 & 0 & 0 & 0 & 1 & 0 & 0 & 0 & 0\\
0 & 0 & 1 & 0 & 0 & 0 & 0 & 0 & 0 & 1
\end{array} \right),
&\quad
Mb_1 &= \left( \begin{array}{*{10}{c@{\,}}c}
1 & 0 & 1 & 0 & 1 & 0 & 0 & 1 & 0 & 0\\
0 & 0 & 0 & 0 & 1 & 1 & 0 & 0 & 0 & 0\\
0 & 0 & 1 & 0 & 0 & 0 & 0 & 0 & 0 & 0\\
0 & 0 & 1 & 0 & 0 & 0 & 0 & 1 & 1 & 0\\
0 & 0 & 0 & 0 & 0 & 0 & 0 & 1 & 0 & 0\\
0 & 1 & 0 & 0 & 0 & 0 & 0 & 1 & 0 & 0\\
0 & 0 & 1 & 0 & 0 & 0 & 1 & 0 & 0 & 0\\
0 & 0 & 0 & 0 & 1 & 0 & 0 & 0 & 0 & 0\\
0 & 0 & 1 & 1 & 1 & 0 & 0 & 0 & 0 & 0\\
0 & 0 & 1 & 0 & 1 & 0 & 0 & 1 & 0 & 1
\end{array} \right),\displaybreak[0]\\[4pt]
Mc_1 &= \left( \begin{array}{*{10}{c@{\,}}c}
1 & 0 & 0 & 0 & 0 & 1 & 0 & 0 & 1 & 0\\
0 & 0 & 1 & 0 & 1 & 1 & 0 & 0 & 0 & 0\\
0 & 0 & 1 & 0 & 0 & 0 & 0 & 0 & 0 & 0\\
0 & 0 & 0 & 0 & 0 & 0 & 0 & 1 & 1 & 0\\
0 & 1 & 0 & 0 & 0 & 0 & 0 & 0 & 1 & 0\\
0 & 0 & 1 & 0 & 0 & 0 & 0 & 0 & 1 & 0\\
0 & 0 & 1 & 0 & 0 & 0 & 1 & 0 & 0 & 0\\
0 & 0 & 1 & 1 & 0 & 1 & 0 & 0 & 0 & 0\\
0 & 0 & 1 & 0 & 0 & 1 & 0 & 0 & 0 & 0\\
0 & 0 & 1 & 0 & 0 & 1 & 0 & 0 & 1 & 1
\end{array} \right),&\quad
Md_1 &= \left( \begin{array}{*{10}{c@{\,}}c}
1 & 0 & 0 & 1 & 0 & 1 & 0 & 1 & 0 & 0\\
0 & 0 & 1 & 0 & 1 & 0 & 0 & 0 & 1 & 0\\
0 & 1 & 0 & 1 & 1 & 1 & 0 & 1 & 1 & 0\\
0 & 0 & 0 & 0 & 0 & 1 & 0 & 1 & 0 & 0\\
0 & 0 & 0 & 1 & 0 & 0 & 0 & 0 & 1 & 0\\
0 & 0 & 0 & 0 & 1 & 0 & 0 & 1 & 1 & 0\\
0 & 0 & 0 & 1 & 1 & 1 & 1 & 0 & 0 & 0\\
0 & 0 & 0 & 1 & 1 & 0 & 0 & 1 & 1 & 0\\
0 & 0 & 0 & 0 & 1 & 1 & 0 & 1 & 0 & 0\\
0 & 0 & 0 & 1 & 1 & 1 & 0 & 0 & 0 & 1
\end{array} \right),\displaybreak[0]\\[4pt]
Mt_1 &= \left( \begin{array}{*{10}{c@{\,}}c}
1 & 0 & 1 & 1 & 0 & 0 & 0 & 0 & 1 & 0\\
0 & 1 & 1 & 0 & 1 & 1 & 0 & 0 & 1 & 0\\
0 & 0 & 1 & 1 & 1 & 1 & 0 & 0 & 0 & 0\\
0 & 0 & 1 & 1 & 1 & 0 & 0 & 0 & 0 & 0\\
0 & 0 & 1 & 0 & 1 & 1 & 0 & 0 & 0 & 0\\
0 & 0 & 1 & 1 & 0 & 1 & 0 & 0 & 0 & 0\\
0 & 0 & 1 & 0 & 0 & 0 & 1 & 0 & 0 & 0\\
0 & 0 & 0 & 1 & 0 & 1 & 0 & 0 & 1 & 0\\
0 & 0 & 0 & 1 & 1 & 0 & 0 & 1 & 1 & 0\\
0 & 0 & 0 & 0 & 1 & 0 & 0 & 0 & 1 & 1
\end{array} \right),&\quad
Mg_1 &= \left( \begin{array}{*{10}{c@{\,}}c}
1 & 0 & 1 & 0 & 0 & 0 & 0 & 0 & 0 & 0\\
0 & 0 & 1 & 0 & 0 & 0 & 1 & 1 & 1 & 0\\
0 & 0 & 1 & 0 & 0 & 0 & 0 & 0 & 0 & 0\\
0 & 0 & 1 & 1 & 0 & 0 & 0 & 1 & 1 & 0\\
0 & 0 & 1 & 0 & 1 & 0 & 0 & 1 & 0 & 0\\
0 & 0 & 1 & 0 & 0 & 1 & 0 & 0 & 1 & 0\\
0 & 1 & 1 & 0 & 0 & 0 & 0 & 1 & 1 & 0\\
0 & 0 & 0 & 0 & 0 & 0 & 0 & 1 & 0 & 0\\
0 & 0 & 0 & 0 & 0 & 0 & 0 & 0 & 1 & 0\\
0 & 0 & 0 & 0 & 0 & 0 & 0 & 1 & 0 & 1
\end{array} \right),\displaybreak[0]\\[4pt]
Mh_1 &= \left( \begin{array}{*{10}{c@{\,}}c}
1 & 0 & 0 & 0 & 0 & 0 & 0 & 0 & 1 & 0\\
0 & 1 & 0 & 0 & 0 & 1 & 0 & 0 & 0 & 0\\
0 & 0 & 1 & 0 & 0 & 1 & 0 & 0 & 1 & 0\\
0 & 0 & 0 & 1 & 0 & 0 & 0 & 0 & 1 & 0\\
0 & 0 & 0 & 0 & 0 & 1 & 0 & 1 & 0 & 0\\
0 & 0 & 0 & 0 & 0 & 1 & 0 & 0 & 0 & 0\\
0 & 0 & 0 & 0 & 0 & 0 & 0 & 0 & 1 & 1\\
0 & 0 & 0 & 0 & 1 & 1 & 0 & 0 & 0 & 0\\
0 & 0 & 0 & 0 & 0 & 0 & 0 & 0 & 1 & 0\\
0 & 0 & 0 & 0 & 0 & 0 & 1 & 0 & 1 & 0
\end{array} \right)&\quad \text{and }
Mi_1 &= \left( \begin{array}{*{10}{c@{\,}}c}
0 & 0 & 0 & 0 & 1 & 1 & 0 & 0 & 0 & 1\\
0 & 1 & 1 & 1 & 1 & 0 & 0 & 0 & 0 & 0\\
0 & 0 & 1 & 1 & 1 & 1 & 0 & 0 & 0 & 0\\
0 & 0 & 1 & 0 & 1 & 0 & 0 & 0 & 0 & 0\\
0 & 0 & 1 & 0 & 1 & 1 & 0 & 0 & 0 & 0\\
0 & 0 & 0 & 1 & 1 & 0 & 0 & 0 & 0 & 0\\
0 & 0 & 0 & 1 & 1 & 1 & 1 & 0 & 0 & 0\\
0 & 0 & 0 & 1 & 0 & 1 & 0 & 0 & 1 & 0\\
0 & 0 & 1 & 1 & 0 & 0 & 0 & 1 & 0 & 0\\
1 & 0 & 1 & 1 & 0 & 1 & 0 & 0 & 0 & 0
\end{array} \right).
\end{alignat*}}
\item[\rm(c)] The second irreducible representation $V_2$ of
$\M_{22}$ is described by the transpose inverse matrices of the
generating matrices of $\M_{22}$ defining $V_1$:
$Ma_2 = [Ma_1^{-1}]^{T}$, $Mb_2 = [Mb_1^{-1}]^{T}$, $Mc_2 =
[Mc_1^{-1}]^{T}$, $Md_2 = [Md_1^{-1}]^{T}$, $Mt_2 =
[Mt_1^{-1}]^{T}$, $Mg_2 = [Mg_1^{-1}]^{T}$, $Mh_2 =
[Mh_1^{-1}]^{T}$ and $Mi_2 = [Mi_1^{-1}]^{T}$.

\item[\rm(d)] 
$\dim_F[H^2(\M_{22},V_i)] = 0$ for $i = 1,2$.

\item[\rm(e)] $E_1 = V_1\rtimes \M_{22} = \langle
a,b,c,d,t,g,h,i,v_1,v_2,v_3,v_4,v_5,v_6,v_8,v_8,v_9,v_{10}
\rangle$ has 
set ${\mathcal R}(E_1)$ of defining relations
consisting of ${\mathcal R}(\M_{22})$, ${\mathcal R_1}(V_1\rtimes
\M_{22})$ and the following set ${\mathcal R_2}(V_1\rtimes
\M_{22})$ of essential relations:
\begin{alignat*}{1}
& av_1a^{-1}v_1v_3 = av_2a^{-1}v_4 = av_3a^{-1}v_3 = av_4a^{-1}v_2 =
  av_5a^{-1}v_3v_8 = av_6a^{-1}v_9 = 1,\displaybreak[0]\\ 
& av_7a^{-1}v_3v_7 = av_8a^{-1}v_3v_5 = av_9a^{-1}v_6 =
  av_{10}a^{-1}v_3v_{10} = 1,\displaybreak[0]\\ 
& bv_1b^{-1}v_1v_3v_5v_8 = bv_2b^{-1}v_5v_6 = bv_3b^{-1}v_3 =
  bv_4b^{-1}v_3v_8v_9 = bv_5b^{-1}v_8 = 1,\displaybreak[0]\\ 
& bv_6b^{-1}v_2v_8 = bv_7b^{-1}v_3v_7 = bv_8b^{-1}v_5 =
  bv_9b^{-1}v_3v_4v_5 = bv_{10}b^{-1}v_3v_5v_8v_{10} =
  1,\displaybreak[0]\\
& cv_1c^{-1}v_1v_6v_9 = cv_2c^{-1}v_3v_5v_6 = cv_3c^{-1}v_3
  =cv_4c^{-1}v_8v_9 = cv_5c^{-1}v_2v_9 = 1,\displaybreak[0]\\ 
& cv_6c^{-1}v_3v_9 = cv_7c^{-1}v_3v_7 = cv_8c^{-1}v_3v_4v_6 =
  cv_9c^{-1}v_3v_6 = cv_{10}c^{-1}v_3v_6v_9v_{10} =
  1,\displaybreak[0]\\ 
& dv_1d^{-1}v_1v_4v_6v_8 = dv_2d^{-1}v_3v_5v_9 =
  dv_3d^{-1}v_2v_4v_5v_6v_8v_9 = dv_4d^{-1}v_6v_8 =
  1,\displaybreak[0]\\ 
& dv_5d^{-1}v_4v_9 = dv_6d^{-1}v_5v_8v_9 = dv_7d^{-1}v_4v_5v_6v_7 =
  dv_8d^{-1}v_4v_5v_8v_9 = 1,\displaybreak[0]\\ 
& dv_9d^{-1}v_5v_6v_8 = dv_{10}d^{-1}v_4v_5v_6v_{10} =
  tv_1t^{-1}v_1v_3v_5v_6v_8 = tv_2t^{-1}v_2v_4v_8 =
  1,\displaybreak[0]\\ 
& tv_3t^{-1}v_4v_5v_6 = tv_4t^{-1}v_3v_5 = tv_5t^{-1}v_3v_6 =
  tv_6t^{-1}v_3v_4 = tv_7t^{-1}v_4v_5v_6v_7 = 1, \displaybreak[0]\\
& tv_8t^{-1}v_4v_6v_8v_9 = tv_9t^{-1}v_4v_5v_8 =
  tv_{10}t^{-1}v_3v_4v_5v_6v_8v_{10} = 1, \displaybreak[0]\\
& gv_1g^{-1}v_1v_3 = gv_2g^{-1}v_3v_7v_8v_9 = gv_3g^{-1}v_3 =
  gv_4g^{-1}v_3v_4v_8v_9 = 1, \displaybreak[0]\\
& gv_5g^{-1}v_3v_5v_8 = gv_6g^{-1}v_3v_6v_9 = gv_7g^{-1}v_2v_3v_8v_9 =
  gv_8g^{-1}v_8 = 1, \displaybreak[0]\\
& gv_9g^{-1}v_9 = gv_{10}g^{-1}v_8v_{10} = hv_1h^{-1}v_1v_9 =
  hv_2h^{-1}v_2v_6 = hv_3h^{-1}v_3v_6v_9 = 1, \displaybreak[0]\\
& hv_4h^{-1}v_4v_9 = hv_5h^{-1}v_6v_8 = hv_6h^{-1}v_6 =
  hv_7h^{-1}v_9v_{10} = hv_8h^{-1}v_5v_6 = 1, \displaybreak[0]\\
& hv_9h^{-1}v_9 = hv_{10}h^{-1}v_7v_9 = iv_1i^{-1}v_5v_6v_{10} =
  iv_2i^{-1}v_2v_3v_4v_5 = 1, \displaybreak[0]\\
& iv_3i^{-1}v_3v_4v_5v_6 = iv_4i^{-1}v_3v_5 = iv_5i^{-1}v_3v_5v_6 =
  iv_6i^{-1}v_4v_5 = 1, \displaybreak[0]\\
& iv_7i^{-1}v_4v_5v_6v_7 = iv_8i^{-1}v_4v_6v_9 = iv_9i^{-1}v_3v_4v_8 =
  iv_{10}i^{-1}v_1v_3v_4v_6 = 1. 
\end{alignat*}
\item[\rm(f)] $E_2 = V_2\rtimes \M_{22} = \langle
a,b,c,d,t,g,h,i,v_1,v_2,v_3,v_4,v_5,v_6,v_8,v_8,v_9,v_{10}
\rangle$ has 
set ${\mathcal R}(E_2)$ of defining relations
consisting of ${\mathcal R}(\M_{22})$, ${\mathcal R_1}(V_2\rtimes
\M_{22})$ and the following set ${\mathcal R_2}(V_2\rtimes
\M_{22})$ of essential relations:
\begin{alignat*}{1}
& av_1a^{-1}v_1 =  av_2a^{-1}v_4 =  av_3a^{-1}v_1v_3v_5v_7v_8v_{10} =
  av_4a^{-1}v_2 = av_5a^{-1}v_8 =1, \displaybreak[0]\\
& av_6a^{-1}v_9 = av_7a^{-1}v_7 = av_8a^{-1}v_5 = av_9a^{-1}v_6 =
  av_{10}a^{-1}v_{10} = bv_1b^{-1}v_1 = 1, \displaybreak[0]\\
& bv_2b^{-1}v_6 = bv_3b^{-1}v_1v_3v_4v_7v_9v_{10} = bv_4b^{-1}v_9 =
  bv_5b^{-1}v_1v_2v_8v_9v_{10} = 1, \displaybreak[0]\\
& bv_6b^{-1}v_2 = bv_7b^{-1}v_7 = bv_8b^{-1}v_1v_4v_5v_6v_{10} =
  bv_9b^{-1}v_4 = bv_{10}b^{-1}v_{10} = 1, \displaybreak[0]\\
& cv_1c^{-1}v_1 = cv_2c^{-1}v_5 = cv_3c^{-1}v_2v_3v_6v_7v_8v_9v_{10} =
  cv_4c^{-1}v_8 = cv_5c^{-1}v_2 = 1, \displaybreak[0]\\
& cv_6c^{-1}v_1v_2v_8v_9v_{10} = cv_7c^{-1}v_7 = cv_8c^{-1}v_4 =
  cv_9c^{-1}v_1v_4v_5v_6v_{10} = 1, \displaybreak[0]\\
& cv_{10}c^{-1}v_{10} = dv_1d^{-1}v_1 = dv_2d^{-1}v_3 = dv_3d^{-1}v_2
  = dv_4d^{-1}v_1v_3v_5v_7v_8v_{10} = 1, \displaybreak[0]\\
& dv_5d^{-1}v_2v_3v_6v_7v_8v_9v_{10} = dv_6d^{-1}v_1v_3v_4v_7v_9v_{10}
  = dv_7d^{-1}v_7 = 1, \displaybreak[0]\\
& dv_8d^{-1}v_1v_3v_4v_6v_8v_9 = dv_9d^{-1}v_2v_3v_5v_6v_8 =
  dv_{10}d^{-1}v_{10} = tv_1t^{-1}v_1 = 1, \displaybreak[0]\\
& tv_2t^{-1}v_2 = tv_3t^{-1}v_1v_2v_3v_4v_5v_6v_7 =
  tv_4t^{-1}v_1v_3v_4v_6v_8v_9 = 1, \displaybreak[0]\\
& tv_5t^{-1}v_2v_3v_4v_5v_9v_{10} = tv_6t^{-1}v_2v_3v_5v_6v_8 =
  tv_7t^{-1}v_7 = tv_8t^{-1}v_9 = 1, \displaybreak[0]\\
& tv_9t^{-1}v_1v_2v_8v_9v_{10} = tv_{10}t^{-1}v_{10} = gv_1g^{-1}v_1 =
  gv_2g^{-1}v_7 = 1, \displaybreak[0]\\
& gv_3g^{-1}v_1v_2v_3v_4v_5v_6v_7 = gv_4g^{-1}v_4 = gv_5g^{-1}v_5 =
  gv_6g^{-1}v_6 = gv_7g^{-1}v_2 = 1, \displaybreak[0]\\
& gv_8g^{-1}v_2v_4v_5v_7v_8v_{10} = gv_9g^{-1}v_2v_4v_6v_7v_9 =
  gv_{10}g^{-1}v_{10} = hv_1h^{-1}v_1 = 1, \displaybreak[0]\\
& hv_2h^{-1}v_2 = hv_3h^{-1}v_3 =  hv_4h^{-1}v_4 = hv_5h^{-1}v_8 =
  hv_6h^{-1}v_2v_3v_5v_6v_8 = 1, \displaybreak[0]\\
& hv_7h^{-1}v_{10} = hv_8h^{-1}v_5 = hv_9h^{-1}v_1v_3v_4v_7v_9v_{10} =
  hv_{10}h^{-1}v_7 = 1, \displaybreak[0]\\
& iv_1i^{-1}v_{10} = iv_2i^{-1}v_2 = iv_3i^{-1}v_2v_3v_4v_5v_9v_{10} =
  iv_4i^{-1}v_2v_3v_6v_7v_8v_9v_{10} = 1, \displaybreak[0]\\
& iv_5i^{-1}v_1v_2v_3v_4v_5v_6v_7 =  iv_6i^{-1}v_1v_3v_5v_7v_8v_{10} =
  iv_7i^{-1}v_7 = 1, \displaybreak[0]\\
& iv_8i^{-1}v_9 =  iv_9i^{-1}v_8 =  iv_{10}i^{-1}v_1 = 1.
\end{alignat*}
\end{enumerate}
\end{lemma}

\begin{proof}
The 
two irreducible 
$F\!\M_{22}$-modules $V_i$, $i =1,2$, occur as
composition factors with multiplicity $1$ in the permutation
module $(1_{\M_{21}})^{\M_{22}}$ of degree $22$ where $\M_{21} =
\langle a,b,c,d,t,g,h,i \rangle$. They are dual to each other.
Using the faithful permutation representation of $\M_{22}$ stated
in (a) and the Meataxe 
Algorithm implemented in \textsc{Magma} 
we obtain the generating matrices of $\M_{22}$ 
given in (b) that define $V_1$. 
Their dual matrices define $V_2$, yielding (c). 

(d) The cohomological dimensions $d_i = 
\dim_F[H^2(\M_{22},V_i)]$,
$i =1,2$, have been calculated by means of \textsc{Magma} using Holt's
Algorithm 7.4.5 of \cite{michler}. Its hypothesis is given by the
presentation of $\M_{22}$ in Definition 8.2.1 of \cite{michler}
and all the data stated in (a), (b) and (c). It follows that $d_1
= 0 = d_2$.

(e) Since $d_1 = 0$ there exists only the split extension $E_1$ of
$\M_{22}$ by $V_1$. The presentation of $E_1$ has been obtained
automatically by application of Lemma \ref{l. presentation-split}
and \textsc{Magma}.

(f) Since $d_2 = 0$ the split extension $E_2$ of $\M_{22}$ by
$V_2$ is well defined. Its presentation has been obtained by
application of Lemma \ref{l. presentation-split}.
\end{proof}

\begin{lemma}\label{l. Aut(M22)-extensions}
Let $\M_{22} = \langle a,b,c,d,t,g,h,i \rangle$ be the finitely
presented group of Definition 8.2.1 of \cite{michler}. Let $A_{22}
= 
{\rm Aut}(\M_{22})$. Then the following statements hold:
\begin{enumerate}
\item[\rm(a)] $A_{22}$ has a faithful permutation representation
of degree $44$ and is isomorphic to the subgroup $\langle Pp,Pq
\rangle$ of the symmetric group $S_{44}$ generated by the
following permutations:
\begin{eqnarray*}
Pp &= &(1,2,7,4)(3,21,44,25)(5,24,22,26)(6,8,39,36)(9,35)(10,37,12,13)\\
&&\ \cdot(11,20,14,17)(18,19)(23,42,41,32)(27,29)(28,33)(30,31,43,34),\\
Pq &=& (1,3,16,44,13,5)(2,15)(4,28,36,22,17,29)(6,24,12,43,39,38)(7,30)\\
&&\ \cdot(8,33,11,25,19,41)(9,31,18,21,10,42)(14,27,20,23,35,32)(26,37)(34,40).
\end{eqnarray*}

\item[\rm(b)] $A_{22} = \langle p,q \rangle$ has the following set
${\mathcal R}(A_{22})$ of defining relations:
\begin{eqnarray*}
&& p^4 = q^6 = 1,\\
&& q^{-1}p^{-2}q^{-1}p^{-1}q^{-1}p^{-2}q^{-1}p^{-1}q^{-1}p^2q^{-1}p^{-1} = 1,\\ 
&& pq^{-3}p^{-1}q^{-1}p^{-2}q^{-1}pqp^{-1}qpq^{-1} = (p^{-1}q^{-2}p^{-1}qpq^{-1}p^{-1} q^{-1})^2 = 1,\\
&& q^{-1}p^{-2}q^{-1}pq^{-1}p^{-2}q^{-1}pq^{-1}p^2 q^{-1}p = (q^{-1}p^{-1}q^{-1})^5 = 1,\\
&& qpq^2pqp^{-1} q^{-1}p^{-2}q^{-2}p^2q^{-1}p^{-1}q^{-2}p^{-1} =1. 
\end{eqnarray*}

\item[\rm(c)] $A_{22} = \langle p, q \rangle$ has (up to
isomorphism) two irreducible modules $V_3$ and $V_4$ of degree
$10$ over $F = 
\GF(2)$. $V_3$ is described by the following
matrices: {\renewcommand{\arraystretch}{0.5}
\scriptsize
$$
Mp = \left( \begin{array}{*{10}{c@{\,}}c}
 1 & 0 & 0 & 0 & 0 & 0 & 0 & 0 & 0 & 0\\
 0 & 0 & 0 & 0 & 0 & 0 & 1 & 0 & 0 & 1\\
 1 & 1 & 0 & 0 & 1 & 0 & 1 & 0 & 0 & 1\\
 0 & 1 & 1 & 1 & 0 & 0 & 1 & 0 & 0 & 0\\
 1 & 1 & 0 & 0 & 0 & 0 & 1 & 0 & 0 & 0\\
 0 & 1 & 0 & 0 & 0 & 0 & 0 & 0 & 1 & 1\\
 1 & 0 & 0 & 1 & 0 & 1 & 1 & 0 & 0 & 0\\
 1 & 0 & 0 & 1 & 0 & 0 & 0 & 0 & 0 & 1\\
 1 & 1 & 0 & 1 & 0 & 0 & 0 & 0 & 0 & 1\\
 1 & 1 & 0 & 1 & 0 & 0 & 1 & 1 & 0 & 1
\end{array} \right)\quad \text{and}~\quad
Mq = \left( \begin{array}{*{10}{c@{\,}}c}
 0 & 0 & 1 & 0 & 1 & 0 & 0 & 1 & 0 & 0\\
 0 & 0 & 1 & 0 & 0 & 0 & 0 & 0 & 0 & 0\\
 0 & 0 & 0 & 0 & 1 & 0 & 0 & 0 & 0 & 0\\
 1 & 0 & 1 & 0 & 0 & 1 & 0 & 0 & 0 & 0\\
 0 & 0 & 0 & 0 & 1 & 0 & 0 & 0 & 1 & 0\\
 0 & 0 & 1 & 0 & 0 & 0 & 1 & 0 & 0 & 0\\
 0 & 0 & 0 & 0 & 1 & 1 & 0 & 0 & 0 & 0\\
 0 & 0 & 0 & 1 & 1 & 1 & 0 & 0 & 0 & 0\\
 0 & 0 & 0 & 0 & 1 & 1 & 0 & 0 & 0 & 1\\
 0 & 1 & 1 & 0 & 1 & 1 & 0 & 0 & 0 & 0
\end{array} \right).
$$}
\item[\rm(d)] $V_4$ is described by the matrices $Mp_1 = Mp^{*} =
[Mp^{-1}]^T$, $Mq_1 = Mq^{*} = [Mq^{-1}]^T$ in $\GL_{10}(2)$.

\item[\rm(e)] 
$\dim_F[H^2(A_{22},V_3)] = 1$ and
$\dim_F[H^2(A_{22},V_4)] = 0$.

\item[\rm(f)] $E_3 = V_3\rtimes A_{22} = \langle
p,q,v_1,v_2,v_3,v_4,v_5,v_6,v_8,v_8,v_9,v_{10} \rangle $ has 
set ${\mathcal R}(E_3)$ of defining relations consisting of ${\mathcal
R}(A_{22})$, ${\mathcal R_1}(V_3\rtimes A_{22})$ and the following
set ${\mathcal R_2}(V_3\rtimes A_{22})$ of essential relations:
\begin{eqnarray*}
&& pv_1p^{-1}v_1 = pv_2p^{-1}v_8v_9 = pv_3p^{-1}v_1v_2v_4v_9 = pv_4p^{-1}v_2v_5v_9 = 1,\\
&& pv_5p^{-1}v_1v_2v_3v_8v_9 = pv_6p^{-1}v_2v_7v_8 = pv_7p^{-1}v_1v_5v_8v_9 = 1,\\
&& pv_8p^{-1}v_1v_5v_8v_{10} = pv_9p^{-1}v_1v_2v_5v_6 = pv_{10}p^{-1}v_1v_2v_5v_8v_9 = 1,\\
&& qv_1q^{-1}v_2v_3v_4v_7 = qv_2q^{-1}v_2v_7v_{10} = qv_3q^{-1}v_2 = 1,\\
&&qv_4q^{-1}v_7v_8 = qv_5q^{-1}v_3 = qv_6q^{-1}v_3v_7 = qv_7q^{-1}v_2v_6 = 1,\\
&&qv_8q^{-1}v_1v_2v_3 = qv_9q^{-1}v_3v_5 = qv_{10}q^{-1}v_7v_9 = 1.
\end{eqnarray*}

\item[\rm(g)] $E_4 = V_4\rtimes A_{22} = \langle
p_1,q_1,v_1,v_2,v_3,v_4,v_5,v_6,v_8,v_8,v_9,v_{10} \rangle $ has 
set ${\mathcal R}(E_4)$ of defining relations consisting of
${\mathcal R}(A_{22})$, ${\mathcal R_1}(V_4\rtimes A_{22})$ and
the following set ${\mathcal R_2}(V_4\rtimes A_{22})$ of essential
relations:
\begin{eqnarray*}
&&p_1v_1p_1^{-1}v_1v_3v_5v_7v_8v_9v_{10} = p_1v_2p_1^{-1}v_3v_4v_5v_6v_9v_{10}= 1,\\
&&p_1v_3p_1^{-1}v_4 = p_1v_4p_1^{-1}v_4v_7v_8v_9v_{10}= 1,\\
&&p_1v_5p_1^{-1}v_3 = p_1v_6p_1^{-1}v_7 = p_1v_7p_1^{-1}v_2v_3v_4v_5v_7v_{10}= 1,\\
&&p_1v_8p_1^{-1}v_{10} = p_1v_9p_1^{-1}v_6 = p_1v_{10}p_1^{-1}v_2v_3v_6v_8v_9v_{10}= 1,\\
&&q_1v_1q_1^{-1}v_4 = q_1v_2q_1^{-1}v_{10} = q_1v_3q_1^{-1}v_1v_2v_4v_6v_{10}= 1,\\
&&q_1v_4q_1^{-1}v_8 = q_1v_5q_1^{-1}v_1v_3v_5v_7v_8v_9v_{10} = q_1v_6q_1^{-1}v_4v_7v_8v_9v_{10}= 1,\\
&&q_1v_7q_1^{-1}v_6 = q_1v_8q_1^{-1}v_1 = q_1v_9q_1^{-1}v_5 = q_1v_{10}q_1^{-1}v_9= 1.
\end{eqnarray*}

\item[\rm(h)] The 
non-split extension $E_5 = \langle
p_2,q_2,v_1,v_2,v_3,v_4,v_5,v_6,v_7,v_8,v_9,v_{10} \rangle $ of
$A_{22}$ by $V_3$ has a set ${\mathcal R}(E_5)$ of defining
relations consisting of ${\mathcal R_1}(V_3\rtimes A_{22})$ and
the following relations:
\begin{alignat*}{1}
&p_2^8 = q_2^{12} = 1,\displaybreak[0]\\
&(p_2, v_1^{-1}) = p_2^{-1}v_2p_2v_7^{-1}v_{10}^{-1} =
  p_2^{-1}v_3p_2v_1^{-1}v_2^{-1}v_5^{-1}v_7^{-1}v_{10}^{-1} = 1, \displaybreak[0]\\
&p_2^{-1}v_4p_2v_2^{-1}v_3^{-1}v_4^{-1}v_7^{-1} =
  p_2^{-1}v_5p_2v_1^{-1}v_2^{-1}v_7^{-1} = 1, \displaybreak[0]\\
&p_2^{-1}v_6p_2v_2^{-1}v_9^{-1}v_{10}^{-1} =
  p_2^{-1}v_7p_2v_1^{-1}v_4^{-1}v_6^{-1}v_7^{-1} = 1, \displaybreak[0]\\
&p_2^{-1}v_8p_2v_1^{-1}v_4^{-1}v_{10}^{-1} = p_2^{-1}  v_9  p_2
  v_1^{-1}  v_2^{-1}  v_4^{-1}  v_{10}^{-1} = 1, \displaybreak[0]\\
&p_2^{-1}v_{10}p_2v_1^{-1}v_2^{-1}v_4^{-1}v_7^{-1}v_8^{-1}v_{10}^{-1}
  = 1, \displaybreak[0]\\
&q_2^{-1}v_1q_2v_3^{-1}v_5^{-1}v_8^{-1} = q_2^{-1}v_2q_2v_3^{-1} =
  q_2^{-1}v_3q_2v_5^{-1} = q_2^{-1}v_4q_2v_1^{-1}v_3^{-1}v_6^{-1} = 1,
  \displaybreak[0]\\
&q_2^{-1}v_5q_2v_5^{-1}v_9^{-1} = q_2^{-1}v_6q_2v_3^{-1}v_7^{-1} =
  q_2^{-1}v_7q_2v_5^{-1}v_6^{-1} = 1, \displaybreak[0]\\
&q_2^{-1}v_8q_2v_4^{-1}v_5^{-1}v_6^{-1} =
  q_2^{-1}v_9q_2v_5^{-1}v_6^{-1}v_{10}^{-1} =
  q_2^{-1}v_{10}q_2v_2^{-1}v_3^{-1}v_5^{-1}v_6^{-1} = 1, \displaybreak[0]\\
&p_2^4v_3^{-1}v_4^{-1}v_5^{-1}v_6^{-1}v_7^{-1}v_8^{-1}v_9^{-1}v_{10}^{-1}
  = q_2^6v_1^{-1}v_4^{-1}v_8^{-1} = 1, \displaybreak[0]\\
&q_2^{-1}p_2^{-2}q_2^{-1}p_2^{-1}q_2^{-1}p_2^{-2}q_2^{-1}p_2^{-1}q_2^{-1}p_2^2q_2^{-1}
  p_2^{-1}v_1^{-1}v_2^{-1}v_4^{-1} = 1, \displaybreak[0]\\
&p_2q_2^{-3}p_2^{-1}q_2^{-1}p_2^{-2}q_2^{-1}p_2q_2p_2^{-1}q_2p_2q_2^{-1}v_1^{-1}v_2^{-1}
  v_5^{-1}v_6^{-1}v_8^{-1}v_{10}^{-1} = 1, \displaybreak[0]\\
&q_2^{-1}p_2^{-2}q_2^{-1}p_2q_2^{-1}p_2^{-2}q_2^{-1}p_2q_2^{-1}p_2^2q_2^{-1}p_2v_5^{-1}
  v_6^{-1}v_7^{-1}v_9^{-1} = 1, \displaybreak[0]\\
&q_2^{-1}p_2^{-1}q_2^{-2}p_2^{-1}q_2^{-2}p_2^{-1}q_2^{-2}p_2^{-1}q_2^{-2}p_2^{-1}q_2^{-1}
  v_1^{-1}v_4^{-1}v_5^{-1}v_6^{-1}v_7^{-1}v_8^{-1}v_9^{-1} = 1, \displaybreak[0]\\
&p_2^{-1}q_2^{-2}p_2^{-1}q_2p_2q_2^{-1}p_2^{-1}q_2^{-1}p_2^{-1}q_2^{-2}p_2^{-1}
  q_2p_2q_2^{-1}p_2^{-1}q_2^{-1}\\
&\quad \cdot
  v_1^{-1}v_2^{-1}v_3^{-1}v_4^{-1}v_5^{-1}v_6^{-1}v_8^{-1}v_9^{-1} =
  1, \displaybreak[0]\\
&q_2p_2q_2^2p_2q_2p_2^{-1}q_2^{-1}p_2^{-2}q_2^{-2}p_2^2q_2^{-1}p_2^{-1}q_2^{-2}p_2^{-1}
  v_1^{-1}v_4^{-1}v_5^{-1}v_9^{-1} = 1.
\end{alignat*} 
\end{enumerate}
\end{lemma}

\begin{proof}
(a) Using Todd's permutation representation 
$P\!\M_{22}$ of $\M_{22}$
stated in Lemma 8.2.2 of \cite{michler} and the \textsc{Magma} command
$\verb"AutomorphismGroup(PM_{22})"$ we calculated the automorphism
group $A_{22} = 
{\rm Aut}(\M_{22})$ of $\M_{22}$. \textsc{Magma} provided 
eleven generators and a faithful permutation representation 
$P\!A_{22}$ of
degree $44$ by means of the command
$\verb"PermutationGroup(A_{22})"$. Using this permutation
representation we were able to show that 
$P\!A_{22}$ is even generated
by the 
two permutations $Pp,Pq \in S_{44}$ given in the statement.
Also it has been checked computationally that the derived subgroup
$A'_{22}$ is a simple group which is isomorphic to $\M_{22}$ and
that 
$|A_{22}:A'_{22}| = 2$.

(b) The given presentation of $A_{22} = \langle p, q \rangle$ in
the above generators has been obtained by \textsc{Magma} and the faithful
permutation representation of degree $44$.

(c) Decomposing the permutation module $PA_{22}$ of degree
$44$ into irreducible composition factors by means of the Meataxe
Algorithm we obtained the first irreducible representation $V_3$
of $A_{22}$ over $F = 
\GF(2)$ as given in the statement.

(d) The second irreducible representation $V_4$ of $A_{22}$ is the
dual of $V_3$, which has been checked to be 
non-isomorphic to
$V_3$.

(e) The cohomological dimensions $d_i = 
\dim_F[H^2(A_{22},V_i)]$,
$i =3,4$, have been calculated by means of \textsc{Magma} using Holt's
Algorithm 7.4.5 of \cite{michler}, the presentation of $A_{22}$
given in (b), the faithful permutation representation $PA_{22}$
constructed in (a) and the 
two irreducible representations of
$A_{22}$ given in (c) and 
(d). As a result of the computation, we get $d_3 = 1$ and $d_4= 0$.

(f) and (g). The two presentations of the split extensions $E_3$
and $E_4$ have been obtained by means of Lemma \ref{l.
presentation-split}.

(h) The presentation of the 
non-split extension $E_5$ has been
obtained by means of Holt's Algorithm 7.4.5 of \cite{michler}
implemented in \textsc{Magma} \cite{holt}. This completes the proof.

\end{proof}

\begin{lemma}\label{l. classes} Keep the notation of Lemmas \ref{l.
M22-extensions} and \ref{l. Aut(M22)-extensions}. Then the
following statements hold:
\begin{enumerate}
\item[\rm(a)] Each of the 
four split extensions $E_i$, $i \in
\{1,2,3,4\}$ has a faithful permutation representation of degree
$1024$. Its stabilizer is the complement of $V_i$ in $E_i$, which
generates $E_i$ together with $v_1 \in V_i$ for all 4 groups.

\item[\rm(b)] The 
non-split extension $E_5 = \langle
p_2,q_2,v_1,v_2,v_3,v_4,v_5,v_6,v_7,v_8,v_9,v_{10} \rangle =
\langle p_2, q_2 \rangle$ has a faithful permutation
representation of degree $88$ with stabilizer 
$U = \langle
q_2^2, (p_2q_2^2)^2 \rangle$.

\item[\rm(c)] The split extension $E_1$ of $\M_{22}$ by $V_1$ has
$47$ conjugacy classes. 
The element $z_1 =(tv_1)^3$ represents the unique
conjugacy class of $2$-central involutions of $E_1$ and
$|C_{E_1}(z_1)| = 2^{17}\cdot 3^2\cdot 5$.

\item[\rm(d)] The split extension $E_2$ of $\M_{22}$ by $V_2$ has
$43$ conjugacy classes. 
The element $z_2 =(iv_1)^2$ 
represents the unique
conjugacy class of $2$-central involutions of $E_2$ and
$|C_{E_2}(z_2)| = 2^{17}\cdot 3\cdot 5$.

\item[\rm(e)] The split extension $E_3$ of $A_{22}$ by $V_3$ has
$79$ conjugacy classes. 
The el\nobreak{}ement $z_3 = (pq^2v_1)^5$ 
represents the unique
conjugacy class of $2$-central involutions of $E_3$ and
$|C_{E_3}(z_3)| = 2^{18}\cdot 3^2\cdot 5$.

\item[\rm(f)] The split extension $E_4$ of $A_{22}$ by $V_4$ has
$77$ conjugacy classes. 
The element $z_4 = (p_1^2q_1v_1)^{10}$ 
represents the
unique conjugacy class of $2$-central involutions of $E_4$ and
$|C_{E_4}(z_4)| = 2^{18}\cdot 3\cdot 5$.

\item[\rm(g)] The 
non-split extension $E_5$ of $A_{22}$ by $V_3$
has $79$ conjugacy classes. 
The element $z_5 = p_2^4$ 
represents the unique
conjugacy class of $2$-central involutions of $E_5$ and
$|C_{E_5}(z_5)| = 2^{18}\cdot 3^2\cdot 5$.
\end{enumerate}
\end{lemma}

\begin{proof}
(a) The presentations of the $4$ split extensions $E_i$ are given
in Lemmas~\ref{l. M22-extensions} and~\ref{l. Aut(M22)-extensions}. 
Taking their complements $\M_{22}$ or
$A_{22}$ of $V_i$ as stabilizers \textsc{Magma} provides the faithful
permutation representation $PE_i$ of $E_i$ of degree $1024$ for $i
= 1,2,3,4$. The final assertions have been proved by means of
these permutation representations and \textsc{Magma}.

(b) The presentation of the 
non-split extension $E_5$ is given in
Lemma \ref{l. M22-extensions}(h). Let $U = \langle q_2^2,
(p_2q_2^2)^2 \rangle$. Using the \textsc{Magma} command
$\verb"CosetAction(E_5,U)"$ we obtained a faithful permutation
representation $PE_5$ of $E_5$ with stabilizer $U$ and degree
$88$. It has been used to show that $E_5 = \langle p_2, q_2
\rangle$.

All remaining assertion were proved as follows. For each extension
group $E_i$, $i \in \{1,2,3,4,5\}$ we used its faithful
permutation representation $PE_i$, \textsc{Magma} and Kratzer's Algorithm
5.3.18 of \cite{michler} to determine a system of representatives
of all conjugacy classes of $E_i$ in terms of the given generators
of the presentation of~$E_i$. In view of the limited space they
are not stated in this article, except for the given words of the
$2$-central involutions $z_i$ of $E_i$ which in each case were
uniquely determined up to conjugation in $E_i$. This completes the
proof.
\end{proof}

\section{Construction of the $2$-central involution centralizers}

In this 
Section we apply Algorithm 2.5 of \cite{michler1} to the
extension groups $E_3$ and $E_2$ in order to construct $2$ groups
$H_3$ and $H_2$ which are isomorphic to the centralizers of a
$2$-central involution of the simple groups $\Co_2$ and
$\Fi_{22}$, respectively.

\begin{proposition}\label{prop. D(Co_2)} Keep the notation of Lemmas \ref{l.
M22-extensions} and \ref{l. Aut(M22)-extensions}. Let $E_3 = \nobreak
\langle p,q,v_1 \rangle$ be the split extension of $A_{22} =
{\rm Aut}(\M_{22})$ by its simple module $V_3$ of dimension $10$ over $F
= \GF(2)$. Then the following statements hold:
\begin{enumerate}
\item[\rm(a)] $z = (pq^2v_1)^5$ is a $2$-central involution of
$E_3 = \langle p, q, v_1 \rangle$ with centralizer
$D=\nobreak C_{E_3}(z) = \langle n_i, f_j \mid 1 \le i \le 12,\, 1 \le
  j \le 5 \rangle$, where 
\begin{alignat*}{3}
 n_1 &= (p^2qp^2q^2v_1)^6, &\quad  n_2 &= (p^2qpv_1qpv_1qv_1)^6, &\quad n_3 &= (v_1qp^3qpv_1qp)^4, \\
 n_4 &= (pqpq^4pqv_1q)^7, &\quad   n_5 &= (pqpv_1qp)^3, &\quad n_6 &= (pqp^2q^2v_1)^6, \\
 n_7 &= (pq^2pv_1qp)^6, &\quad     n_8 &= (qpqp^2qv_1)^6, &\quad n_9 &= (p^3q^3p^2q^2)^4,\\
 n_{10} &= (pqpq^2p^2qpq)^6, &\quad  n_{11} &= (pq^2p^3q^3p)^4, &\quad n_{12} &= (qp^2q^4pqp)^3, \\
 f_1 &= (p^3q^2v_1qpv_1q)^5, &\quad  f_2 &= (q^2pqv_1qpqpq)^7, &\quad f_3 &= (pqp^2q^3v_1qpq)^6, \\
 f_4 &= (p^3q^2pqpv_1q)^7, &\quad  f_5 &= (p^2q^3v_1qv_1q^2)^2. 
\end{alignat*}

\item[\rm(b)] $D$ has a unique extra-special normal subgroup $Q$
of order $512$. It is generated by the following involutions:
\begin{alignat*}{2}
 q_1 &= n_8n_{10}n_{11}n_8n_{12}, &\quad q_2 &= n_5n_6n_7n_{12}, \\
 q_3 &= n_5n_7n_9n_6, &\quad q_4 &= n_6n_8n_9n_{10}, \\
 q_5 &= n_5, &\quad q_6 &= n_5n_6n_7, \\
 q_7 &= n_5n_6n_7n_8, &\quad q_8 &= n_5n_7n_8. 
\end{alignat*}

\item[\rm(c)] $Q$ has a complement $W$ in $D$ of order $|W| =
2^9\cdot3^2\cdot5$. The Fitting subgroup $B = \langle f_j \mid 1
\le j \le 5\rangle$ 
of~$W$ is elementary abelian of order $2^5$ and has a
complement $L = \langle s_k \mid 1 \le k \le 3 \rangle \cong S_6$
in $W$, where
\[\qquad\qquad s_1 = n_2n_3n_2n_1n_3n_4,\quad  s_2 = n_1n_2n_3n_1n_2n_3,\quad\mbox{and}\quad  s_3 = (n_3n_1n_4)^2.\]

\item[\rm(d)] $D = \langle q_i, f_j, s_k \mid 1 \le i \le 8, 1 \le
j \le 5, 1\le k \le 3\rangle$ is a finitely presented group with
center 
$Z(D) = \langle z\rangle$ of order $2$, where $z=(q_4q_8)^2$, and
the following set ${\mathcal R}(D)$ of defining relations:
\begin{alignat*}{1}
& s_1^4 = s_2^5 = s_3^3 = 1, \displaybreak[0]\\
&(s_2s_1)^2 = s_1s_3s_2^{-1}s_1^{-1}s_3^{-1}s_2 = (s_2s_3)^3 =
  s_2^2s_3s_1s_2^{-1}s_1s_3 = 1, \displaybreak[0]\\
&s_2^{-1}s_1s_3^{-1}s_2^{-1}s_1s_3^{-1}s_2s_3 = 1, \displaybreak[0]\\
&f_1^2  = f_2^2  = f_3^2  = f_4^2  = f_5^2  = 1, \displaybreak[0]\\
&f_1^{s_1}(f_1)^{-1} = f_1^{s_2}(f_1)^{-1} = f_1^{s_3}(f_1)^{-1} =
  1,\displaybreak[0]\\ 
&f_2^{s_1}(f_1f_4f_5)^{-1} = f_2^{s_2}(f_1f_2f_3f_5)^{-1} =
  f_2^{s_3}(f_3f_5)^{-1} = 1,\displaybreak[0]\\ 
&f_3^{s_1}(f_3 f_5)^{-1} = f_3^{s_2}(f_2f_3)^{-1} =
  f_3^{s_3}(f_2f_5)^{-1} = 1,\displaybreak[0]\\ 
&f_4^{s_1}(f_1f_3f_4f_5)^{-1} = f_4^{s_2}(f_1f_3f_4)^{-1} =
  f_4^{s_3}(f_1f_3f_4)^{-1} = 1,\displaybreak[0]\\ 
&f_5^{s_1}(f_2f_3f_4)^{-1} = f_5^{s_2}(f_2f_3f_4)^{-1} =
  f_5^{s_3}(f_1f_2)^{-1} = 1,\displaybreak[0]\\ 
&[f_1, f_2] = [f_1, f_3] = [f_1, f_4] = [f_1, f_5] = [f_2, f_3] =
  [f_2, f_4] = [f_2, f_5] = 1,\displaybreak[0]\\ 
&[f_3, f_4]  = [f_3, f_5]  = [f_4, f_5]  = 1,\displaybreak[0]\\
&q_1^2 = q_2^2 = q_3^2 = q_4^2 = q_5^2 = q_6^2 = q_7^2 = q_8^2 =
  1,\displaybreak[0]\\ 
& [q_1, q_2] = [q_1, q_3] = [q_1, q_4] = [q_1, q_6] = [q_1, q_8] =
  1,\quad [q_1, q_5] = [q_1, q_7] = (q_4 q_8)^2,\displaybreak[0]\\ 
& [q_2, q_3] = [q_2, q_4] = [q_2, q_5] = [q_2, q_6] = [q_2, q_7] = 1,
  \quad [q_2, q_8] = (q_4 q_8)^2,\displaybreak[0]\\ 
&[q_3, q_4] = [q_3, q_6] = [q_3, q_7] = 1, \quad [q_3, q_5] = [q_3,
  q_8] = (q_4 q_8)^2, \displaybreak[0]\\
&[q_4, q_5] = [q_4, q_6] = [q_4, q_7] = [q_4, q_8] = (q_4 q_8)^2,\displaybreak[0]\\
&[q_5, q_6] = [q_5, q_7] = [q_5, q_8] = [q_6, q_7] = [q_6, q_8] =
  [q_7, q_8] = 1, \displaybreak[0]\\
&q_1^{s_1}(q_3q_4q_6q_7q_8 )^{-1} = (q_4q_8)^2,\quad
  q_1^{s_2}(q_1q_2q_3q_8)^{-1} = q_1^{s_3}(q_1q_3q_5q_7q_8)^{-1} = 1, \displaybreak[0]\\
&q_1^{f_1}(q_1q_5q_7 )^{-1} = q_1^{f_2}(q_1)^{-1} =
  q_1^{f_3}(q_1q_5q_6q_7)^{-1} = q_1^{f_4}(q_1)^{-1} = 1, \displaybreak[0]\\
&q_1^{f_5}(q_1 q_6 q_8 )^{-1} = q_2^{s_1}(q_1q_4q_5)^{-1} = (q_4
  q_8)^2, \displaybreak[0]\\
& q_2^{s_2}(q_2q_4q_5q_6)^{-1} = q_2^{s_3}(q_1q_2q_4q_5 )^{-1} = 1, \displaybreak[0]\\
&q_2^{f_1}(q_2q_6)^{-1} = (q_4 q_8)^2,\quad q_2^{f_2}(q_2 )^{-1} =
  q_2^{f_3}(q_2q_5q_6q_7 )^{-1} = 1, \displaybreak[0]\\
&q_2^{f_4}(q_2q_5q_7)^{-1} = q_2^{f_5}(q_2q_6q_7 )^{-1} =
  (q_4q_8)^2,\quad q_3^{s_1}(q_3q_5q_7q_8 )^{-1} = (q_4q_8)^2, \displaybreak[0]\\
& q_3^{s_2}(q_2q_3q_6q_8 )^{-1} = (q_4 q_8)^2,\quad
  q_3^{s_3}(q_1q_5q_6q_7q_8 )^{-1} = 1,\quad q_3^{f_1}(q_3q_7 )^{-1} =
  (q_4q_8)^2, \displaybreak[0]\\
&q_3^{f_2}(q_3q_6)^{-1} = q_3^{f_3}(q_3q_5q_8)^{-1} = 1,\quad
  q_3^{f_4}(q_3q_5q_6q_7q_8 )^{-1} = q_3^{f_5}(q_3 q_7 )^{-1} = (q_4
  q_8)^2, \displaybreak[0]\\
&q_4^{s_1}(q_2q_4q_6q_7)^{-1} = q_4^{s_2}(q_1q_3q_5q_7q_8)^{-1} =
  1,\quad q_4^{s_3}(q_1q_2q_3q_5q_8)^{-1} = (q_4q_8)^2, \displaybreak[0]\\
&q_4^{f_1}(q_4q_6q_8)^{-1} = 1, \quad q_4^{f_2}(q_4q_5q_7)^{-1} =
  q_4^{f_5}(q_4q_5q_7)^{-1} = (q_4q_8)^2, \displaybreak[0]\\
& q_4^{f_3}(q_4 q_5 q_8 )^{-1} = q_4^{f_4}  ( q_4 )^{-1} = 1, \displaybreak[0]\\
& q_5^{s_1}(q_6q_8)^{-1} = q_5^{s_2}(q_5q_7)^{-1} = 1, \quad
  q_5^{s_3}(q_7)^{-1} = (q_4 q_8)^2,\quad q_5^{f_1}(q_5)^{-1} = 1, \displaybreak[0]\\
& q_5^{f_2}(q_5 )^{-1} = q_5^{f_3}(q_5)^{-1} = q_5^{f_4}(q_5)^{-1} =
  q_5^{f_5}(q_5)^{-1} = 1, \displaybreak[0]\\
&q_6^{s_1}(q_5q_6q_7q_8)^{-1} = q_6^{s_2}(q_8 )^{-1} =
  q_6^{s_3}(q_5q_7q_8)^{-1} = (q_4q_8)^2, \displaybreak[0]\\
&q_6^{f_1}(q_6)^{-1} = 1,\quad q_6^{f_2}(q_6)^{-1} =
  q_6^{f_3}(q_6)^{-1} = q_6^{f_4}(q_6)^{-1}  = q_6^{f_5}(q_6)^{-1} =
  1, \displaybreak[0]\\
&q_7^{s_1}(q_7)^{-1} = 1, \quad q_7^{s_2}(q_6q_7)^{-1} = q_7^{s_3}(q_5
  q_7 )^{-1} = (q_4 q_8)^2, \displaybreak[0]\\
&q_7^{f_1}(q_7)^{-1} = q_7^{f_2}(q_7)^{-1} = q_7^{f_3}(q_7)^{-1} =
  q_7^{f_4}(q_7)^{-1} = q_7^{f_5}(q_7)^{-1} = 1, \displaybreak[0]\\
&q_8^{s_1}(q_5q_6q_7)^{-1} = q_8^{s_2}(q_5q_8)^{-1} = (q_4
  q_8)^2,\quad q_8^{s_3}(q_6 q_7 q_8 )^{-1} = 1, \displaybreak[0]\\
&q_8^{f_1}(q_8)^{-1} =  q_8^{f_2}(q_8)^{-1} =  q_8^{f_3}(q_8)^{-1} =
  q_8^{f_4}(q_8)^{-1} = q_8^{f_5}(q_8 )^{-1} = 1, \displaybreak[0]\\
&[(q_4 q_8)^2, s_k] = 1 \quad \mbox{for} \quad k = 1, 2, 3, \displaybreak[0]\\
&[(q_4 q_8)^2, f_j] = 1 \quad \mbox{for} \quad 1 \le j \le 5, \displaybreak[0]\\
&[(q_4 q_8)^2, q_i] = 1 \quad \mbox{for} \quad 1 \le i \le 8.
\end{alignat*}

\item[\rm(e)]  $D_1 = D/Z(D)$ has a faithful permutation
representation of degree $256$ with stabilizer $W$. Furthermore,
$V = Q/Z(D)$ is the unique elementary abelian normal subgroup $V =
Q/Z(D)$ of order $256$ in $D_1$, and $D_1$ is a semidirect product
of $W$ by $V$.

\item[\rm(f)] There is a basis $\mathcal B = \{v_i = q_iZ(D)
\,|\, 1 \le i \le 8\}$ of $V$ such that the 
three generators $s_k$ of $L$
are represented by the following matrices with respect to
$\mathcal B$:
{\renewcommand{\arraystretch}{0.5}
\scriptsize
\[ Ms_1 = \left( \begin{array}{*{8}{c@{\,}}c}
  0 &   1 &   0 &   0 &   0 &   0 &   0 &   0 \\
  0 &   0 &   0 &   1 &   0 &   0 &   0 &   0 \\
  1 &   0 &   1 &   0 &   0 &   0 &   0 &   0 \\
  1 &   1 &   0 &   1 &   0 &   0 &   0 &   0 \\
  0 &   1 &   1 &   0 &   0 &   1 &   0 &   1 \\
  1 &   0 &   0 &   1 &   1 &   1 &   0 &   1 \\
  1 &   0 &   1 &   1 &   0 &   1 &   1 &   1 \\
  1 &   0 &   1 &   0 &   1 &   1 &   0 &   0 \\
\end{array} \right),
\quad
Ms_2 = \left( \begin{array}{*{8}{c@{\,}}c}
 1 &   0 &   0 &   1 &   0 &   0 &   0 &   0 \\
 1 &   1 &   1 &   0 &   0 &   0 &   0 &   0 \\
 1 &   0 &   1 &   1 &   0 &   0 &   0 &   0 \\
 0 &   1 &   0 &   0 &   0 &   0 &   0 &   0 \\
 0 &   1 &   0 &   1 &   1 &   0 &   0 &   1 \\
 0 &   1 &   1 &   0 &   0 &   0 &   1 &   0 \\
 0 &   0 &   0 &   1 &   1 &   0 &   1 &   0 \\
 1 &   0 &   1 &   1 &   0 &   1 &   0 &   1 \\
\end{array} \right),\quad\text{and}~\quad
Ms_3 = \left( \begin{array}{*{8}{c@{\,}}c}
 1 &   1 &   1 &   1 &   0 &   0 &   0 &   0 \\
 0 &   1 &   0 &   1 &   0 &   0 &   0 &   0 \\
 1 &   0 &   0 &   1 &   0 &   0 &   0 &   0 \\
 0 &   1 &   0 &   0 &   0 &   0 &   0 &   0 \\
 1 &   1 &   1 &   1 &   0 &   1 &   1 &   0 \\
 0 &   0 &   1 &   0 &   0 &   0 &   0 &   1 \\
 1 &   0 &   1 &   0 &   1 &   1 &   1 &   1 \\
 1 &   0 &   1 &   1 &   0 &   1 &   0 &   1 \\
\end{array} \right).
\]}

\item[\rm(g)] $W = \langle u, r, s_1, s_2, s_3 \rangle$ where $u =
f_1$ is the central involution of $W$ and $r = f_2f_4$.

\item[\rm(h)] The subgroup $S = \langle f_j, a_s, w  \mid 1\le j
\le 5, 1 \le s \le 6 \rangle $ is a Sylow $2$-subgroup of $W$
which has a unique maximal elementary abelian normal $2$-subgroup
$A = \langle a_1,a_2,a_3,a_4,a_5 = u,a_6 \rangle$ of order $64$
where
\begin{alignat*}{3}
w &= (s_1s_2^2s_3s_2)^3, &\quad  a_1 &= s_2^2f_3s_1s_3s_2, &\quad a_2 &= (f_5s_2s_1^3)^3,\\
a_3 &= (s_1f_3)^2,&\quad a_4 &= (s_1)^2, &\quad a_6 &= (f_2s_1)^4.
\end{alignat*}

\item[\rm(i)] Let 
$M\!A$ be the subgroup of $\GL_8(2)$ generated by
the matrices $Ma_1$, $Ma_2$, $Ma_3$, $Ma_4$, $Mu$, $Ma_6$ of the
generators of $A$ with respect to $\mathcal B$. Then 
$N\!A = 
N_{\GL_8(2)}(M\!A)$ has a cyclic Sylow $7$-subgroup generated by the
matrix:
{\renewcommand{\arraystretch}{0.5}
\scriptsize
$$ 
Ms = \left( \begin{array}{*{8}{c@{\,}}c}
 1 &   1 &   0 &   0 &   0 &   0 &   0 &   0 \\
 1 &   1 &   0 &   0 &   1 &   0 &   0 &   1 \\
 1 &   1 &   0 &   0 &   0 &   1 &   0 &   0 \\
 0 &   1 &   0 &   1 &   0 &   0 &   0 &   0 \\
 0 &   0 &   1 &   1 &   0 &   0 &   0 &   0 \\
 0 &   1 &   0 &   0 &   1 &   1 &   0 &   0 \\
 0 &   1 &   0 &   1 &   0 &   1 &   1 &   1 \\
 1 &   1 &   1 &   0 &   0 &   0 &   0 &   0 \\
\end{array} \right).
$$
}
\item[\rm(j)] Let $MW = \langle Mr,Mu,Ms_1,Ms_2,Ms_3 \rangle$. Let
$M\!K$ be the subgroup of $\GL_8(2)$ generated by $MW$ and the
matrix $Ms$. Then $M\!K$ is a simple group of order $|M\!K| =
2^9\cdot3^4\cdot5\cdot7$ which is isomorphic to $\Sp_6(2)$.

\item[\rm(k)] $M\!K$ is an irreducible subgroup of $\GL_8(2)$
generated by the five matrices $m_1 = Ms$, $m_2 = Mr$, $m_3 = Ms_1$,
$m_4 = Ms_2$ and $m_5 = Ms_3$ of respective orders $7$, $2$, $4$,
$5$, and~$3$. With respect to this set of generators 
$M\!K$ has the
following set ${\mathcal R}(K)$ of defining relations:
\begin{eqnarray*}
&&m_1^7 = m_2^2 = m_3^4 = m_4^5 = m_5^3 = 1,\\
&&m_2  m_4^{-1}  m_3^{-1}  m_2  m_3  m_4 = (m_3^{-1}  m_4^{-1})^2 = 1,\\
&&m_3m_5m_4^{-1}m_3^{-1}m_5^{-1}m_4 = (m_5^{-1}m_4^{-1})^3 = 1,\\
&&(m_1^{-1}m_2)^4 = (m_2m_3^{-1})^4 = 1,\quad (m_2m_3m_2m_3^{-1})^2 = 1,\\
&&m_5^{-1}  m_3^{-1}  m_4  m_3^{-1}  m_5^{-1}  m_4^{-2} = 1,\\
&&m_2  m_3  m_5^{-1}  m_3^{-1}  m_2  m_3  m_5  m_3^{-1} = 1,\\
&&m_4  m_1^{-1}  m_4  m_3  m_1  m_4^{-1}  m_3  m_4 = 1,\\
&&m_4^{-1}  m_3  m_5^{-1}  m_4^{-1}  m_3  m_5^{-1}  m_4  m_5 = 1,\\
&&[m_5, m_1^{-1}, m_5] = 1, \quad (m_1^{-2}  m_2)^3 = 1,\\
&&m_5  m_3  m_1  m_2  m_1^{-1}  m_5^{-1}  m_2  m_5^{-1}  m_4^{-1} = 1,\\
&&(m_3^{-1}  m_1^{-1}  m_2)^3 = 1,\quad (m_3  m_1^{-1}  m_2)^3 = 1,\\
&&m_1  m_3^{-1}  m_2  m_3  m_1^{-1}  m_4^{-2}  m_3  m_5 = 1,\\
&&m_1^{-1}  m_3  m_2  m_1^3  m_3  m_1^2  m_3^{-1} = 1,\\
&&m_2  m_1  m_3  m_1^3  m_2  m_1  m_2  m_3^{-1} = 1.
\end{eqnarray*}

\item[\rm(l)] Let $K$ be the finitely presented group constructed
in 
{\rm (k).} Then $V$ is an irreducible $8$-dimensional representation
of $K$ over 
$F = \GF(2)$. Furthermore, $K$ has a faithful
permutation representation $PK$ of degree $63$.

\item[\rm(m)] Let $H_1 = 
\langle m_i, v_j \,|\,1 \le i \le 5, 1 \le j
\le 8 \rangle$ be the split extension of $K$ by~$V$. Then $H_1$
has a set ${\mathcal R}(H_1)$ of defining relations consisting of
${\mathcal R}(K)$, ${\mathcal R_1}(V\rtimes K)$ and the following
set ${\mathcal R_2}(V\rtimes K)$ of essential relations:
\begin{alignat*}{1}
&m_1v_1m_1^{-1}v_7v_8 = m_1v_2m_1^{-1}v_1 = m_1 v_3m_1^{-1}v_2v_8 =
  m_1v_4m_1^{-1}v_3 =1, \displaybreak[0]\\
&m_1 v_5m_1^{-1}v_4v_8 = m_1 v_6m_1^{-1}v_5v_8 =m_1v_7m_1^{-1}v_6
  =m_1v_8m_1^{-1}v_8 =1, \displaybreak[0]\\
&m_2v_1m_2^{-1}v_4v_5v_6=m_2v_2m_2^{-1}v_1v_2v_3v_4v_7v_8 =
  m_2v_3m_2^{-1}v_5v_6v_7v_8 =1, \displaybreak[0]\\
&m_2v_4m_2^{-1}v_3v_4v_5v_6v_7v_8 =m_2v_5m_2^{-1}v_1v_4v_6 =
  m_2v_6m_2^{-1}v_3v_5v_7v_8 = 1, \displaybreak[0]\\
&m_2v_7m_2^{-1}v_1v_4v_5v_6v_7 = m_2v_8m_2^{-1}v_8=1, \displaybreak[0]\\
&m_3v_1m_3^{-1}v_1v_2v_7v_8 =m_3 v_2m_3^{-1}v_1v_3v_4v_5v_7 =
  m_3v_3m_3^{-1}v_1v_2v_3 = 1, \displaybreak[0]\\
&m_3v_4m_3^{-1}v_1v_4v_5v_6v_7 = m_3v_5m_3^{-1}v_2v_3v_4 =
  m_3v_6m_3^{-1}v_1v_3v_4 =1, \displaybreak[0]\\
&m_3v_7m_3^{-1}v_1v_3v_5 = m_3v_8m_3^{-1}v_8 = 1, \displaybreak[0]\\
&m_4v_1m_4^{-1}v_1v_2v_5v_7
  =m_4v_2m_4^{-1}v_1v_3v_7v_8=m_4v_3m_4^{-1}v_4v_6v_7v_8 =1, \displaybreak[0]\\
&m_4v_4m_4^{-1}v_1v_3v_4v_5v_6v_7v_8=m_4v_5m_4^{-1}v_1v_2v_3v_6
  =m_4v_6m_4^{-1}v_4v_5v_8=1, \displaybreak[0]\\
&m_4v_7m_4^{-1}v_1v_2v_6v_8 = m_4v_8m_4^{-1}v_3v_4v_5v_7=1, \displaybreak[0]\\
&m_5v_1m_5^{-1}v_1v_4v_8 =
  m_5v_2m_5^{-1}v_1v_3v_8=m_5v_3m_5^{-1}v_2v_4v_5v_6v_7v_8 =1, \displaybreak[0]\\
&m_5v_4m_5^{-1}v_1v_3v_5v_6v_7v_8 = m_5v_5m_5^{-1}v_1v_6 =
  m_5v_6m_5^{-1}v_4v_5v_6v_8=1, \displaybreak[0]\\
&m_5v_7m_5^{-1}v_1v_2v_6v_7v_8 = m_5v_8m_5^{-1}v_3v_5v_6v_7v_8=1.
\end{alignat*}

\item[\rm(n)] 
$\dim_F[H^2(H_1,F)] = 2$ and there exists a unique
central extension $H$ of $H_1$ whose Sylow $2$-subgroups are
isomorphic to the ones of $D = C_{E_3}(z)$. $H$ is a split
extension of $K$ by its Fitting subgroup 
$Q$ and $H$ is
isomorphic to the finitely presented group 
$H = \langle h_i \,|\, 1 \le
i \le 14\rangle$ with the following set ${\mathcal R}(H)$ of
defining relations:
\begin{alignat*}{1}
&h_1^7 = h_2^2 = h_3^4 = h_4^5 = h_5^3 = h_6^2 = h_7^2 = h_{10}^2 =
  h_{13}^2 = h_{14}^2 = 1, \displaybreak[0]\\
&h_8^2 = h_9^2 = h_{11}^2 = h_{12}^2 = h_{14}, \quad [h_1,h_{14}] = 1,
  \displaybreak[0]\\
&[h_2,h_{14}] = [h_3,h_{14}] = [h_4,h_{14}] = [h_5,h_{14}] =
  [h_6,h_{14}] = [h_7,h_{14}] = 1, \displaybreak[0]\\
&[h_8,h_{14}] = [h_9,h_{14}] = [h_{10},h_{14}] = [h_{11},h_{14}] =
  [h_{12},h_{14}] = [h_{13},h_{14}] = 1, \displaybreak[0]\\
&h_2  h_4^{-1}  h_3^{-1}  h_2  h_3  h_4 = (h_3^{-1}  h_4^{-1})^2 = h_3
  h_5  h_4^{-1}  h_3^{-1}  h_5^{-1}  h_4 = 1, \displaybreak[0]\\
&(h_5^{-1}  h_4^{-1})^3 = (h_1^{-1}  h_2)^4 = (h_2  h_3^{-1})^4 = (h_2
  h_3  h_2  h_3^{-1})^2 = 1, \displaybreak[0]\\ 
& h_5^{-1}  h_3^{-1}  h_4  h_3^{-1}  h_5^{-1}  h_4^{-2} = h_2  h_3
  h_5^{-1}  h_3^{-1}  h_2  h_3  h_5  h_3^{-1} = 1, \displaybreak[0]\\
& h_4  h_1^{-1}  h_4  h_3  h_1  h_4^{-1}  h_3  h_4 = h_4^{-1}  h_3
  h_5^{-1}  h_4^{-1}  h_3  h_5^{-1}  h_4  h_5 = 1,  \displaybreak[0]\\
& [h_5, h_1^{-1}, h_5] = (h_1^{-2}  h_2)^3 = h_5  h_3  h_1  h_2
  h_1^{-1}  h_5^{-1}  h_2  h_5^{-1}  h_4^{-1} = 1,  \displaybreak[0]\\
& (h_3^{-1}  h_1^{-1}  h_2)^3 = (h_3  h_1^{-1}  h_2)^3 = h_1  h_3^{-1}
  h_2  h_3  h_1^{-1}  h_4^{-2}  h_3  h_5 = 1,  \displaybreak[0]\\
& h_1^{-1}  h_3  h_2  h_1^3  h_3  h_1^2  h_3^{-1} = h_2  h_1  h_3
  h_1^3  h_2  h_1  h_2  h_3^{-1} = 1,  \displaybreak[0]\\
& h_1  h_6  h_1^{-1}  h_{12}  h_{13} = h_1  h_7  h_1^{-1}  h_6 = h_1
  h_8  h_1^{-1}  h_7  h_{13} = 1,  \displaybreak[0]\\
& h_1  h_9  h_1^{-1}  h_8  h_{14} = h_1  h_{10}  h_1^{-1}  h_9  h_{13}
  = h_1  h_{11}  h_1^{-1}  h_{10}  h_{13} = 1,  \displaybreak[0]\\
& h_1  h_{12}  h_1^{-1}  h_{11}  h_{14} = h_1  h_{13}  h_1^{-1}
  h_{13} = h_2  h_6  h_2^{-1}  h_9  h_{10}  h_{11} = 1,  \displaybreak[0]\\
& h_2  h_7  h_2^{-1}  h_6  h_7  h_8  h_9  h_{12}  h_{13}  h_{14} = h_2
  h_8  h_2^{-1}  h_{10}  h_{11}  h_{12}  h_{13} = 1,  \displaybreak[0]\\
& h_2  h_9  h_2^{-1}  h_8  h_9  h_{10}  h_{11}  h_{12}  h_{13} = h_2
  h_{10}  h_2^{-1}  h_6  h_9  h_{11} = 1,  \displaybreak[0]\\
& h_2  h_{11}  h_2^{-1}  h_8  h_{10}  h_{12}  h_{13} = h_2  h_{12}
  h_2^{-1}  h_6  h_9  h_{10}  h_{11}  h_{12}  h_{14} = 1,  \displaybreak[0]\\
& h_2  h_{13}  h_2^{-1}  h_{13} = h_3  h_6  h_3^{-1}  h_6  h_7  h_{12}
  h_{13}  h_{14} = 1,  \displaybreak[0]\\
& h_3  h_7  h_3^{-1}  h_6  h_8  h_9  h_{10}  h_{12}  h_{14} = h_3  h_8
  h_3^{-1}  h_6  h_7  h_8  h_{14} = 1,  \displaybreak[0]\\
& h_3  h_9  h_3^{-1}  h_6  h_9  h_{10}  h_{11}  h_{12} = h_3  h_{10}
  h_3^{-1}  h_7  h_8  h_9 = 1,  \displaybreak[0]\\
& h_3  h_{11}  h_3^{-1}  h_6  h_8  h_9 = h_3  h_{12}  h_3^{-1}  h_6
  h_8  h_{10}  h_{14} = 1,  \displaybreak[0]\\
& h_3  h_{13}  h_3^{-1}  h_{13} = h_4  h_6  h_4^{-1}  h_6  h_7  h_{10}
  h_{12} = 1,  \displaybreak[0]\\
& h_4  h_7  h_4^{-1}  h_6  h_8  h_{12}  h_{13} = h_4  h_8  h_4^{-1}
  h_9  h_{11}  h_{12}  h_{13}  h_{14} = 1,  \displaybreak[0]\\
& h_4  h_9  h_4^{-1}  h_6  h_8  h_9  h_{10}  h_{11}  h_{12}  h_{13} =
  h_4  h_{10}  h_4^{-1}  h_6  h_7  h_8  h_{11}  h_{14} = 1,  \displaybreak[0]\\
& h_4  h_{11}  h_4^{-1}  h_9  h_{10}  h_{13}  h_{14} = h_4  h_{12}
  h_4^{-1}  h_6  h_7  h_{11}  h_{13} = 1,  \displaybreak[0]\\
& h_4  h_{13}  h_4^{-1}  h_8  h_9  h_{10}  h_{12}  h_{14} = h_5  h_6
  h_5^{-1}  h_6  h_9  h_{13}  h_{14} = 1,  \displaybreak[0]\\
& h_5  h_7  h_5^{-1}  h_6  h_8  h_{13} = h_5  h_8  h_5^{-1}  h_7  h_9
  h_{10}  h_{11}  h_{12}  h_{13}  h_{14} = 1,  \displaybreak[0]\\
& h_5  h_9  h_5^{-1}  h_6  h_8  h_{10}  h_{11}  h_{12}  h_{13} = h_5
  h_{10}  h_5^{-1}  h_6  h_{11} = 1,  \displaybreak[0]\\
& h_5  h_{11}  h_5^{-1}  h_9  h_{10}  h_{11}  h_{13}  h_{14} = h_5
  h_{12}  h_5^{-1}  h_6  h_7  h_{11}  h_{12}  h_{13}  h_{14} = 1,  \displaybreak[0]\\
& h_5  h_{13}  h_5^{-1}  h_8  h_{10}  h_{11}  h_{12}  h_{13} =
  h_6^{-1}  h_7^{-1}  h_6  h_7  h_{14} = 1,  \displaybreak[0]\\
& h_6^{-1}  h_8^{-1}  h_6  h_8  h_{14} = h_6^{-1}  h_9^{-1}  h_6  h_9
  h_{14} = [h_6, h_{10}] = 1,  \displaybreak[0]\\
& h_6^{-1}  h_{11}^{-1}  h_6  h_{11}  h_{14} = [h_6, h_{12}] =
  h_6^{-1}  h_{13}^{-1}  h_6  h_{13}  h_{14} = 1,  \displaybreak[0]\\
& [h_7, h_8] = h_7^{-1}  h_9^{-1}  h_7  h_9  h_{14} = [h_7, h_{10}] =
  h_7^{-1}  h_{11}^{-1}  h_7  h_{11}  h_{14} = 1,  \displaybreak[0]\\
& h_7^{-1}  h_{12}^{-1}  h_7  h_{12}  h_{14} = h_7^{-1}  h_{13}^{-1}
  h_7  h_{13}  h_{14} = 1,  \displaybreak[0]\\
& h_8^{-1}  h_9^{-1}  h_8  h_9  h_{14} = h_8^{-1}  h_{10}^{-1}  h_8
  h_{10}  h_{14} = [h_8, h_{11}] = 1,  \displaybreak[0]\\
& [h_8, h_{12}] = h_8^{-1}  h_{13}^{-1}  h_8  h_{13}  h_{14} = [h_9,
  h_{10}] = [h_9, h_{11}] = [h_9, h_{12}] = 1,  \displaybreak[0]\\
& h_9^{-1}  h_{13}^{-1}  h_9  h_{13}  h_{14} = [h_{10}, h_{11}] =
  h_{10}^{-1}  h_{12}^{-1}  h_{10}  h_{12}  h_{14} = 1,  \displaybreak[0]\\
& h_{10}^{-1}  h_{13}^{-1}  h_{10}  h_{13}  h_{14} = h_{11}^{-1}
  h_{12}^{-1}  h_{11}  h_{12}  h_{14} = 1,  \displaybreak[0]\\
& h_{11}^{-1}  h_{13}^{-1}  h_{11}  h_{13}  h_{14} = h_{12}^{-1}  h_{13}^{-1}  h_{12}  h_{13}  h_{14} = 1.
\end{alignat*}
\end{enumerate}
\end{proposition}

\begin{proof}
(a) By Lemma \ref{l. classes} the split extension $E_3 =
V_3\rtimes A_{22} = \langle p, q, v_1\rangle$ has a unique
conjugacy class of involutions of highest defect. It is
represented by $z = (pq^2v_1)^5$ and its centralizer $D =
C_{E_3}(z)$ has order $2^{18}\cdot 3^2\cdot 5$.

Using the faithful permutation representation of $E_3$ with
stabilizer $A_{22}$ and \textsc{Magma} it has been checked that $D$ has a
unique normal subgroup $Q$ of order $512$ with center $Z(Q) =
\langle z \rangle$. Furthermore, $Q$ is extra-special and it has a
complement $W$ of order $2^9\cdot 3^2\cdot 5$. Another application
of \textsc{Magma} yields that the Fitting subgroup $B$ of $W$ is elementary
abelian of order $2^5$, and that $B$ has a complement $L \cong
S_6$ in $W$. Using then a stand alone program 
written by the first
author we found the generators of the subgroups $B = \langle f_j
\mid 1 \le j \le 5$, $L = \langle n_l \mid 1 \le l \le 4 \rangle$
and $Q = \langle n_l \mid 5 \le l \le 12 \rangle$ of $D$. Hence
(a) holds.

(b) Since $Q$ is extra-special of order $2^9$ and its center $Z(Q)
= Z(D) = \langle z \rangle$ we determined by means of \textsc{Magma}
eight involutions $q_i$ generating $Q$. Using the stand alone program
mentioned before the words in the generators $n_l$ of $Q$ of the
eight involutions were calculated. Since their residue classes
generate the elementary abelian normal subgroup $V = Q/Z(Q)$ of
$D_1 = D/Z(Q)$ we obtained the following set $\mathcal R(Q)$ of
defining relations of $Q = \langle q_i \mid 1 \le i \le 8
\rangle$:
\begin{alignat*}{1}
&q_1^2 = q_2^2 = q_3^2 = q_4^2 = q_5^2 = q_6^2 = q_7^2 = q_8^2 = 1, \displaybreak[0]\\
& [q_1, q_2] = [q_1, q_3] = [q_1, q_4] = [q_1, q_6] = [q_1, q_8] =
  1,\quad [q_1, q_5] = [q_1, q_7] = z, \displaybreak[1]\\
& [q_2, q_3] = [q_2, q_4] = [q_2, q_5] = [q_2, q_6] = [q_2, q_7] = 1,
  \quad [q_2, q_8] = z, \displaybreak[1]\\
&[q_3, q_4] = [q_3, q_6] = [q_3, q_7] = 1, \quad [q_3, q_5] = [q_3, q_8] = z,\displaybreak[1]\\
&[q_4, q_5] = [q_4, q_6] = [q_4, q_7] = [q_4, q_8] = z, \displaybreak[0]\\
&[q_5, q_6] = [q_5, q_7] = [q_5, q_8] = [q_6, q_7] = [q_6, q_8] = [q_7, q_8] = 1.
\end{alignat*}
In particular, $z=(q_4q_8)^2$, hence (b) holds.

(c) Another application of \textsc{Magma} and the permutation
representation 
$P\!E_3$ yields that the permutation group $PL$ of
the complement $L = \langle n_k \mid 1 \le k \le 4 \rangle$ can be
generated by the 
three elements $s_1$, $s_2$ and $s_3$. For these
generators $s_i$ \textsc{Magma} provided the following set $\mathcal R(L)$
of defining relations:
\begin{eqnarray*}
&& s_1^4 = s_2^5 = s_3^3 = 1,\\
&&(s_2s_1)^2 = s_1s_3s_2^{-1}s_1^{-1}s_3^{-1}s_2 = (s_2s_3)^3 = s_2^2s_3s_1s_2^{-1}s_1s_3 = 1,\\
&&s_2^{-1}s_1s_3^{-1}s_2^{-1}s_1s_3^{-1}s_2s_3 = 1.
\end{eqnarray*}
Using \textsc{Magma} again it has been checked that $L \cong S_6$. Thus (c)
holds.

(d) Since the elementary abelian subgroup $B$ is normal in the
semidirect product $W$ of $L$ by $B$ the presentation of $W =
\langle f_j, s_k \rangle$ can easily be calculated by means of
$\mathcal R(L)$ and the conjugation action of the $s_k \in L$ on
$B$. Hence $W$ has a defining set $\mathcal R(W)$ of relations
consisting of $\mathcal R(L)$ and the following relations:
\begin{alignat*}{1}
&f_1^2  = f_2^2  = f_3^2  = f_4^2  = f_5^2  = 1, \displaybreak[0]\\
&f_1^{s_1}(f_1)^{-1} = f_1^{s_2}(f_1)^{-1} = f_1^{s_3}(f_1)^{-1} = 1, \displaybreak[0]\\
&f_2^{s_1}(f_1f_4f_5)^{-1} = f_2^{s_2}(f_1f_2f_3f_5)^{-1} =
  f_2^{s_3}(f_3f_5)^{-1} = 1, \displaybreak[0]\\
&f_3^{s_1}(f_3 f_5)^{-1} = f_3^{s_2}(f_2f_3)^{-1} =
  f_3^{s_3}(f_2f_5)^{-1} = 1, \displaybreak[0]\\
&f_4^{s_1}(f_1f_3f_4f_5)^{-1} = f_4^{s_2}(f_1f_3f_4)^{-1} =
  f_4^{s_3}(f_1f_3f_4)^{-1} = 1, \displaybreak[0]\\
&f_5^{s_1}(f_2f_3f_4)^{-1} = f_5^{s_2}(f_2f_3f_4)^{-1} =
  f_5^{s_3}(f_1f_2)^{-1} = 1, \displaybreak[0]\\
&[f_1, f_2] = [f_1, f_3] = [f_1, f_4] = [f_1, f_5] = [f_2, f_3] =
  [f_2, f_4] = [f_2, f_5] = 1, \\
&[f_3, f_4]  = [f_3, f_5]  = [f_4, f_5]  = 1.
\end{alignat*}
As $Q$ is normal in the semidirect product $D$ of $W$ by $Q$ the
conjugate action of the $f_j$ and $s_k$ on $Q$ provides the
following set $\mathcal{R}(Q)$ of relations:
\begin{alignat*}{1}
&q_1^{s_1}(q_3q_4q_6q_7q_8 )^{-1} = z,\quad
  q_1^{s_2}(q_1q_2q_3q_8)^{-1} = q_1^{s_3}(q_1q_3q_5q_7q_8)^{-1} =
  1,\\ 
&q_1^{f_1}(q_1q_5q_7 )^{-1} = q_1^{f_2}(q_1)^{-1} =
  q_1^{f_3}(q_1q_5q_6q_7)^{-1} = q_1^{f_4}(q_1)^{-1} = 1, \displaybreak[0]\\
&q_1^{f_5}(q_1 q_6 q_8 )^{-1} = q_2^{s_1}(q_1q_4q_5)^{-1} = (q_4
  q_8)^2, \displaybreak[0]\\
& q_2^{s_2}(q_2q_4q_5q_6)^{-1} = q_2^{s_3}(q_1q_2q_4q_5 )^{-1} = 1, \displaybreak[0]\\
&q_2^{f_1}(q_2q_6)^{-1} = (q_4 q_8)^2,\quad q_2^{f_2}(q_2 )^{-1} =
  q_2^{f_3}(q_2q_5q_6q_7 )^{-1} = 1, \displaybreak[0]\\
&q_2^{f_4}(q_2q_5q_7)^{-1} = q_2^{f_5}(q_2q_6q_7 )^{-1} = z,\quad
  q_3^{s_1}(q_3q_5q_7q_8 )^{-1} = z,  \displaybreak[0]\\
& q_3^{s_2}(q_2q_3q_6q_8 )^{-1} = (q_4 q_8)^2,\quad
  q_3^{s_3}(q_1q_5q_6q_7q_8 )^{-1} = 1,\quad q_3^{f_1}(q_3q_7 )^{-1} =
  z, \displaybreak[0]\\
&q_3^{f_2}(q_3q_6)^{-1} = q_3^{f_3}(q_3q_5q_8)^{-1} = 1,\quad
  q_3^{f_4}(q_3q_5q_6q_7q_8 )^{-1} = q_3^{f_5}(q_3 q_7 )^{-1} = (q_4
  q_8)^2, \displaybreak[0]\\
&q_4^{s_1}(q_2q_4q_6q_7)^{-1} = q_4^{s_2}(q_1q_3q_5q_7q_8)^{-1} =
  1,\quad q_4^{s_3}(q_1q_2q_3q_5q_8)^{-1} = z, \displaybreak[0]\\
&q_4^{f_1}(q_4q_6q_8)^{-1} = 1, \quad q_4^{f_2}(q_4q_5q_7)^{-1} =
  q_4^{f_5}(q_4q_5q_7)^{-1} = z, \displaybreak[0]\\
& q_4^{f_3}(q_4 q_5 q_8 )^{-1} = q_4^{f_4}  ( q_4 )^{-1} = 1,\displaybreak[0]\\
& q_5^{s_1}(q_6q_8)^{-1} = q_5^{s_2}(q_5q_7)^{-1} = 1, \quad
  q_5^{s_3}(q_7)^{-1} = (q_4 q_8)^2,\quad q_5^{f_1}(q_5)^{-1} = 1, \displaybreak[0]\\
& q_5^{f_2}(q_5 )^{-1} = q_5^{f_3}(q_5)^{-1} = q_5^{f_4}(q_5)^{-1} =
  q_5^{f_5}(q_5)^{-1} = 1, \displaybreak[0]\\
&q_6^{s_1}(q_5q_6q_7q_8)^{-1} = q_6^{s_2}(q_8 )^{-1} =
  q_6^{s_3}(q_5q_7q_8)^{-1} = z, \displaybreak[0]\\
&q_6^{f_1}(q_6)^{-1} = 1,\quad q_6^{f_2}(q_6)^{-1} =
  q_6^{f_3}(q_6)^{-1} = q_6^{f_4}(q_6)^{-1}  = q_6^{f_5}(q_6)^{-1} =
  1, \displaybreak[0]\\
&q_7^{s_1}(q_7)^{-1} = 1, \quad q_7^{s_2}(q_6q_7)^{-1} = q_7^{s_3}(q_5
  q_7 )^{-1} = (q_4 q_8)^2, \displaybreak[0]\\
&q_7^{f_1}(q_7)^{-1} = q_7^{f_2}(q_7)^{-1} = q_7^{f_3}(q_7)^{-1} =
  q_7^{f_4}(q_7)^{-1} = q_7^{f_5}(q_7)^{-1} = 1, \displaybreak[0]\\
&q_8^{s_1}(q_5q_6q_7)^{-1} = q_8^{s_2}(q_5q_8)^{-1} = (q_4
  q_8)^2,\quad q_8^{s_3}(q_6 q_7 q_8 )^{-1} = 1, \\
&q_8^{f_1}(q_8)^{-1} =  q_8^{f_2}(q_8)^{-1} =  q_8^{f_3}(q_8)^{-1} =
  q_8^{f_4}(q_8)^{-1} = q_8^{f_5}(q_8 )^{-1} = 1, 
\end{alignat*}
Since $z$ generates the center $Z(D)$ of $D$, the additional relations
of the set $\mathcal R(D)$ of defining relations of $D$ given in (d)
are clear by the proof of (b). Using \textsc{Magma} again we have
checked that this is a presentation of $D$.

(e) All the statements follow immediately from the presentation of
$D$ given in~(d).

(f) From (c) and (d) follows that the Fitting subgroup $B$ of $W$
has order $2^5$ and a complement $L = \langle s_1, s_2, s_3\rangle
\cong S_6$ in $W$. Hence $W$ is isomorphic to the centralizer
$H(\Sp_6(2))$ of a $2$-central involution of $\Sp_6(2)$ by
Proposition 8.6.5 of~\cite{michler}. Thus we may apply Algorithm
7.4.8 of \cite{michler} to construct a simple 
overgroup $K$ of
$W$ such that $|K:W|$ is odd and $K$ normalizes $V$. For that we
choose the basis $\mathcal B = \{v_i = q_iZ(Q) \mid 1 \le i \le 8
\}$ of $V$. With respect to $\mathcal B$ the conjugate action of
the $s_k$ on $V$ is described by the 
three matrices $Ms_k$ of the
statement.

(g) From the presentation of $D$ follows that $u = f_1$ is an
involution of $Z(W)$. Another application of \textsc{Magma} yields that $W
= \langle u,r,y_1,y_2,y_3 \rangle$ for $r = f_2f_4$ and that the
conjugate actions of $r$ and $u$ on $V$ have the following
matrices 
with respect to~$\mathcal B$: 
{\renewcommand{\arraystretch}{0.5}
\scriptsize
$$
Mu = \left( \begin{array}{*{8}{c@{\,}}c}
 1 &   0 &   0 &   0 &   0 &   0 &   0 &   0 \\
 0 &   1 &   0 &   0 &   0 &   0 &   0 &   0 \\
 0 &   0 &   1 &   0 &   0 &   0 &   0 &   0 \\
 0 &   0 &   0 &   1 &   0 &   0 &   0 &   0 \\
 1 &   0 &   0 &   0 &   1 &   0 &   0 &   0 \\
 0 &   1 &   0 &   1 &   0 &   1 &   0 &   0 \\
 1 &   0 &   1 &   0 &   0 &   0 &   1 &   0 \\
 0 &   0 &   0 &   1 &   0 &   0 &   0 &   1 
\end{array} \right)\quad \text{and}~\quad
Mr = \left( \begin{array}{*{8}{c@{\,}}c}
 1 &   0 &   0 &   0 &   0 &   0 &   0 &   0 \\
 0 &   1 &   0 &   0 &   0 &   0 &   0 &   0 \\
 0 &   0 &   1 &   0 &   0 &   0 &   0 &   0 \\
 0 &   0 &   0 &   1 &   0 &   0 &   0 &   0 \\
 0 &   1 &   1 &   1 &   1 &   0 &   0 &   0 \\
 0 &   0 &   0 &   0 &   0 &   1 &   0 &   0 \\
 0 &   1 &   1 &   1 &   0 &   0 &   1 &   0 \\
 0 &   0 &   1 &   0 &   0 &   0 &   0 &   1 
\end{array} \right).
$$
}

(h) Using the faithful permutation representation $PD_1$ of $D_1$
given in (e) and \textsc{Magma} we find the generators of the given Sylow
$2$-subgroup $S$ of $W$. It has a unique maximal elementary
abelian normal subgroup $A = \langle a_1,a_2,a_3,a_4,a_5=u,a_6
\rangle$ of order $64$. Its generators $a_i$ are represented by
the following matrices with respect to $\mathcal B$:
{\renewcommand{\arraystretch}{0.5}
\scriptsize
\begin{alignat*}{3}
Ma_1 &= \left( \begin{array}{*{8}{c@{\,}}c}
 1 &   0 &   0 &   0 &   0 &   0 &   0 &   0 \\
 0 &   1 &   0 &   0 &   0 &   0 &   0 &   0 \\
 1 &   0 &   1 &   1 &   0 &   0 &   0 &   0 \\
 0 &   0 &   0 &   1 &   0 &   0 &   0 &   0 \\
 0 &   0 &   0 &   0 &   1 &   0 &   0 &   0 \\
 0 &   0 &   0 &   0 &   0 &   1 &   0 &   0 \\
 0 &   0 &   0 &   0 &   1 &   0 &   1 &   1 \\
 0 &   0 &   0 &   0 &   0 &   0 &   0 &   1 
\end{array} \right), &\quad
Ma_2 &= \left( \begin{array}{*{8}{c@{\,}}c}
 0 &   0 &   0 &   1 &   0 &   0 &   0 &   0 \\
 0 &   1 &   0 &   0 &   0 &   0 &   0 &   0 \\
 0 &   1 &   1 &   0 &   0 &   0 &   0 &   0 \\
 1 &   0 &   0 &   0 &   0 &   0 &   0 &   0 \\
 0 &   0 &   0 &   0 &   0 &   0 &   0 &   1 \\
 0 &   0 &   0 &   0 &   1 &   1 &   0 &   1 \\
 0 &   0 &   0 &   0 &   1 &   1 &   1 &   0 \\
 0 &   0 &   0 &   0 &   1 &   0 &   0 &   0 
\end{array} \right),
&\qquad
Ma_3 &= \left( \begin{array}{*{8}{c@{\,}}c}
 0 &   0 &   0 &   1 &   0 &   0 &   0 &   0 \\
 1 &   1 &   0 &   1 &   0 &   0 &   0 &   0 \\
 1 &   1 &   1 &   0 &   0 &   0 &   0 &   0 \\
 1 &   0 &   0 &   0 &   0 &   0 &   0 &   0 \\
 0 &   0 &   0 &   0 &   0 &   0 &   0 &   1 \\
 0 &   0 &   0 &   0 &   0 &   1 &   0 &   0 \\
 0 &   0 &   0 &   0 &   0 &   1 &   1 &   0 \\
 0 &   0 &   0 &   0 &   1 &   0 &   0 &   0 
\end{array} \right),\\
Ma_4 &= \left( \begin{array}{*{8}{c@{\,}}c}
 0 &   0 &   0 &   1 &   0 &   0 &   0 &   0 \\
 1 &   1 &   0 &   1 &   0 &   0 &   0 &   0 \\
 1 &   1 &   1 &   0 &   0 &   0 &   0 &   0 \\
 1 &   0 &   0 &   0 &   0 &   0 &   0 &   0 \\
 1 &   0 &   0 &   0 &   0 &   0 &   0 &   1 \\
 1 &   1 &   0 &   0 &   0 &   1 &   0 &   0 \\
 1 &   0 &   1 &   1 &   0 &   1 &   1 &   0 \\
 0 &   0 &   0 &   1 &   1 &   0 &   0 &   0 
\end{array} \right), &\ 
\text{and\ }Ma_6 &= \left( \begin{array}{*{8}{c@{\,}}c}
 1 &   0 &   0 &   0 &   0 &   0 &   0 &   0 \\
 0 &   1 &   0 &   0 &   0 &   0 &   0 &   0 \\
 0 &   0 &   1 &   0 &   0 &   0 &   0 &   0 \\
 0 &   0 &   0 &   1 &   0 &   0 &   0 &   0 \\
 1 &   0 &   0 &   1 &   1 &   0 &   0 &   0 \\
 0 &   0 &   0 &   0 &   0 &   1 &   0 &   0 \\
 1 &   1 &   0 &   0 &   0 &   0 &   1 &   0 \\
 1 &   0 &   0 &   1 &   0 &   0 &   0 &   1 
\end{array} \right).%
\end{alignat*}}%
Observe that the map 
$\varphi\colon Px \rightarrow Mx$ sending the
permutation $Px$ of $PD_1$ to its matrix $Mx$ with respect to the
basis $\mathcal{B}$ yields an anti-epimorphism from $PD_1$ onto
$MW = \langle Mr,Mu,Ms_1,Ms_2,Ms_3 \rangle \le 
\GL_8(2)$ with kernel $PV$.

(i) 
$M\!A = \langle Ma_1,Ma_2,Ma_3,Ma_4,Ma_5 = Mu,Ma_6 \rangle$ is
an elementary abelian subgroup of $\Gamma = \GL_8(2)$ of order
$2^6$. In view of Proposition 8.6.5(i) of \cite{michler} we now
calculate 
$N\!A = N_{\Gamma}(M\!A)$. The matrix $Ms$ of the statement
generates a Sylow $7$-subgroup of order $7$ of 
$N\!A$. 

(j) Let 
$M\!K = \langle MW, Ms \rangle$. Then an application of
\textsc{Magma} yields that $MU = 
N\!A \cap M\!K = N_{M\!K}(M\!A)$ has order $
2^9\cdot3\cdot7$. It is a split extension of $\GL_3(2)$ by 
$M\!A$. Furthermore, 
$M\!K  = \langle MW, MU \rangle$ is isomorphic to the
simple group 
$\Sp_6(2)$. 

(k) The group 
$M\!K$ has a faithful permutation representation of
degree $63$ with stabilizer $MW$. Using it and \textsc{Magma} the set
${\mathcal R}(K)$ of the defining relations of $K = \langle
m_1,m_2,m_3,m_4,m_5 \rangle$ has been calculated where $m_1 = Ms$,
$m_2 = Mr$, $m_3 = Ms_1$, $m_4 = My_4$ and $m_5 = Ms_3$.

(l) By (k) and 
(j), $K$ has an irreducible $8$-dimensional matrix
representation~$V$. The given presentation of the split extension
$H_1 = V\rtimes K$ has been calculated by means of \textsc{Magma} using the
presentation of $K$ given in (k).

(m) Assertion (c) implies that $D_1 = D/Z(D)$ splits over $V =
Q/Z(D)$ with complement $W$. By Algorithm 2.5 of \cite{michler1}
the centralizer $H = C_G(z)$ of $z$ in the 
(at this time unknown) target simple group $G$ must
have 
odd index $|H : D| = |K :
W|$. Therefore Theorem 1.4.15 of \cite{michler} implies that $H$
is a central extension of $H_1$ by a cyclic subgroup $Z(H) =
\langle z \rangle$ having a normal subgroup $Q$ containing $z$
such that $V = Q/Z(H)$ has a complement in $H/Z(H)$ isomorphic to
$K$.

By construction $H_1 = \langle
m_1,m_2,m_3,m_4,m_5,v_1,v_2,v_3,v_4,v_5,v_6,v_7,v_8 \rangle$ has a
faithful permutation representation $PH_1$ of degree $256$ with
stabilizer $K$. Let $FPH_1$ be the presentation of $H_1$ given in
(m). Then we can apply Holt's Algorithm implemented in \textsc{Magma}
\cite{holt} to the trivial matrix representation of $H_1$ over $F
= 
\GF(2)$. It yields that the second cohomological dimension
$\dim_F[H^2(H_1,F)] = 2$. Thus there are 
three non-split central
extensions $E_{0,1}, E_{1,0}, E_{1,1}$ of $H_1$. As $D$ does not
split over $Z(D)$ $H$ can only be isomorphic to a 
non-split extension.

Let $TH_1 : = \verb"GModule(PH_1, FEalg)"$ be the trivial module
of the matrix algebra $\verb"FEalg"$ generated by the 
thirteen identity matrices corresponding to the 
thirteen generators of~$H_1$. Using the \textsc{Magma} command
\[\verb"P_H :=ExtensionProcess(PH_1,TH_1,FPH_1)"\]
we construct a presentation for each of the 
three non-split central
extensions
\begin{eqnarray*}
E_{0,1} &:=& \verb"Extension(P_H, [0,1])",\\
E_{1,0} &:=& \verb"Extension(P_H, [1,0])",\mbox{\ and}\\
E_{1,1} &:=& \verb"Extension(P_H, [1,1])",
\end{eqnarray*}
corresponding to the 
three linearly independent $2$-cocycles $[a,b] \in F^2$. Each of these
three presentations has $14$ generators where the 
$14$th generator corresponds to the new central element $z$. The first
five generators correspond to $m_1, m_2,m_3,m_4, m_5$ in the factor
group $H_1$. Taking the corresponding generators in $E_{1,0}$ as
the stabilizer of a permutation representation $PE_{1,0}$ of
$E_{1,0}$ it follows by another application of \textsc{Magma} that
$PE_{1,0}$ is a faithful permutation of $E_{1,0}$ of degree $512$,
and that $E_{1,0}$ has an extra-special normal subgroup $Q$ of
order $2^9$ with complement $K =\nobreak \langle m_1, m_2,m_3,m_4, m_5
\rangle$. We also have checked that the Sylow $2$-subgroups of
$E_{1,0}$ and $D$ are isomorphic, and that this is not the case
for the the 
two other extensions $E_{0,1}$ and $E_{1,1}$.
Therefore only $E_{1,0}$ can be isomorphic to the centralizer $H =
C_G(z)$. The presentation of $E_{1,0}$ is given in the statement.
This completes the proof.
\end{proof}

\begin{lemma}\label{l. H(Co_2)} Keep the notation of Lemmas \ref{l.
M22-extensions}, \ref{l. Aut(M22)-extensions} and Proposition
\ref{prop. D(Co_2)}. Let $H = 
\langle h_i \,|\, 1 \le i \le 14 \rangle $
be the finitely presented group constructed in
Proposition~\ref{prop. D(Co_2)}. Then the following statements hold: 
\begin{enumerate}
\item[\rm(a)] $H$ has a faithful permutation representation of
degree $512$ with stabilizer $\langle h_1, h_4 \rangle$.

\item[\rm(b)] Each Sylow $2$-subgroup $S$ of $H$ has a unique
maximal elementary abelian normal subgroup $A$ of order $2^{10}$
and $N_H(A) \cong D = C_{E_3}(z)$.

\item[\rm(c)] There is a Sylow $2$-subgroup $S$ such that $D =
N_{H}(A) = \langle x, y \rangle$, where $x =\nobreak h_2h_{16}h_1$ and $y
= h_9h_{14}h_{15}$ have 
orders $12$ and~$6$, respectively. 
Furthermore, $H = \langle x, y, h \rangle$ where $h = h_1$ has
order $7$.

\item[\rm(d)] The Goldschmidt index of the amalgam $H \leftarrow D
\rightarrow E_3$ is $2$.

\item[\rm(e)] A system of representatives $r_i$ of the $100$
conjugacy classes of $H$ and the corresponding centralizers orders
$|C_H(r_i)|$ are given in Table \ref{Co_2cc H}.

\item[\rm(f)] A system of representatives $d_i$ of the $148$
conjugacy classes of $D$ and the corresponding centralizers orders
$|C_D(d_i)|$ are given in Table \ref{Co_2cc D}.

\item[\rm(g)] Let $\sigma:N_H(A) \rightarrow D = C_{E_3}(z)$ be
the 
isomorphism given in {\rm (b).} Then there is an element $e \in E_3$
of order $4$ such that $E_3 = \langle \sigma(D), e\rangle$. A
system of representatives $e_i$ of the $79$ conjugacy classes of
$E_3$ and the corresponding centralizers orders $|C_{E_3}(e_i)|$
are given in Table \ref{Co_2cc E}.

\item[\rm(h)] The character tables of $H$, $D$ and $E_3$ are given
in Tables \ref{Co_2ct_H}, 
\ref{Co_2ct_D}, and~\ref{Co_2ct_E},
respectively.
\end{enumerate}
\end{lemma}

\begin{proof}
(a) This assertion follows at once from Proposition \ref{prop.
D(Co_2)}(n).

In particular, $H$ has a faithful permutation representation $PH$
of degree $512$. Using it and \textsc{Magma} it is straightforward to
verify statements (b) and (c).

(d) The Goldschmidt index has been calculated by means of
Kratzer's Algorithm 7.1.10 of \cite{michler}.

The systems of representatives of the conjugacy classes of $H$ and
$D$ have been calculated by means of $PH$, \textsc{Magma} and Kratzer's
Algorithm 5.3.18 of \cite{michler}.

(g) It has been checked with \textsc{Magma} that $e=p$ satisfies
$E_3 = \langle \sigma(D),e\rangle$.

The character tables 
mentioned in (h) have been calculated by means of $PH$
and \textsc{Magma}.
\end{proof}

\begin{proposition}\label{prop. D(Fi_22)} Keep the notation of Lemma \ref{l.
M22-extensions}. Let 
\[E_2 = \langle a,b,c,d,t,g,h,i,v_1 \rangle\]
be the split extension of $\M_{22}$ by its simple module $V_2$ of
dimension $10$ over 
$F = \GF(2)$. Then the following statements
hold:

\begin{enumerate}
\item[\rm(a)] $z = (iv_1)^2$ is a $2$-central involution of $E_2$
with centralizer $D = C_{E_2}(z)$.

\item[\rm(b)] $D$ is a finitely presented group $D = \langle 
y_i \,|\, 1
\le i \le 11\rangle$ having the following set ${\mathcal R}(D)$ of
defining relations:
\begin{alignat*}{1}
&\smash{y_2^2 = y_3^2 = y_5^2 = y_6^2 = y_7^2 = y_8^2 = y_9^2 = y_{11}^2 =
  (y_5y_9)^2 = (y_5y_{11})^2 = (y_9y_{11})^2 =1,}\\ 
&\smash{y_{11}y_3y_5y_4y_{10}y_3y_5y_4y_{10}y_5^{-1}
  y_{10}^{-1}y_4^{-1}y_5^{-1}y_3^{-1}y_{10}^{-1}y_4^{-1}y_5^{-1}y_3^{-1}
  = 1,}\displaybreak[0]\\ 
&\smash{y_{11}y_3y_5y_4y_{10}y_3y_5y_4y_{10}y_9^{-1}y_{10}^{-1}y_4^{-1}y_5^{-1}y_3^{-1}
  y_{10}^{-1}y_4^{-1}y_5^{-1}y_3^{-1}y_5 
  = 1,}\displaybreak[0]\\ 
&\smash{y_{11}y_{10}^{-1}y_4^{-1}y_5^{-1}y_3^{-1}y_{10}^{-1}y_4^{-1}y_5^{-1}y_3^{-1}
  y_9^{-1}y_5^{-1}y_3y_5y_4y_{10}y_3y_5y_4y_{10}
  = 1}, \displaybreak[0]\\ 
&\smash{y_{10}y_5^{-1}y_{10}^{-1}y_4^{-1}y_5^{-1}y_3^{-1}y_5y_9y_3y_5y_4 =
  y_{11}y_3y_5y_4y_{10}y_9^{-1}y_{10}^{-1}y_4^{-1}y_5^{-1}y_3^{-1}y_5y_9
  = 1,}\displaybreak[0]\\
&\smash{y_{11}y_{10}^{-1}y_4^{-1}y_5^{-1}y_3^{-1}y_5^{-1}y_3y_5y_4y_{10}=
  (y_{10}y_3y_5y_4)^3 = y_9y_4^{-1}y_5^{-1}y_4 =
  y_9y_4^{-1}y_9^{-1}y_4y_5 = 1}, \displaybreak[0]\\
& \smash{y_{11}y_4^{-1}y_{11}^{-1}y_4y_5 =
  y_{11}y_3y_5y_4y_{10}y_4^{-1}y_3y_5y_4y_{10}y_4^{-1} y_9 =1,} \displaybreak[0]\\
&
  \smash{y_{10}y_4^{-1}y_{10}^{-1}y_4^{-1}y_5^{-1}y_3^{-1}y_4^{-1}y_3y_5y_4y_{10}y_4y_3y_5y_4
  =1,} \displaybreak[0]\\
&
  \smash{y_{11}y_3y_5y_4y_{10}y_3y_5y_4y_{10}y_4^{{-2}}y_{10}^{-1}
  y_4^{-1}y_5^{-1}y_3^{-1}y_4y_9 =1,}\displaybreak[0]\\ 
&\smash{y_{11}y_{10}y_5^{-1}y_{10}^{-1}y_{11}^{-1}y_3^{-1}y_5y_9y_3y_5y_4y_{10}
  y_3y_5y_4y_{10}y_4^{-1}y_{10}^{-1}y_4^{-1}y_5^{-1} = 1,} \displaybreak[0]\\
& \smash{y_{11}y_{10}y_9^{-1}y_{10}^{-1}y_{11}^{-1}y_3^{-1}
  y_5y_3y_5y_4y_{10}y_4^{-1}y_3 = 1,} \displaybreak[0]\\
&\smash{y_{11}y_{10}^{-1}y_{11}^{-1}y_3^{-1}y_4y_{10}^{-1}y_4^{-1}
  y_5^{-1}y_3^{-1}y_{11}^{-1}y_9^{-1}y_5^{-1}y_3y_{11}y_{10} = 1,} \displaybreak[0]\\
&\smash{y_{11}y_4^{-1}y_5^{-1}y_3^{-1}y_{10}^{-1}y_4^{-1}y_5^{-1}y_3^{-1}
  y_{10}^{-1}y_{11}^{-1}y_3^{-1}y_5y_9y_{11}y_3y_5y_4y_{10}y_3y_5y_4y_{10}y_3
  = 1,} \displaybreak[0]\\
& \smash{y_{11}y_4^{-1}y_5^{-1}y_3^{-1}y_4y_{10}^{-1}y_4^{-1}y_5^{-1}y_3^{-1}
  y_{11}^{-1}y_9^{-1}y_{10}^{-1}y_{11}^{-1}y_3^{-1}}\\
&\smash{\quad \cdot y_5y_3y_5y_4y_{10}y_3y_5y_4y_{10}y_4^{-1}y_3 =1,} \displaybreak[1]\\
& y_{11}y_{10}y_3y_{11}y_{10}y_5y_4^{-1}y_3 =
  y_{11}y_{10}y_3y_5y_9y_{10}^{-1}y_5y_3y_5y_4y_{10}y_4^{-1}y_{10}^{-1}y_4^{-1}y_5^{-1}
  =1, \displaybreak[0]\\
& \smash{y_{11}y_3y_5y_4y_{10}y_3y_5y_4y_{10}y_4^{-1}y_3y_5y_{10}^{-1}y_4^{-1}
  y_5^{-1}y_3^{-1}y_{10}^{-1}y_4^{-1}y_5^{-1}y_3^{-1}y_5^{-1}y_3^{-1}y_5
  = 1,} \displaybreak[0]\\
& \smash{y_{11}y_4^{-1}y_{10}^{-1}y_4^{-1}y_5^{-1}y_3^{-1}y_5^{-1}y_3^{-1}
  y_{10}^{-1}y_{11}^{-1}y_3^{-1}y_{11}^{-1}y_9^{-1}y_5^{-1} y_3y_5^2
  =1,} \displaybreak[0]\\
& \smash{y_{11}y_{10}y_5^{-1}y_3^{-1}y_{10}^{-1}y_{11}^{-1}
  y_3^{-1}y_4y_9^{-1}y_3y_5y_9y_3 = 1,}\displaybreak[0]\\ 
& \smash{y_{11}y_5^{-1}y_3^{-1}y_{10}^{-1}y_{11}^{-1}y_5
  y_4y_{10}y_4y_{11}^{-1}y_3y_5y_9 = (y_5y_3y_5)^2 =
  1,}\displaybreak[0]\\ 
& \smash{y_{11}y_3y_5y_4y_{10}y_4^{-1}y_3y_5y_4y_{10}
  y_3y_5y_4y_3y_5y_4^{-1}y_5^{-1}y_3^{-1}y_5 = 1,} \displaybreak[0]\\
& \smash{y_9y_3y_5y_4^{-1}y_2y_3^{-1}y_2^{-1}y_5 =
  y_9y_4^{-1}y_2y_4^{-1}y_2^{-1} = y_9y_2y_5^{-1}y_2^{-1} = 1,} \displaybreak[0]\\
& \smash{y_9y_2^{-1}y_5^{-1}y_2 = y_{10}y_2^{-1}y_5^{-1}y_4^{-1}y_{10}^{-1}y_2
  = y_{11}y_2^{-1}y_{11}^{-1}y_9^{-1}y_5^{-1}y_2 =1,} \displaybreak[0]\\
& \smash{y_9y_2y_1y_2y_1y_2^{-1}y_1^{-1}y_2^{-1}y_1^{-1}y_5
  =y_9y_4^{-1}y_1y_2y_1y_3^{-1}y_1^{-1}y_2^{-1}y_1^{-1}y_5y_3 =1,} \displaybreak[0]\\
& \smash{y_9y_4^{-1}y_1y_2y_1y_4^{-1}y_1^{-1}y_2^{-1}y_1^{-1}y_5 =
  y_9y_1y_2y_1y_5^{-1}y_1^{-1}y_2^{-1}y_1^{-1} = 1,} \displaybreak[0]\\
& \smash{y_9y_1^{-1}y_2^{-1}y_1^{-1}y_5^{-1}y_1y_2y_1 =
  y_{10}y_1^{-1}y_2^{-1}y_1^{-1}y_4^{-1}y_9^{-1}y_{10}^{-1}y_9^{-1}y_1
  y_2y_1 = 1,} \displaybreak[0]\\
& \smash{y_{11}y_1^{-1}y_2^{-1}y_1^{-1}y_{11}^{-1}y_9^{-1}y_5^{-1}y_1y_2 y_1
  = y_9y_1y_2y_1^2y_2y_1y_5 = y_9y_4y_1y_2y_1^2y_2^{-1}y_1^{-1} =1,} \displaybreak[0]\\
& \smash{y_{11}y_4^{-1}y_1y_3^{-1}y_1^{-1}y_9y_{11}y_3y_9 = [y_4^{-1}, y_1] =
  y_9y_1y_5^{-1}y_1^{-1} = y_9y_1^{-1}y_9^{-1}y_5^{-1}y_1 =1,} \displaybreak[0]\\
& \smash{y_{10}y_1^{-1}y_4y_{10}y_1 = y_{11}y_1^{-1}y_{11}^{-1}y_5^{-1}y_1
  =y_9y_1y_4^{-1}y_5^{-1}y_1y_5 =y_4y_1^{-2} =1,} \displaybreak[0]\\
& \smash{y_{11}y_2y_1y_6y_3y_{11}y_{10}y_7y_1^{-1}y_2^{-1} y_4^{-1}
  y_{10}^{-1}y_3^{-1}y_5^{-1}y_7^{-1}y_{10}^{-1}y_{11}^{-1}
  y_3^{-1}y_6^{-1}y_5y_9y_{11}y_{10}^2 =1,} \displaybreak[1]\\
& \smash{y_{11}y_{10}^{-1}y_9y_2y_1y_2y_1y_6y_3y_{11}y_{10}
  y_7y_1^{-1}y_2^{-1}y_1^{-1}y_2^{-1}y_4y_9^{-1}y_{10}y_4^{-1}}\\ 
&\smash{\quad \cdot
  y_{11}^{-1}y_5^{-1}y_7^{-1}y_{10}^{-1}y_{11}^{-1}y_3^{-1}y_6^{-1}
  =1,}\displaybreak[1]\\ 
& y_{11}y_{10}y_7y_4y_{10}^{-1}y_4y_5^{-1}y_7^{-1}y_{10}^{-1}
  y_{11}^{-1}y_3^{-1}y_6^{-1}y_5y_9y_{11}y_{10}^2y_6y_3 =1, \displaybreak[0]\\
&\smash{y_{11}y_4^{-1}y_{10}^{-1}y_5y_2y_6y_3y_{11}y_{10}y_7y_2^{-1}y_3^{-1}y_{10}
  y_3^{-1}y_7^{-1}y_{10}^{-1}y_{11}^{-1}y_3^{-1}y_6^{-1}
  y_9 =1,} \displaybreak[0]\\
& \smash{y_{11}y_3y_4y_9y_{10}y_4^{-1}y_6y_3y_{11}y_{10}y_7y_4^{-1}y_{10}^{-1}
  y_9^{-1}y_3^{-1}y_{11}^{-1}y_9^{-1}y_5^{-1}y_7^{-1}y_{10}^{-1}
  y_{11}^{-1}y_3^{-1}y_6^{-1} =1,}\displaybreak[0]\\ 
& \smash{y_{11}y_4^{-1}y_{10}y_9y_7^{-1}y_{10}^{-1}y_{11}^{-1}y_3^{-1}y_6^{-1}
  y_4^{-1}y_{10}^{-1}y_4^{-1}y_{11}^{-1}y_9^{-1}y_6y_3y_{11}y_{10}y_7
  =1,} \displaybreak[0]\\
& \smash{y_{11}y_7^{-1}y_{10}^{-1}y_{11}^{-1}y_3^{-1}y_6^{-1}y_{11}^{-1}
  y_9^{-1}y_6y_3y_{11}y_{10}y_7y_9 = 1,} \displaybreak[0]\\ 
& \smash{y_{11}y_{10}y_7y_{10}^{-1}y_6y_3y_{11}y_{10}
  y_7y_5y_9y_{10}^{-1}y_4^{-1}y_{10}y_6y_3 =1,} \displaybreak[0]\\
& \smash{y_{11}y_3y_2y_1y_2y_1y_6y_3y_{11}y_{10}y_7y_3 y_{10}y_6y_3
  y_{11}y_{10}y_7y_4y_7^{-1}y_{10}^{-1}y_{11}^{-1}
  y_3^{-1}y_6^{-1}y_5y_3y_4 =1,}\displaybreak[0]\\ 
& \smash{y_{11}y_{10}y_7y_4^{-1}y_2^{-1}y_7^{-1}y_{10}^{-1}y_{11}^{-1}
  y_3^{-1}y_6^{-1}y_5y_9y_{10}y_3y_{10}y_1y_2y_1^2y_6
  y_3y_{11}y_{10}y_7y_9y_6y_3 =1,} \displaybreak[0]\\ 
& \smash{y_{11}y_3y_{11}y_{10}y_6y_3y_{11}y_{10}y_7y_5y_6y_3y_{11}y_{10}
  y_7y_9^{-1}y_4^{-1}y_7^{-1}y_{10}^{-1}y_{11}^{-1}
  y_3^{-1}y_6^{-1}y_9 =1,} \displaybreak[0]\\
& \smash{y_7y_4^{-1}y_2^{-1}y_6^{-1}y_2^{-1}y_4y_5^{-1}
  =y_9y_6y_2^2y_8^{-1}y_5 =y_9y_6y_2y_5^{-1}y_2^{-1}y_6^{-1} = 1,} \displaybreak[0]\\
& \smash{y_9y_2^{-1}y_6^{-1}y_5^{-1}y_6y_2 =
  y_{11}y_2^{-1}y_6^{-1}y_{11}^{-1}y_9^{-1}y_5^{-1}y_6y_2 = 1,} \displaybreak[0]\\
& \smash{y_{11}y_{10}y_7y_3y_4y_2^{-1}y_6^{-1}y_3^{-1}y_7^{-1}y_{10}^{-1}
  y_{11}^{-1}y_3^{-1}y_6^{-1}y_1^{-1}y_2^{-1}y_1^{-1}y_2^{-1}y_4^{-1}
  y_{10}^{-1}y_3^{-1}}\\
&\smash{\quad \cdot y_9^{-1}y_6y_2y_5y_3y_4y_{10}^{-1}y_6 y_3 =1,} \displaybreak[1]\\
& y_{11}y_4y_{10}y_2y_1y_2y_1y_6y_3y_{11}y_{10}
  y_7y_2y_3y_2^{-1}y_6^{-1}y_{10}y_7^{-1}y_{10}^{-1}y_{11}^{-1}
  y_3^{-1}y_6^{-1}y_1^{-1}y_2^{-1}y_1^{-1}y_9^{-1}\\
&\smash{\quad \cdot y_{10}y_4^{-1}y_6 y_2y_5y_9 =1,}\displaybreak[1]\\
& y_{11}y_{10}y_7y_2^{-1}y_6^{-1}y_9^{-1}y_7^{-1}y_{10}^{-1}y_{11}^{-1}
  y_3^{-1}y_6^{-1}y_1^{-2}y_2^{-1}y_1^{-1}y_2^{-1}\\ 
&\smash{\quad \cdot y_9^{-1}y_{10}^{-1}
  y_{11}^{-1}y_3^{-1}y_9^{-1}y_5^{-1}y_6y_2y_5y_3y_{11}y_{10}
  y_9y_2y_1y_2y_1^2y_6y_3 =1,} \displaybreak[1]\\
& y_{11}y_{10}y_7y_4y_9y_6y_2y_1^{-1}y_2^{-1}y_1^{-1}y_2^{-1}
  y_{10}^{-1}y_5^{-1}y_{10}^{-1}y_2^{-1}y_6^{-1} y_2y_1y_2 y_1y_6y_3
  =1, \displaybreak[0]\\
& \smash{y_{11}y_3y_4y_{10}y_4y_2y_1y_2y_1y_6y_3y_{11}
  y_{10}y_7y_3y_4y_6y_2y_1^{-1}y_3^{-1}y_9^{-1}y_{10}}\\ 
&\smash{\quad \cdot y_9^{-1}y_5^{-1}y_2^{-1}y_6^{-1}y_5y_9 =1,} \displaybreak[1]\\
&y_{11}y_3y_{11}y_{10}y_6^{-1}y_{10}y_7^{-1}y_{10}^{-1}y_{11}^{-1}
  y_3^{-1}y_6^{-1}y_2^{-1}y_3^{-1}y_{10}^{-1}y_{11}^{-1}y_9^{-1}y_6y_2
  = 1, \displaybreak[0]\\
&\smash{y_{11}y_{10}y_2^{-1}y_6^{-1}y_{10}y_7^{-1}y_{10}^{-1}y_{11}^{-1}y_3^{-1}
  y_6^{-1}y_4^{-1}y_9^{-1}y_{10}^{-1}y_5^{-1}y_6y_2y_5y_3y_4^{-1} = 1,}
  \displaybreak[0]\\
&\smash{(y_6y_2^2)^2 =
  y_9y_4^{-1}y_2y_6y_2y_6y_2y_4^{-1}y_2^{-1}y_6^{-1} = 1,} \displaybreak[0]\\
&\smash{y_{11}y_3y_5y_2y_1y_2y_1^2y_6y_3y_{11}y_{10}
  y_7y_2y_6y_2y_4y_{10}^{-1}y_6y_2y_{10}y_2^{-1}y_6^{-1} y_9 =1,} \\
&\smash{y_{11}y_{10}y_7y_2y_{10}^{-1}y_6y_2^2y_{10}^{-1}y_6y_2y_5^{-1}
  y_3^{-1}y_2^{-1}y_6^{-1}y_5y_9y_4y_{10}y_9y_2y_6y_3  
  =1.} 
\end{alignat*}
\item[\rm(c)] $D$ has a faithful permutation representation $PD$
of degree $1024$ with stabilizer $U = \langle
y_1,y_4,y_{10}\rangle$.

\item[\rm(d)] $D$ has a unique normal 
nonabelian subgroup $Q$ of
order $1024$ with center $Z(Q) = \langle z_1, z_2\rangle$ of order
$4$, where $z_1 = y_9y_1y_2y_1y_2$, $z_2 = y_9y_{11}y_4y_1$.
Furthermore, $Q$ has a complement $W$ 
in $D$ of order $|W| = 2^7\cdot3\cdot5$, where $W=\langle
w_1,w_2,w_3\rangle$ and
\begin{eqnarray*}
w_1 &=& y_{11}y_{10}^2y_4^{-1}y_1y_6y_{10}^{-1}y_8y_{10}y_8y_2,\\
w_2 &=& y_{11}y_{10}^{-1}y_1y_6y_{10}y_4y_8y_{10}y_7,\ \mbox{and}\\
w_3 &=& y_{11}y_{10}y_3y_4y_1y_6y_{10}y_4y_8y_{10}y_7.
\end{eqnarray*}

\item[\rm(e)] Let $\alpha: D \rightarrow D_1 = D/Z(Q)$ be the
canonical epimorphism with kernel 
${\rm ker}(\alpha) = Z(Q)$. Let $V =
\alpha(Q)$ and let $u_i = \alpha(y_i) \in D_1$ for $i =
1, 2,\ldots,11$. Then $V$ is an elementary abelian normal subgroup of
order $2^8$ of $D_1$ having a complement $W_1 = \langle k_1, k_2,
k_3 \rangle \cong W$, where $k_j = \alpha(w_j)$ for $j = 1,2,3$.
Furthermore, $D_1$ has a faithful permutation representation of
degree $256$ with stabilizer $W_1$, and $V = \langle 
q_l\,|\, 1 \le l
\le 8\rangle$, where
\begin{alignat*}{1}
q_1 & =  u_5u_9u_4u_{10}^{-1}u_4^{-1}u_1u_6u_3u_1u_{10},\\
q_2 & =  u_3u_{11}u_4u_1u_7u_3, \qquad q_3  =  u_9u_7,
\qquad q_4 =  u_{11}u_1u_6u_4,\\
q_5 & =  u_9u_{11}u_3u_5u_1u_7u_{10}u_8u_3u_6,\\
q_6 & =  u_9u_{11}u_3u_{11}u_{10}u_1u_6u_3u_1^{-1}u_3u_7,\\
q_7 & =  u_5u_{11}u_3u_9u_{11}u_4^{-1}u_6u_3u_8u_{10}^{-1}u_7,\\
q_8 & =  u_9u_3u_4^{-1}u_1u_6u_3u_8u_{10}^{-1}u_7u_1.
\end{alignat*}

\item[\rm(f)] The conjugate action of the 
three generators $k_j$ of
$W_1$ on $V$ 
with respect to the basis $\mathcal B = 
\{q_l \, |\, 1 \le l \le 8\}$
of $V$ is given by the following matrices:
{\renewcommand{\arraystretch}{0.5}
\scriptsize
$$
\qquad\qquad
Mk_1 = \left( \begin{array}{*{8}{c@{\,}}c}
1 & 1 & 0 & 0 & 0 & 0 & 0 & 0 \\
0 & 1 & 0 & 0 & 0 & 0 & 0 & 0 \\
0 & 0 & 1 & 1 & 0 & 0 & 0 & 0 \\
0 & 0 & 0 & 1 & 0 & 0 & 0 & 0 \\
0 & 0 & 0 & 0 & 0 & 0 & 1 & 0 \\
0 & 1 & 0 & 0 & 1 & 1 & 1 & 1 \\
0 & 0 & 0 & 0 & 1 & 0 & 0 & 0 \\
0 & 0 & 0 & 0 & 0 & 0 & 0 & 1 
\end{array} 
\right),\quad 
Mk_2 = \left( \begin{array}{*{10}{c@{\,}}c}
1 & 0 & 1 & 0 & 0 & 0 & 0 & 0 \\
1 & 1 & 1 & 1 & 0 & 0 & 0 & 0 \\
1 & 1 & 0 & 1 & 0 & 0 & 0 & 0 \\
1 & 0 & 0 & 1 & 0 & 0 & 0 & 0 \\
0 & 1 & 0 & 0 & 0 & 1 & 0 & 0 \\
0 & 0 & 0 & 1 & 0 & 0 & 1 & 1 \\
0 & 1 & 0 & 0 & 1 & 0 & 1 & 1 \\
1 & 0 & 1 & 0 & 0 & 0 & 1 & 0 
\end{array} \right),
\quad\mbox{and}\quad
Mk_3 = \left( \begin{array}{*{8}{c@{\,}}c}
1 & 1 & 0 & 0 & 0 & 0 & 0 & 0 \\
1 & 0 & 0 & 0 & 0 & 0 & 0 & 0 \\
1 & 0 & 1 & 1 & 0 & 0 & 0 & 0 \\
1 & 1 & 1 & 0 & 0 & 0 & 0 & 0 \\
1 & 0 & 0 & 0 & 0 & 1 & 0 & 0 \\
0 & 1 & 0 & 0 & 1 & 1 & 0 & 0 \\
0 & 0 & 0 & 0 & 1 & 0 & 1 & 1 \\
1 & 0 & 0 & 0 & 0 & 1 & 1 & 0 
\end{array} \right).
$$}
\item[\rm(g)] The Fitting subgroup $A$ of $W_1$ is elementary
abelian of order $16$ and generated by $a_1 = k_1k_3k_2k_1k_2^3$,
$a_2 = k_2k_3^2k_2k_3^2$, $a_3 = k_2k_1k_2^3k_1k_3$, 
and $a_4 =
k_1k_3k_2k_3k_2k_1k_2$. It has a complement $L = \langle k_1, k_2
\rangle$ in $W_1$ which is isomorphic to the symmetric group
$S_5$.

\item[\rm(h)] 
The Fitting subgroup $A$ 
is a maximal
elementary abelian normal subgroup of the Sylow $2$-subgroup $S =
\langle A, k_1, r \rangle$ of $W_1$ with center $Z(S) = \langle u
\rangle$, where $r = k_2k_1k_3k_2^2k_1k_2k_3$ and $u =
k_1(k_3k_2)^2k_1k_2$. The centralizer $C_{W_1}(u)$ of $u$ has order
$2^7\cdot 3$. It is generated by $S$ and the element $d =
(k_1k_2)^2(k_2k_1)^2$ of order $3$.

\item[\rm(i)] Let $MW_1$ be the subgroup of $\GL_8(2)$ generated
by the matrices of the conjugate action of the generators of $W_1$
on $V$ with respect to the basis~$\mathcal B$. Let $Ma_i$, $Mu$
and $Md$ be the corresponding matrix of $a_i$, $u$ and $d$,
respectively. Let $MX = C_{\GL_8(2)}(Mu) \cap C_{\GL_8(2)}(Md)$.
Then $MX$ has an abelian Sylow $3$-subgroup $MT$ of order $3^3$
and $MT$ contains 
the matrix 
{\renewcommand{\arraystretch}{0.5}%
\scriptsize%
$$
Mx = \left( \begin{array}{*{8}{c@{\,}}c}
0 & 1 & 0 & 0 & 0 & 0 & 0 & 0 \\
1 & 1 & 0 & 0 & 0 & 0 & 0 & 0 \\
1 & 1 & 0 & 1 & 0 & 0 & 0 & 1 \\
1 & 0 & 1 & 1 & 0 & 1 & 1 & 1 \\
0 & 0 & 0 & 1 & 1 & 1 & 0 & 0 \\
0 & 1 & 1 & 0 & 1 & 0 & 0 & 0 \\
1 & 1 & 1 & 1 & 1 & 0 & 0 & 0 \\
0 & 1 & 1 & 1 & 0 & 0 & 0 & 0 \\
\end{array} \right)%
$$}%
of order $3$ such that $MC = \langle C_{MW_1}(Mu), Mx \rangle$ has
order $|MC| = 2^7\cdot 3^2$.

\item[\rm(j)] Let 
$M\!K = \langle MC, MW_1 \rangle$. Then the
derived subgroup 
$M\!K'$ of $M\!K$  is a simple group of order
$|M\!K'| = 2^6\cdot 3^4\cdot 5$, which is isomorphic to the
unitary group $U_4(2)$.

\item[\rm(k)] 
$M\!K$ is an irreducible subgroup of $\GL_8(2)$
generated by the matrices $m_1 = Mk_1$, $m_2 = Mk_2$, $m_3 =
Mk_3$, $m_4 = Mx$ of respective orders $2$, $5$, 
$3$, and~$3$. With
respect to this set of generators 
$M\!K$ has the following set
${\mathcal R}(K)$ of defining relations:
\begin{eqnarray*}
&&m_1^2 = m_2^5 = m_3^3 = m_4^3 = 1,\quad [m_3,m_4] = 1,\quad m_3^{-1}m_2m_1m_2^{-1}m_3m_1 = 1,\\
&&(m_2^{-1}m_1)^4 = 1,\quad m_3^{-1}m_2m_3^{-1}m_1m_2m_3m_1m_2^{-1} = (m_2m_3^{-1})^4 = 1,\\
&&[m_4,m_2^{-1},m_4] = 1,\quad m_1m_4^{-1}m_1m_4^{-1}m_1m_4m_1m_4 = 1,\\
&&m_2^{-1}m_3^{-1}m_1m_2^{-2}m_3m_4m_1m_4^{-1} = 1,\quad m_2m_3^{-1}m_2m_4m_2^2m_3^{-1}m_2^{-1}m_3m_4 = 1.\\
\end{eqnarray*}

\item[\rm(l)] Let $K$ be the finitely presented group constructed
in 
{\rm (k).} Then $V$ is an irreducible $8$-dimensional representation
of $K$ over 
$F = \GF(2)$ of second cohomological dimension
$\dim_{F}[H^2(K,V)] = 0$. 
Furthermore, $K$ has a faithful
permutation representation $PK$ of degree $640$ having a
stabilizer which is the Sylow $3$-subgroup $\langle (m_2m_4)^2,
(m_1m_3^2m_4)^2 \rangle $ of $K$.

\item[\rm(m)] Let $H_1 = \langle m_i, 
q_j \,|\, 1 \le i \le 4, 1 \le j
\le 8 \rangle$ be the split extension of $K$ by~$V$. Then $H_1$
has a set ${\mathcal R}(H_1)$ of defining relations consisting of
${\mathcal R}(K)$, ${\mathcal R_1}(V\rtimes K)$ and the following
set ${\mathcal R_2}(V\rtimes K)$ of essential relations:
\begin{alignat*}{1}
&\smash{m_1q_1m_1^{-1}q_1q_2  = 1,\quad m_1q_2m_1^{-1}q_2  = 1,\quad m_1q_3m_1^{-1}q_3q_4 = 1,}\\
&\smash{m_1q_4m_1^{-1}q_4 = 1,\quad m_1q_5m_1^{-1}q_7 = 1,\quad
  m_1q_6m_1^{-1}q_2q_5q_6q_7q_8 = 1,} \displaybreak[0]\\
&\smash{m_1q_7m_1^{-1}q_5 = 1,\quad m_1q_8m_1^{-1}q_8 = 1,\quad
  m_2q_1m_2^{-1}q_1q_2q_3 = 1,}\displaybreak[0]\\ 
&\smash{m_2q_2m_2^{-1}q_3q_4 = 1,\quad m_2q_3m_2^{-1}q_2q_3 = 1,\quad
  m_2q_4m_2^{-1}q_1q_2q_3q_4 = 1,}\displaybreak[0]\\ 
&\smash{m_2q_5m_2^{-1}q_1q_2q_6q_7 = 1,\quad m_2q_6m_2^{-1}q_3q_4q_5 =
  1,\quad m_2q_7m_2^{-1}q_1q_8 = 1,} \displaybreak[0]\\
&\smash{m_2q_8m_2^{-1}q_2q_3q_4q_6q_8 = 1,\quad m_3q_1m_3^{-1}q_2 = 1,\quad
  m_3q_2m_3^{-1}q_1q_2 = 1,} \displaybreak[0]\\
&\smash{m_3q_3m_3^{-1}q_1q_4 = 1,\quad m_3q_4m_3^{-1}q_1q_2q_3q_4 = 1,\quad
  m_3q_5m_3^{-1}q_1q_5q_6 = 1,} \displaybreak[0]\\
&\smash{m_3q_6m_3^{-1}q_2q_5 = 1,\quad m_3q_7m_3^{-1}q_5q_8 = 1,\quad
  m_3q_8m_3^{-1}q_1q_6q_7q_8 = 1,} \displaybreak[0]\\
&\smash{m_4q_1m_4^{-1}q_1q_2 = 1,\quad m_4q_2m_4^{-1}q_1 = 1,\quad
  m_4q_3m_4^{-1}q_2q_6q_7q_8 = 1,} \displaybreak[0]\\
&\smash{m_4q_4m_4^{-1}q_1q_2q_6q_7 = 1,\quad
    m_4q_5m_4^{-1}q_1q_2q_7q_8 = 1,\quad m_4q_6m_4^{-1}q_5q_6q_8 = 1,}
  \\
&\smash{m_4q_7m_4^{-1}q_1q_2q_3q_4q_5q_7 = 1,\quad m_4q_8m_4^{-1}q_1q_3q_6q_7 = 1.}
\end{alignat*}
\item[\rm(n)] 
$\dim_F[H^2(H_1,F)] = 4$ 
and there exists a unique
central extension $H$ of $H_1$ of order $|H| = 2^{17}\cdot
3^4\cdot 5$ whose Sylow $2$-subgroups are isomorphic to a Sylow
$2$-subgroup of $D$. $H$ is a split extension of $K$ by its
Fitting 
subgroup~$Q$, and $H$ is isomorphic to the finitely
presented group $H = \langle 
h_i\,|\, \le i \le 14\rangle$ with the
following set ${\mathcal R}(H)$ of defining relations:
\begin{alignat*}{1}
& \smash{h_1^2 = h_2^5 = h_3^3 = h_4^3 = h_5^2 = h_6^2 = h_7^2 = h_8^2 =
   h_9^2 = h_{11}^2 = h_{12}^2 = h_{13}^2 = h_{14}^2 = 1,} \\
& \smash{h_{10}^2h_{14}^{-1} = 1, \quad [h_1, h_{14}^{-1}] = [h_2,
   h_{14}^{-1}] = [h_3, h_{14}^{-1}] = [h_4, h_{14}^{-1}] = [h_5,
   h_{14}^{-1}] = 1,} \displaybreak[0]\\
& \smash{[h_6, h_{14}^{-1}] = [h_7, h_{14}^{-1}] = [h_8, h_{14}^{-1}] =
   [h_9, h_{14}^{-1}] = [h_{10}, h_{14}^{-1}] = [h_{11},
   h_{14}^{-1}]=1,} \displaybreak[0]\\
& \smash{[h_{12},h_{14}^{-1}] = [h_{13}, h_{14}^{-1}] =
   h_1^{-1}h_{13}h_1h_{13}^{-1}h_{14}^{-1} = 1,} \displaybreak[0]\\
& \smash{[h_2, h_{13}^{-1}] = [h_3, h_{13}^{-1}] = [h_4, h_{13}^{-1}] = [h_5,
   h_{13}^{-1}] = [h_6, h_{13}^{-1}] =  [h_7, h_{13}^{-1}] = 1,}  \displaybreak[0]\\
& \smash{[h_8, h_{13}^{-1}] = [h_9, h_{13}^{-1}] = [h_{10}, h_{13}^{-1}] =
   [h_{11}, h_{13}^{-1}] = [h_{12}, h_{13}^{-1}] = [h_3, h_4] =1,}
   \displaybreak[0]\\ 
& \smash{h_3^{-1}h_2 h_1  h_2^{-1}  h_3  h_1 = (h_2^{-1}  h_1)^4 =
   h_3^{-1}h_2h_3^{-1}h_1h_2h_3h_1h_2^{-1} = (h_2h_3^{-1})^4 =1,}
   \displaybreak[0]\\ 
& \smash{[h_4, h_2^{-1}, h_4] = h_1h_4^{-1}h_1h_4^{-1}h_1h_4h_1h_4 =
   h_2^{-1}h_3^{-1}h_1h_2^{-2}h_3h_4h_1h_4^{-1} =1,} \displaybreak[0]\\
& \smash{h_2h_3^{-1}h_2h_4h_2^2h_3^{-1}h_2^{-1}h_3h_4 =
   h_1h_5h_1^{-1}h_5h_6h_{13}^{-1} = h_1h_6h_1^{-1}h_6h_{14}^{-1} =1,}
   \displaybreak[0]\\ 
& \smash{h_1h_7h_1^{-1}h_7  h_8  h_{14}^{-1} = h_1h_8h_1^{-1}h_8 =
   h_1h_9h_1^{-1}h_{11}h_{13}^{-1}h_{14}^{-1} =1,} \displaybreak[0]\\
& \smash{h_1h_{10}h_1^{-1}h_6h_9h_{10}h_{11}h_{12}h_{13}^{-1} =
   h_1h_{11}h_1^{-1}h_9  h_{13}^{-1} = 1,} \displaybreak[0]\\
& \smash{h_1h_{12}h_1^{-1}h_{12}h_{14}^{-1} =  h_2  h_5h_2^{-1}h_5h_6h_7
   h_{14}^{-1} =1,} \displaybreak[0]\\
& \smash{h_2  h_6h_2^{-1}h_7h_8h_{13}^{-1}h_{14}^{-1} = h_2
   h_7h_2^{-1}h_6h_7h_{14}^{-1} = h_2  h_8h_2^{-1}h_5h_6h_7  h_8 =1,}
   \displaybreak[0]\\ 
& \smash{h_2  h_9h_2^{-1}h_5h_6h_{10}h_{11}h_{13}^{-1}h_{14}^{-1} = h_2
   h_{10}  h_2^{-1}  h_7  h_8  h_9  h_{13}^{-1} =1,} \displaybreak[0]\\
& \smash{h_2  h_{11}  h_2^{-1}  h_5  h_{12}  h_{14}^{-1} = h_2  h_{12}
   h_2^{-1}  h_6  h_7  h_8  h_{10}h_{12} =1,} \displaybreak[0]\\
& \smash{h_3  h_5  h_3^{-1}  h_6  h_{14}^{-1} = h_3  h_6  h_3^{-1}  h_5  h_6
   h_{13}^{-1} = h_3  h_7  h_3^{-1}  h_5  h_8 =1,} \displaybreak[0]\\
& \smash{h_3  h_8  h_3^{-1}  h_5  h_6  h_7  h_8 = h_3  h_9  h_3^{-1}  h_5
   h_9  h_{10}  h_{14}^{-1} = h_3  h_{10}  h_3^{-1}  h_6  h_9 =1,} \displaybreak[0]\\
& \smash{h_3  h_{11}  h_3^{-1}  h_9  h_{12}  h_{14}^{-1} = h_3  h_{12}
   h_3^{-1}  h_5  h_{10}  h_{11}  h_{12}h_{13}^{-1}  h_{14}^{-1} =1,} \displaybreak[0]\\
& \smash{h_4  h_5  h_4^{-1}  h_5  h_6  h_{13}^{-1}h_{14}^{-1} = h_4  h_6
   h_4^{-1}  h_5  h_{14}^{-1} =1,} \displaybreak[0]\\
& \smash{h_4  h_7  h_4^{-1}  h_6  h_{10}  h_{11}  h_{12}h_{14}^{-1} = h_4
   h_8  h_4^{-1}  h_5  h_6  h_{10}  h_{11} h_{13}^{-1}  h_{14}^{-1}
   =1,} \displaybreak[0]\\
& \smash{h_4  h_9  h_4^{-1}  h_5  h_6  h_{11}  h_{12}h_{14}^{-1} = h_4
   h_{10}  h_4^{-1}  h_9  h_{10}  h_{12}h_{13}^{-1}h_{14}^{-1} =1,}
   \displaybreak[0]\\ 
& \smash{h_4  h_{11}h_4^{-1}h_5h_6h_7h_8h_9h_{11}h_{13}^{-1} = h_4
   h_{12}h_4^{-1}h_5h_7h_{10}h_{11}h_{14}^{-1} =1,} \displaybreak[0]\\
& \smash{[h_5, h_6] = [h_5, h_7] = [h_5, h_8] = [h_5, h_9] =
   h_5^{-1}h_{10}^{-1}h_5  h_{10}h_{14}^{-1} =1,} \displaybreak[0]\\
& \smash{h_5^{-1}h_{11}^{-1}h_5  h_{11}h_{14}^{-1} = [h_5, h_{12}] = [h_6,
   h_7] = [h_6, h_8] = h_6^{-1}  h_9^{-1}  h_6  h_9  h_{14}^{-1} =1,}
   \displaybreak[0]\\ 
&  \smash{h_6^{-1}h_{10}^{-1}h_6h_{10}h_{14}^{-1} =
   h_6^{-1}h_{11}^{-1}h_6h_{11}h_{14}^{-1} = [h_6, h_{12}] = [h_7,
   h_8] = [h_7, h_9] =1,} \displaybreak[0]\\
&  \smash{h_7^{-1}h_{10}^{-1}h_7h_{10}h_{14}^{-1} =
   h_7^{-1}h_{11}^{-1}h_7h_{11}h_{14}^{-1} =
   h_7^{-1}h_{12}^{-1}h_7h_{12}h_{14}^{-1} =1,} \displaybreak[0]\\
&  \smash{h_8^{-1}h_9^{-1}h_8h_9h_{14}^{-1} = [h_8, h_{10}] =
   h_8^{-1}h_{11}^{-1}h_8h_{11}h_{14}^{-1} = [h_8, h_{12}] = [h_9,
   h_{10}] =1,} \\
& \smash{[h_9, h_{11}] = [h_9, h_{12}] =
   h_{10}^{-1}h_{11}^{-1}h_{10}h_{11}h_{14}^{-1} = [h_{10}, h_{12}] =
   [h_{11}, h_{12}] = 1.} 
\end{alignat*}
\end{enumerate}
\end{proposition}

\begin{proof}
(a) By Table \ref{Fi_22cc E} the split extension $E_2 = V_2\rtimes
\M_{22}$ has a unique conjugacy class of involutions of highest
defect. It is represented by $z = (iv_1)^2$ and its centralizer $D
= C_{E_2}(z)$ has order $2^{17}\cdot 3\cdot 5$.

(b) The presentation of $D$ given in the statement was obtained using the faithful
permutation representation of $E_2$ with 
stabilizer $\M_{22}$ and the \textsc{Magma} command
$\verb"FPGroupStrong(PD)"$. 

(c) Another application of \textsc{Magma} yields that the finitely
presented group given in (b) has a faithful permutation
representation $PD$ of degree $1024$ with stabilizer $X = \langle
y_1,y_4,y_{10} \rangle$.

(d) Using this permutation representation $PD$ of $D$ and the
\textsc{Magma} command $\verb"NormalSubgroups(PD)"$ it follows that $D$ has
a unique 
nonabelian normal subgroup $Q$ of order $2^{10}$ with
center $Z(Q) = \langle z_1, z_2\rangle$ of order $4$. The words of
$z_1, z_2$ in the generators $y_i$ of $D$ given in the statement
have been obtained computationally 
as well. Another application of
\textsc{Magma} determined the generators of the stated complement $W =
\langle w_1,w_2,w_3\rangle$ of $Q$ in $D$.

(e) Since $Q$ is a characteristic subgroup of $D$ its center
$Z(Q)$ is normal in $D$. Let $\alpha: D \rightarrow D_1 = D/Z(Q)$
be the canonical epimorphism with 
${\rm ker}(\alpha)=Z(Q)$. Let
$V = \alpha(Q)$ and let $u_i = \alpha(y_i) \in D_1$ for $i =
1, 2, \ldots, 11$. Then $W_1 = \alpha(W) = \langle k_1, k_2, k_3\rangle$
is a complement of $V$, where $k_j = \alpha(w_j)$. Furthermore,
$D_1$ has a faithful permutation representation $PD_1$ of degree
$256$ with stabilizer $W_1$. Another application of \textsc{Magma} now
yields that the $V$ is elementary abelian and generated by the 
eight elements $q_l$ given in the statement.

(f) This assertion has been verified computationally.

(g) 
The faithful permutation representation
$PD_1$ and \textsc{Magma}
make it straightforward to see that the Fitting subgroup $A =
\langle a_1,a_2,a_3,a_4\rangle$ of $W_1$ is elementary abelian of
order $16$, and that $L = \langle k_1, k_2\rangle$ is a complement
of $A$ in $W_1$. Also the generators $a_i$ of $A$ and the
isomorphism $L \cong S_5$ have been obtained by this application
of \textsc{Magma}.

(h) Another calculation with \textsc{Magma} in the permutation group $PD_1$
shows that $S = \langle A, k_1, r\rangle$ is a Sylow $2$-subgroup
of $W_1$ of order $2^7$, where $r = k_2k_1k_3k_2^2k_1k_2k_3$.
Furthermore, $A$ is the unique maximal elementary abelian normal
subgroup of $S$ and the involution $u = k_1(k_3k_2)^2k_1k_2$
generates the center $Z(S)$ of $S$. Its centralizer $C_{W_1}(u) =
\langle S, d\rangle$ has order $2^7\cdot 3$, and $d =
(k_1k_2)^2(k_2k_1)^2$ generates a cyclic Sylow $3$-subgroup of
$C_{W_1}(u)$.

(i) In order to apply Algorithm 2.5 of \cite{michler1}
to construct a larger centralizer of~$u$ we use the
anti-isomorphism between $PW_1$ and its image in 
$\GL_8(2)$ induced
by the conjugate action of $W_1$ on $V$. In particular, the matrix
$d$ 
with respect to the basis $\mathcal B$ is $Md =
(Mk_1Mk_2)^2(Mk_2Mk_1)^2$ and $Mu = Mk_2Mk_1(Mk_2Mk_3)^2Mk_1$. Let
$MW_1$ be the subgroup of $\GL_8(2)$ generated by the matrices
$Mk_1$, $Mk_2$, and $Mk_3$ in $\GL_8(2)$.

Another application of \textsc{Magma} in 
$\GL_8(2)$ shows that 
\[MX = C_{\GL_8(2)}(Mu) \cap C_{\GL_8(2)}(Md)\] 
has an abelian Sylow
$3$-subgroup $MT$ of order $3^3$ and $|MX| = 2^{12}\cdot 3^3\cdot
5$. Furthermore, it follows that the matrix $Mx$ of the statement
belongs to the set of all matrices $My \in MX$ of order $3$ such
that $|\langle C_{MW_1}(Mu), My \rangle| = 2^7\cdot 3^2$. In
particular, $Mu$ is the unique central involution of $MC = \langle
C_{MW_1}(Mu), Mx \rangle$ and $|MC| = 2^7\cdot 3^2$.

(j) This statement has been verified by means of the previous
results and \textsc{Magma}.

(k) Let 
$M\!K$ be the subgroup of $\GL_8(2)$ generated by the
matrices $m_1 = Mk_1$, $m_2 = Mk_2$, $m_3 = Mk_3$, 
and~$m_4=Mx$. 
Taking a Sylow $3$-subgroup of 
$M\!K$ as a stabilizer one obtains a
faithful permutation representation $PK$ of this matrix group
having degree $640$. Then the presentation of the finitely
generated group $K = \langle m_1,m_2,m_3,m_4 \rangle$ has been
calculated by means of the \textsc{Magma} command $\verb"FPGroup(PK)"$.

(l) All assertions of this statement follow easily from (j), 
(k), and Holt's Algorithm 7.4.5 of \cite{michler} implemented in \textsc{Magma}.

(m) The presentation of the split extension $H_1 = V\rtimes K$ has
been calculated by means of Lemma \ref{l. presentation-split}
using the presentation of $K$ given in (k).

(n) 
Part~(e) states that $D_1 = D/Z(D)$ splits over $V =
Q/Z(D)$ with complement $W_1$. By Algorithm 2.5 of
\cite{michler1} the 
desired centralizer $H = C_G(z)$ of $z = z_1$
in the 
(at this time unknown) target simple group $G$ has to have
an odd index $|H : D| = |K : W|$. Therefore Theorem 1.4.15 of
\cite{michler} implies that $H$ is a central extension of $H_1$ by
a central subgroup $Z(H) = \langle z_1, z_2 \rangle$ having a
normal complement $Q$ containing $Z(H)$ such that $V = Q/Z(H)$ has
a complement isomorphic to $K$.

Since $Z(H) = Z(D) = \langle z_1, z_2\rangle$ is a Klein four
group, the central extension H has to be constructed in 
two steps.

Clearly, $H_1 = \langle
m_1,m_2,m_3,m_4,v_1,v_2,v_3,v_4,v_5,v_6,v_7,v_8 \rangle$ has a
faithful permutation $PH_1$ of degree $256$ with stabilizer $K$.
Let $FPH_1$ be the presentation of $H_1$ given in (m). Then we
apply Holt's Algorithm 7.4.5 of \cite{michler} implemented in
\textsc{Magma} \cite{holt} to the trivial matrix representation of $H_1$
over 
$F = \GF(2)$. It yields that the second cohomological
dimension 
$\dim_F[H^2(H_1,F)] = 4$. 
Thus the first central
extension $H_2$ of $H_1$ by a central involution $y \in Z(D)$ is
one of the 
sixteen central extensions $E_{a,b,c,d}$, $0 \le a,b,c,d
\le 1$. As $D$ does not split over $Z(D)$ the group $H_2$ can only
be isomorphic to a 
non-split extension. Let $nmodQ : =
\verb"GModule(PH_1, FEalg)"$ be the trivial module of the matrix
algebra $\verb"FEalg"$ generated by the 
twelve identity matrices
corresponding to the 
twelve generators of $H_1$. Using the \textsc{Magma}
command  $\verb"P_H :=ExtensionProcess(PH_1,nmodQ,FPH_1)"$ we
obtain a presentation for each of the 
fifteen non-split central
extensions $E_{a,b,c,d}$, where $(a,b,c,d) \neq (0,0,0,0)$.
Another application of Theorem 1.4.15 of \cite{michler} implies
that $H_2$ has a normal complement $Q_2$ containing $Z_2 = \langle
y \rangle$ such that $Q_1 = Q_2/Z_2$. In particular, $H_2$ has a
faithful permutation representation of degree $512$, and its Sylow
$2$-subgroups are isomorphic to the ones of $D/Z_2$. Constructing
the corresponding permutation representations of the groups
$E_{a,b,c,d}$ it follows that only the 
three groups $E_{0,1,0,0}$,
$E_{1,0,0,0}$ and $E_{1,1,0,0}$ have a faithful permutation
representation of degree $512$ whose stabilizer is a subgroup
isomorphic to $L$. As $y \in \{z_1, z_2, z_3 = z_1z_2\}$ another
application of \textsc{Magma} yields that only $E_{1,0,0,0}$ has a Sylow
$2$-subgroup which is isomorphic to 
those of exactly one factor
group $D/\langle z_k \rangle$, where $1 \le k \le 3$. In fact,
this $k = 1$.

By construction the extension group $H_2 = E_{1,0,0,0}$ has a
complement and therefore a faithful permutation representation of
degree $512$. Its central involution is the new 
generator. Applying Holt's Algorithm 7.4.5 of \cite{michler}
implemented in \textsc{Magma} \cite{holt} again to the trivial matrix
representation of $H_2$ over 
$F = \GF(2)$ it follows that the second
cohomological dimension 
$\dim_F[H^2(H_1,F)] = 4$. After constructing
the corresponding suitable permutation representations of the 
fifteen non-split extension groups $F_{a,b,c,d}$ we see that only the 
three groups $F_{0,0,1,0}$, 
$F_{1,0,0,0}$, and $F_{1,0,1,0}$ have a faithful
permutation representation of degree $1024$ whose stabilizer is a
subgroup isomorphic to $K$. Now the isomorphism check of
\textsc{Magma} yields that only $F_{1,0,1,0}$ has a Sylow $2$-subgroup
which is isomorphic to a Sylow $2$-subgroup of $D$.  Hence $H =
C_G(z):= F_{1,0,1,0}$. Its presentation is given in the
statement. This completes the proof.
\end{proof}

\begin{lemma}\label{l. H(Fi_22)} Keep the notation of 
Lemma \ref{l. M22-extensions} 
and Proposition \ref{prop. D(Fi_22)}. Let $H =
\langle 
h_i \,|\, 1 \le i \le 14 \rangle$ be the finitely presented
group constructed in Proposition \ref{prop. D(Fi_22)}. Then the
following statements hold:
\begin{enumerate}
\item[\rm(a)] $H$ has a faithful permutation representation of
degree $1024$ with stabilizer $\langle h_1,h_2,h_3,h_4 \rangle$.

\item[\rm(b)] Each Sylow $2$-subgroup $S$ of $H$ has a unique
maximal elementary abelian normal subgroup $A$ of order $2^{10}$
and $N_H(A) \cong D = C_{E_2}(z)$.

\item[\rm(c)] There is a Sylow $2$-subgroup $S$ such that $D =
N_{H}(A) = \langle x, y \rangle$, where
\begin{eqnarray*}
x &=& \smash{\left( h_2 h_4 h_{12} h_4 (h_2 h_4)^4 \right)}^2, \\
\qquad \qquad y &=& \smash{\left((h_2 h_4^2)^2 h_{12}
                       h_4^2h_2h_1h_4^2h_2h_4h_1h_4h_2\right)}^2 
                    (h_2h_4^2h_2)^3h_4^2h_{12}h_4^2h_2h_1h_4h_{12}h_4
\end{eqnarray*}
have orders $2$ and~$6$, respectively. Furthermore, $H = \langle x,
y, h \rangle$, where $h = h_4$ has order $3$.

\item[\rm(d)] The amalgam $H \leftarrow D \rightarrow E_2$ has
Goldschmidt index $1$.

\item[\rm(e)] A system of representatives $r_i$ of the $115$
conjugacy classes of $H$ and the corresponding centralizers orders
$|C_H(r_i)|$ are given in Table \ref{Fi_22cc H}.

\item[\rm(f)] A system of representatives $d_i$ of the $97$
conjugacy classes of $D$ and the corresponding centralizers orders
$|C_D(d_i)|$ are given in Table \ref{Fi_22cc D}.

\item[\rm(g)] Let $\sigma:N_H(A) \rightarrow D = C_{E_2}$ be the
isomorphism given in 
{\rm (b).} Then there is an element $e \in E_2$ of
order $3$ such that $E_2 = \langle \sigma(D), e\rangle$. A system
of representatives $e_i$ of the $43$ conjugacy classes of $E_2$
and the corresponding centralizers orders $|C_{E_2}(e_i)|$ are
given in Table \ref{Fi_22cc E}.

\item[\rm(h)] The character tables of $H$, 
$D$, and $E_2$ are given
in 
Tables~{\rm \ref{Fi_22ct_H}}, {\rm \ref{Fi_22ct_D}}, and\/~{\rm \ref{Fi_22ct_E},}
respectively.
\end{enumerate}
\end{lemma}

\begin{proof}
(a) This assertion follows at once from Proposition \ref{prop.
D(Fi_22)}(n).

In particular, $H$ has a faithful permutation representation $PH$
of degree $1024$. Using it and \textsc{Magma} it is straightforward to
verify statements (b) and (c). The words for the generators $x$
and $y$ of $D$ in the generators $h_i$ of $H$ have been found by a
stand alone program 
written by the first author.

(d) The Goldschmidt index has been calculated by means of
Kratzer's Algorithm 7.1.10 of \cite{michler}.

The systems of representatives of the conjugacy classes of $H$,
$D$, and $E_2$ have been calculated by means of $PH$, \textsc{Magma} and
Kratzer's Algorithm 5.3.18 of \cite{michler}.

(g) It has been checked with \textsc{Magma} that $e=t$ satisfies
$E_3 = \langle\sigma(D),e\rangle$. 

The character tables of (h) have been obtained by using $PH$ and
\textsc{Magma}.
\end{proof}

\section{Construction of Conway's simple group $\Co_2$}

By Lemma \ref{l. H(Co_2)} the amalgam $H \leftarrow D \rightarrow
E_3$ constructed in 
Sections $2$ and~$3$ satisfies the conditions
of Step 5 of Algorithm 2.5 of  \cite{michler1}. Therefore we can
apply Algorithm 7.4.8 of \cite{michler} to give here a new
existence proof for Conway's sporadic group $\Co_2$, see
\cite{conway2}.

The set of all faithful characters of the finite group $U$ is
denoted by 
$f\mathrm{char}_{\mathbb{C}}(U)$, and 
$m\!f\mathrm{char}_{\mathbb{C}}(U)$ 
denotes the set of all multiplicity-free faithful
characters of $U$.

\begin{definition}\label{def. ComP}
Let $U_1$, $U_2$ be a pair of finite groups intersecting in $D$.
Then
\[
\Sigma = \{(\nu, \omega) \in 
m\!f\mathrm{char}_{\mathbb{C}}(U_1) \times
m\!f\mathrm{char}_{\mathbb{C}}(U_2) 
\mid \nu_{|D} =
\omega_{|D}\}\]
is called the set of {\em compatible pairs of multiplicity-free
faithful characters} of $U_1$ and~$U_2$.
For each $(\nu, \omega) \in \Sigma$ the integer $n = \nu(1) =
\omega(1)$ is called the {\em degree} of the compatible pair
$(\nu, \omega)$.
\end{definition}

\begin{theorem}\label{thm. existenceCo_2}
Keep the notation of Lemma \ref{l. H(Co_2)} and Proposition
\ref{prop. D(Co_2)}. Using the notation of the 
three character
tables \ref{Co_2ct_H}, 
\ref{Co_2ct_D}, and \ref{Co_2ct_E} of the
groups $H$, 
$D$, and $E_3$, respectively, the following statements
hold:
\begin{enumerate}
\item[\rm(a)] The smallest degree of a 
nontrivial pair
\[ (\chi,\tau)\in 
m\!f\mathrm{char}_{\mathbb{C}}(H)\times 
m\!f\mathrm{char}_{\mathbb{C}}(E_3)\]
of compatible characters is $23$.

\item[\rm(b)] There is exactly one compatible pair $(\chi, \tau)
\in 
m\!f\mathrm{char}_{\mathbb{C}}(H) \times m\!f
\mathrm{char}_{\mathbb{C}}(E_3)$ 
of degree $23$ of the groups
$H=\langle D, h \rangle$ and $E_3 =\langle D, e \rangle$:
$
(\chi_{2}+ \chi_{\bf 4}, \tau_{2} + \tau_{\bf 6})
$
with common restriction
$ 
\tau_{|D} = \chi_{|D} = \psi_{2}+\psi_{8} +\psi_{\bf 26},
$
where irreducible characters with 
boldface indices denote
faithful irreducible characters.

\item[\rm(c)]Let $\mathfrak V$ and $\mathfrak W$ be the 
uniquely determined (up to isomorphism) faithful 
semisimple 
multiplicity-free $23$-dimensional modules of $H$ and $E_3$ over
$F = \GF(13)$ corresponding to the compatible pair $\chi, \tau $,
respectively.
Let $\kappa_\mathfrak V : H \rightarrow \GL_{23}(13)$ and
$\kappa_\mathfrak W : E_3 \rightarrow \GL_{23}(13)$ be the
representations of $H$ and $E_3$ afforded by the modules
$\mathfrak V$ and $\mathfrak W$, respectively.
Let $\mathfrak h = \kappa_\mathfrak V(h)$, $\mathfrak x =
\kappa_\mathfrak V(x)$, $\mathfrak y = \kappa_\mathfrak V(y)$ in $
\kappa_\mathfrak V(H) \le \GL_{23}(13)$. Then the following
assertions hold:
\begin{enumerate}
\item[\rm(1)] $\mathfrak V_{|D} \cong \mathfrak W_{|D}$, and there
is a transformation matrix $\mathcal T \in \GL_{23}(13)$ such that
\[ \mathfrak x = \mathcal T^{-1} \kappa_\mathfrak W (x_1) \mathcal T,
\mathfrak y = \mathcal T^{-1} \kappa_\mathfrak W(y_1) \mathcal T.\]
Let $\mathfrak e = \mathcal T^{-1} \kappa_\mathfrak W(e) \mathcal
T \in \GL_{23}(13)$.

\item[\rm(2)] In $\mathfrak G_3 = \langle \mathfrak h, \mathfrak
x,\mathfrak y, \mathfrak e \rangle$ the subgroup $\mathfrak E_3 =
\langle \mathfrak x, \mathfrak y, \mathfrak e \rangle$ is the
stabilizer of a $1$-dimensional subspace ${\mathfrak U}$ of
$\mathfrak V$ such that the $\mathfrak G_3$-orbit ${\mathfrak
U}^{{\mathfrak G}}$ has degree~$46575$.

\item[\rm(3)] The generating matrices of $\mathfrak G_3$ are:
{\renewcommand{\arraystretch}{0.5}\scriptsize
\begin{alignat*}{1}
\mathfrak h &= \left( \begin{array}{*{23}{c@{\,}}c}
 0 & 12  & 1 & 11  & 3  & 1 & 11  & 0  & 0  & 0  & 0  & 0  & 0  & 0  & 0  & 0  & 0  & 0  & 0  & 0  & 0  & 0  & 0\\
 11  & 1 & 12 & 11  & 1 & 11  & 4  & 0  & 0  & 0  & 0  & 0  & 0  & 0  & 0  & 0  & 0  & 0  & 0  & 0  & 0  & 0  & 0\\
 7 & 10 & 12  & 6 & 11  & 2  & 3  & 0  & 0  & 0  & 0  & 0  & 0  & 0  & 0  & 0  & 0  & 0  & 0  & 0  & 0  & 0  & 0\\
 8  & 3  & 3  & 7 & 10  & 6  & 0  & 0  & 0  & 0  & 0  & 0  & 0  & 0  & 0  & 0  & 0  & 0  & 0  & 0  & 0  & 0  & 0\\
 0  & 2  & 9  & 9 & 12  & 4  & 5  & 0  & 0  & 0  & 0  & 0  & 0  & 0  & 0  & 0  & 0  & 0  & 0  & 0  & 0  & 0  & 0\\
 7  & 7  & 2 & 12 & 11 & 10  & 6  & 0  & 0  & 0  & 0  & 0  & 0  & 0  & 0  & 0  & 0  & 0  & 0  & 0  & 0  & 0  & 0\\
 0  & 1 & 11 & 11  & 6  & 9 & 10  & 0  & 0  & 0  & 0  & 0  & 0  & 0  & 0  & 0  & 0  & 0  & 0  & 0  & 0  & 0  & 0\\
 0  & 0  & 0  & 0  & 0  & 0  & 0  & 0 & 12  & 2 & 12 & 12 & 12  & 1  & 1  & 2 & 12  & 0  & 1  & 1  & 0 & 11  & 0\\
 0  & 0  & 0  & 0  & 0  & 0  & 0  & 0 & 11  & 2 & 12  & 0 & 12  & 1  & 1  & 1  & 0  & 1  & 1  & 1 & 12 & 11  & 0\\
 0  & 0  & 0  & 0  & 0  & 0  & 0  & 1  & 0 & 12  & 1  & 2  & 1  & 0  & 1 & 11  & 1  & 1 & 11 & 12 & 11  & 3  & 0\\
 0  & 0  & 0  & 0  & 0  & 0  & 0  & 0 & 12  & 0  & 0  & 0  & 0  & 0  & 1  & 1  & 0  & 1  & 0  & 0 & 12  & 0  & 0\\
 0  & 0  & 0  & 0  & 0  & 0  & 0  & 1 & 11  & 0  & 0  & 2  & 1  & 0  & 1  & 0  & 1  & 2 & 12  & 0 & 11  & 1 & 12\\
 0  & 0  & 0  & 0  & 0  & 0  & 0  & 0  & 1 & 12  & 1  & 0  & 0  & 0  & 0 & 12  & 0  & 0 & 12  & 0 & 12  & 1  & 0\\
 0  & 0  & 0  & 0  & 0  & 0  & 0  & 0  & 0  & 0  & 0 & 11 & 12  & 0  & 0  & 1 & 12 & 12  & 1  & 0  & 1 & 12  & 0\\
 0  & 0  & 0  & 0  & 0  & 0  & 0  & 0 & 11  & 2 & 12  & 0  & 0  & 0  & 0  & 1  & 0  & 1  & 1  & 1  & 0 & 11  & 0\\
 0  & 0  & 0  & 0  & 0  & 0  & 0  & 0 & 12  & 1  & 0  & 0 & 12  & 1  & 1  & 1  & 0  & 0  & 0  & 1 & 12 & 12  & 0\\
 0  & 0  & 0  & 0  & 0  & 0  & 0  & 0 & 12  & 1  & 0  & 1  & 0  & 0  & 0  & 0  & 1  & 1 & 12  & 1 & 12  & 0 & 12\\
 0  & 0  & 0  & 0  & 0  & 0  & 0 & 12 & 12  & 2 & 12 & 10 & 11  & 1  & 0  & 2 & 12 & 12  & 3  & 0  & 2  & 9  & 1\\
 0  & 0  & 0  & 0  & 0  & 0  & 0  & 0 & 12  & 0  & 0  & 1  & 0  & 0  & 1  & 0  & 0  & 1 & 12  & 0 & 11  & 1 & 12\\
 0  & 0  & 0  & 0  & 0  & 0  & 0 & 12  & 1  & 0  & 0 & 11 & 12  & 0 & 12  & 1 & 12 & 12  & 1  & 0  & 2 & 12  & 0\\
 0  & 0  & 0  & 0  & 0  & 0  & 0  & 0 & 12  & 1  & 0 & 12  & 0  & 0  & 0  & 1  & 0  & 0  & 1  & 0  & 1 & 12  & 0\\
 0  & 0  & 0  & 0  & 0  & 0  & 0  & 0  & 1 & 12  & 1  & 0  & 0  & 0  & 0 & 12  & 0 & 12 & 12 & 12  & 0  & 2  & 0\\
 0  & 0  & 0  & 0  & 0  & 0  & 0  & 0  & 1 & 12  & 1  & 0  & 0  & 0 & 12 & 12  & 1  & 0 & 12 & 12  & 0  & 1  & 0
\end{array} \right), 
\displaybreak[0]\\[4pt]
\mathfrak x &= \left( \begin{array}{*{23}{c@{\,}}c}
 12  & 0  & 0  & 0  & 0  & 0  & 0  & 0  & 0  & 0  & 0  & 0  & 0  & 0  & 0  & 0  & 0  & 0  & 0  & 0  & 0  & 0  & 0\\
 0  & 2  & 0 & 11 & 12 & 12  & 9  & 0  & 0  & 0  & 0  & 0  & 0  & 0  & 0  & 0  & 0  & 0  & 0  & 0  & 0  & 0  & 0\\
 0  & 0 & 10 & 12  & 8  & 4  & 5  & 0  & 0  & 0  & 0  & 0  & 0  & 0  & 0  & 0  & 0  & 0  & 0  & 0  & 0  & 0  & 0\\
 0 & 11  & 2 & 12 & 12  & 2  & 3  & 0  & 0  & 0  & 0  & 0  & 0  & 0  & 0  & 0  & 0  & 0  & 0  & 0  & 0  & 0  & 0\\
 0  & 9  & 7 & 10  & 6  & 5  & 2  & 0  & 0  & 0  & 0  & 0  & 0  & 0  & 0  & 0  & 0  & 0  & 0  & 0  & 0  & 0  & 0\\
 0  & 6  & 6  & 1  & 0  & 8  & 2  & 0  & 0  & 0  & 0  & 0  & 0  & 0  & 0  & 0  & 0  & 0  & 0  & 0  & 0  & 0  & 0\\
 0  & 3  & 3  & 7  & 0 & 11  & 0  & 0  & 0  & 0  & 0  & 0  & 0  & 0  & 0  & 0  & 0  & 0  & 0  & 0  & 0  & 0  & 0\\
 0  & 0  & 0  & 0  & 0  & 0  & 0  & 0  & 1  & 0 & 12  & 0  & 0  & 0  & 0  & 0 & 12 & 12  & 1  & 0  & 1  & 0  & 1\\
 0  & 0  & 0  & 0  & 0  & 0  & 0 & 12  & 0  & 2 & 12 & 11 & 11  & 1  & 0  & 2 & 12 & 12  & 2  & 1  & 1 & 10  & 1\\
 0  & 0  & 0  & 0  & 0  & 0  & 0  & 0 & 12  & 1  & 0  & 0 & 12  & 0  & 1  & 1  & 0  & 1  & 0  & 1 & 12 & 12  & 0\\
 0  & 0  & 0  & 0  & 0  & 0  & 0 & 12  & 0  & 1 & 12 & 11 & 12  & 0  & 0  & 2 & 12 & 12  & 2  & 1  & 2 & 11  & 0\\
 0  & 0  & 0  & 0  & 0  & 0  & 0 & 12  & 1  & 1 & 12 & 10 & 11  & 0 & 12  & 2 & 11 & 11  & 3  & 1  & 3 & 10  & 1\\
 0  & 0  & 0  & 0  & 0  & 0  & 0  & 0  & 1 & 12  & 1  & 1  & 1  & 0  & 0 & 12  & 1  & 0 & 11  & 0 & 12  & 2 & 12\\
 0  & 0  & 0  & 0  & 0  & 0  & 0  & 0  & 2 & 10  & 1  & 2  & 2 & 12 & 12 & 10  & 1  & 0 & 11 & 12 & 12  & 4 & 12\\
 0  & 0  & 0  & 0  & 0  & 0  & 0  & 0  & 0  & 1  & 0 & 12 & 12  & 1  & 0  & 0  & 0  & 0  & 0  & 0  & 0 & 12  & 0\\
 0  & 0  & 0  & 0  & 0  & 0  & 0 & 12  & 0  & 2 & 11 & 10 & 11  & 1  & 0  & 3 & 11 & 11  & 4  & 1  & 3  & 8  & 2\\
 0  & 0  & 0  & 0  & 0  & 0  & 0  & 0 & 12  & 2 & 12 & 12 & 12  & 1  & 0  & 1  & 0  & 0  & 2  & 0  & 1 & 10  & 1\\
 0  & 0  & 0  & 0  & 0  & 0  & 0 & 12  & 0  & 0  & 0  & 0  & 0  & 0  & 0  & 0  & 0  & 0  & 0  & 0  & 0  & 0  & 0\\
 0  & 0  & 0  & 0  & 0  & 0  & 0 & 12  & 1  & 0  & 0 & 12 & 12  & 0  & 0  & 1 & 12 & 12  & 1  & 0  & 1  & 0  & 0\\
 0  & 0  & 0  & 0  & 0  & 0  & 0  & 0 & 12  & 0  & 0  & 2  & 1  & 0  & 0 & 12  & 1  & 1 & 12  & 0 & 12  & 1 & 12\\
 0  & 0  & 0  & 0  & 0  & 0  & 0  & 0  & 0  & 0  & 0  & 0  & 0  & 0  & 0 & 12  & 0  & 0  & 0  & 0  & 0  & 0  & 0\\
 0  & 0  & 0  & 0  & 0  & 0  & 0  & 0 & 12  & 0  & 0  & 1  & 0  & 0  & 1  & 0  & 0  & 1  & 0  & 0 & 12  & 0  & 0\\
 0  & 0  & 0  & 0  & 0  & 0  & 0  & 0 & 12  & 0  & 0  & 1  & 1  & 0  & 0  & 0  & 1  & 1  & 0  & 0  & 0  & 0  & 0
\end{array} \right),
\displaybreak[0]\\[4pt]
\mathfrak y &= \left( \begin{array}{*{23}{c@{\,}}c}
 12  & 0  & 0  & 0  & 0  & 0  & 0  & 0  & 0  & 0  & 0  & 0  & 0  & 0  & 0  & 0  & 0  & 0  & 0  & 0  & 0  & 0  & 0\\
 0  & 9  & 3 & 12  & 2  & 3  & 9  & 0  & 0  & 0  & 0  & 0  & 0  & 0  & 0  & 0  & 0  & 0  & 0  & 0  & 0  & 0  & 0\\
 0  & 1  & 0 & 12 & 10  & 1  & 7  & 0  & 0  & 0  & 0  & 0  & 0  & 0  & 0  & 0  & 0  & 0  & 0  & 0  & 0  & 0  & 0\\
 0 & 10  & 4 & 11  & 9  & 3  & 3  & 0  & 0  & 0  & 0  & 0  & 0  & 0  & 0  & 0  & 0  & 0  & 0  & 0  & 0  & 0  & 0\\
 0  & 7  & 7 & 12  & 6 & 11  & 2  & 0  & 0  & 0  & 0  & 0  & 0  & 0  & 0  & 0  & 0  & 0  & 0  & 0  & 0  & 0  & 0\\
 0  & 9  & 7 & 10  & 0 & 12  & 2  & 0  & 0  & 0  & 0  & 0  & 0  & 0  & 0  & 0  & 0  & 0  & 0  & 0  & 0  & 0  & 0\\
 0 & 10 & 10  & 6 & 10 & 12  & 0  & 0  & 0  & 0  & 0  & 0  & 0  & 0  & 0  & 0  & 0  & 0  & 0  & 0  & 0  & 0  & 0\\
 0  & 0  & 0  & 0  & 0  & 0  & 0  & 1  & 0 & 12  & 0  & 1  & 1  & 0  & 0 & 12  & 1  & 1 & 12 & 12 & 12  & 1  & 0\\
 0  & 0  & 0  & 0  & 0  & 0  & 0  & 1  & 0 & 12  & 1  & 2  & 2  & 0  & 0 & 11  & 2  & 1 & 11 & 12 & 11  & 3 & 12\\
 0  & 0  & 0  & 0  & 0  & 0  & 0  & 1  & 1 & 11  & 1  & 2  & 2 & 12  & 0 & 11  & 1  & 0 & 11 & 12 & 12  & 4  & 0\\
 0  & 0  & 0  & 0  & 0  & 0  & 0  & 1 & 12  & 1  & 0  & 0  & 0  & 0  & 0  & 0  & 0  & 1  & 0  & 0  & 0 & 12  & 0\\
 0  & 0  & 0  & 0  & 0  & 0  & 0  & 0 & 12  & 2 & 12 & 12 & 12  & 1  & 0  & 1  & 0  & 0  & 2  & 0  & 1 & 10  & 1\\
 0  & 0  & 0  & 0  & 0  & 0  & 0  & 0  & 1  & 0  & 0 & 12  & 0  & 0  & 0  & 0 & 12 & 12  & 0  & 0  & 1  & 0  & 0\\
 0  & 0  & 0  & 0  & 0  & 0  & 0 & 12 & 12  & 2 & 11 & 10 & 11  & 0  & 0  & 3 & 11 & 12  & 4  & 1  & 3  & 8  & 1\\
 0  & 0  & 0  & 0  & 0  & 0  & 0  & 1  & 2  & 9  & 3  & 4  & 3 & 12 & 12  & 8  & 3  & 1  & 8 & 11 & 10  & 7 & 11\\
 0  & 0  & 0  & 0  & 0  & 0  & 0  & 0  & 0  & 1 & 12 & 11 & 12  & 1  & 0  & 1 & 12 & 12  & 2  & 0  & 1 & 11  & 1\\
 0  & 0  & 0  & 0  & 0  & 0  & 0  & 1  & 0 & 11  & 1  & 3  & 2  & 0  & 0 & 10  & 2  & 1 & 10 & 12 & 10  & 4 & 12\\
 0  & 0  & 0  & 0  & 0  & 0  & 0  & 0  & 0 & 12  & 1  & 1  & 1 & 12  & 0 & 12  & 1  & 1 & 12  & 0 & 12  & 2 & 12\\
 0  & 0  & 0  & 0  & 0  & 0  & 0  & 1 & 11  & 1  & 0  & 1  & 0  & 0  & 1  & 0  & 1  & 2 & 12  & 0 & 11  & 0 & 12\\
 0  & 0  & 0  & 0  & 0  & 0  & 0  & 0  & 0 & 12  & 0  & 1  & 1 & 12 & 12 & 12  & 1  & 1 & 12  & 0  & 0  & 1 & 12\\
 0  & 0  & 0  & 0  & 0  & 0  & 0  & 0  & 1 & 11  & 1  & 1  & 1 & 12 & 12 & 11  & 1  & 0 & 12 & 12  & 0  & 2  & 0\\
 0  & 0  & 0  & 0  & 0  & 0  & 0  & 0  & 0 & 12  & 0  & 0  & 0 & 12  & 0  & 0  & 0  & 0  & 0  & 0  & 0  & 1  & 0\\
 0  & 0  & 0  & 0  & 0  & 0  & 0  & 0  & 0  & 0  & 0  & 0  & 0  & 0  & 0  & 0  & 0  & 0  & 0  & 1  & 0  & 0  & 0
\end{array} 
\right),\quad and \displaybreak[0]\\[4pt]
\mathfrak e &= \left( \begin{array}{*{23}{c@{\,}}c}
 1  & 0  & 0  & 0  & 0  & 0  & 0  & 0  & 0  & 0  & 0  & 0  & 0  & 0  & 0  & 0  & 0  & 0  & 0  & 0  & 0  & 0  & 0\\
 0  & 6  & 5  & 5  & 9  & 4  & 1 & 11 & 12 & 10  & 8  & 4  & 6  & 9  & 7  & 5  & 2  & 5  & 3  & 9  & 0  & 9  & 9\\
 0  & 1  & 9 & 11  & 5  & 1  & 5 & 10  & 9  & 5 & 12  & 3  & 6 & 10  & 4 & 10  & 3  & 1 & 11  & 7  & 0  & 1  & 7\\
 0  & 8  & 9  & 4  & 0 & 11 & 12  & 2  & 9  & 2  & 2  & 0 & 11  & 5  & 1  & 0  & 3  & 3  & 8 & 12  & 8  & 9 & 12\\
 0  & 2  & 9 & 10  & 8  & 8  & 6 & 10  & 3  & 3  & 0  & 4  & 7  & 3 & 10  & 3 & 10  & 7  & 3  & 0  & 6  & 4  & 0\\
 0 & 10  & 3  & 7  & 9  & 3  & 7 & 10  & 5  & 8  & 1 & 11  & 1  & 3  & 3  & 1  & 4  & 0  & 8  & 6 & 12  & 6  & 6\\
 0  & 5  & 8 & 10 & 11  & 8 & 10  & 8  & 5  & 5  & 0 & 11  & 3  & 5  & 8  & 5  & 8  & 3  & 5  & 0 & 10 & 11  & 0\\
 0 & 11 & 12  & 4 & 12  & 1  & 1  & 7  & 0  & 6  & 0  & 0  & 0  & 6  & 6  & 6  & 0  & 0  & 0  & 0  & 7  & 0  & 0\\
 0  & 0  & 0  & 0  & 7  & 7  & 0  & 0  & 0  & 6  & 7  & 6  & 6  & 6  & 7  & 7  & 6  & 0  & 0  & 0  & 0  & 0  & 0\\
 0  & 9  & 8  & 8  & 8 & 10  & 1  & 0  & 0 & 12  & 7  & 0  & 0  & 6  & 0  & 6  & 0  & 7  & 6  & 6  & 6  & 1  & 6\\
 0  & 6  & 6  & 1  & 6  & 2 & 11  & 7  & 0  & 5  & 8  & 2  & 1 & 12  & 6 & 11  & 2  & 8  & 4  & 6 & 11  & 3  & 5\\
 0  & 6  & 6  & 1  & 7  & 2  & 0  & 0  & 6  & 6  & 7  & 1  & 7  & 6  & 6  & 6  & 1  & 1  & 6  & 0  & 6  & 7 & 12\\
 0  & 6  & 6  & 1  & 6  & 2 & 11  & 6  & 0  & 7  & 6 & 12 & 12  & 0  & 7  & 1 & 12  & 6  & 7  & 7  & 0 & 12  & 7\\
 0  & 8  & 7 & 10  & 8  & 1 & 12  & 6  & 6  & 2  & 5 & 11  & 5  & 1  & 0  & 2 & 12 & 12  & 9  & 1  & 2  & 3  & 1\\
 0 & 11 & 12  & 4  & 5  & 8 & 12  & 0  & 7 & 12  & 7  & 7  & 1  & 6  & 0  & 6  & 7  & 7 & 12  & 6  & 0  & 8  & 6\\
 0  & 0  & 0  & 0  & 7  & 6  & 2  & 6  & 6  & 7  & 6  & 6 & 12  & 0  & 7  & 1  & 6  & 6  & 1  & 7  & 7  & 5  & 7\\
 0  & 5  & 6  & 3  & 5 & 12  & 1  & 6  & 7  & 0  & 6 & 12  & 6  & 0  & 0  & 1 & 12 & 12  & 8  & 0  & 2  & 5  & 1\\
 0  & 2  & 1  & 9  & 8  & 6 & 12  & 7  & 1 & 12  & 7  & 7  & 7  & 0  & 0 & 12  & 7  & 0 & 12 & 12  & 6  & 2  & 0\\
 0  & 6  & 6  & 1  & 0  & 9 & 11  & 7  & 0 & 12  & 1  & 8  & 7  & 6  & 0  & 5  & 7  & 7  & 5  & 6 & 12  & 2  & 6\\
 0  & 0  & 0  & 0  & 6  & 7 & 11  & 7  & 7  & 6  & 6  & 7  & 1  & 0  & 6 & 12  & 7  & 6  & 0  & 6  & 7  & 7  & 7\\
 0  & 0  & 0  & 0  & 0  & 0  & 0  & 0  & 1 & 12  & 0  & 0  & 1  & 0 & 12 & 12  & 1  & 0  & 0 & 12  & 1  & 1  & 0\\
 0  & 9  & 8  & 8  & 8 & 10  & 1  & 0  & 0  & 0  & 6  & 0  & 0  & 7  & 0  & 6  & 0  & 6  & 7  & 6  & 7  & 0  & 7\\
 0  & 7  & 7 & 12  & 6 & 11  & 0  & 0  & 6  & 7  & 6  & 0  & 6  & 7  & 7  & 7  & 0  & 0  & 7  & 0  & 6  & 6  & 0
\end{array} \right).
\end{alignat*}}
\item[\rm(4)]  $\mathfrak G_3 = \langle \mathfrak h, \mathfrak x,
\mathfrak y, \mathfrak e \rangle$, and $\mathfrak{G}_3$ has $60$
conjugacy classes $\mathfrak{g}_i^{\mathfrak{G}_3}$ with
representatives $\mathfrak{g}_i$ and centralizer orders
$|C_{\mathfrak{G}_3}(\mathfrak{g}_i)|$ as given in Table
\ref{Co_2cc G_3}.

\item[\rm(5)] The character table of $\mathfrak G_3$ coincides
with that of $\Co_2$ in the Atlas 
\cite[pp.~154--155]{atlas}.
\end{enumerate}

\item[\rm(d)] $\mathfrak G_3$ is a finite simple group with
$2$-central involution $\kappa_\mathfrak V(z) =(\mathfrak x
\mathfrak y \mathfrak h)^{15} $ such that 
$
C_{\mathfrak G_3} (\kappa_\mathfrak V(z)) = \kappa_\mathfrak
V(H)$, and $|\mathfrak G_3| = 2^{18}\cdot 3^6\cdot 7\cdot 11
\cdot 13$.
\end{enumerate}
\end{theorem}

\begin{proof}
(a) The character tables of the groups 
$H$, $D$, and $E_3$ are
stated in the 
Appendices. In the following we use their notations.
Using \textsc{Magma} and the character tables of $H$, $D$, and $E_3$ and
the fusion of the classes of $D(H)$ in $H$ and $D(E_3)$ in $E_3$
an application of Kratzer's Algorithm 7.3.10 of \cite{michler}
yields the compatible pair stated in assertion (a).

(b) 
Kratzer's Algorithm also shows that the pair $(\chi,\tau)$ of (b) is the unique
compatible pair of degree $23$ with respect to the fusion of the
$D$-classes into the $H$- and 
$E_3$-classes.

(c) In order to construct the 
semisimple faithful representation
$\mathfrak V$ corresponding to the character $\chi = \chi_2 +
\chi_{\bf 4}$ we determine the irreducible constituents of the faithful
permutation representation $(1_T)^H$ of $H$ with stabilizer $T =
\langle h_1,h_4 \rangle$ given in Lemma~\ref{l. H(Co_2)}.
Calculating inner products with the irreducible characters of $H$
it follows $(1_T)^H$ has 
five irreducible constituents of degrees
$1$, $16$, $120$, 
$135$, and $240$. By the character table of $H$
we know that 
$\chi_{\bf 4}$ is the only irreducible character of $H$ of
degree $16$. Constructing the permutation matrices of the 
three generators $h$, 
$x$, and $y$ of $H$ over the field $K = \GF(13)$
and applying the Meataxe 
Algorithm implemented in \textsc{Magma} we obtain
the matrices $\mathfrak Jx, \mathfrak Jy, \mathfrak Jh \in
\GL_{16}(13)$ of the generators of $H$ in the $16$-dimensional
representation $\mathfrak V_1$. Furthermore, it has been checked
that $\mathfrak V_1$ restricted to $D(H) = \langle x,y \rangle$ is
irreducible.

In order to find the irreducible module $\mathfrak V_2$
corresponding to the character $\chi_2$ of degree $7$ we applied
Theorem 2.5.4 of \cite{michler} by calculating the 
$4$th exterior power $\bigwedge^4 \mathfrak V_1$ of dimension $1820$
over $K = 
\GF(13)$ and decomposing it into irreducible composition
factors. It follows that it has 
three composition factors $\mathfrak
W_j$, $1 \le j \le 3$, of dimensions 
$35$, $840$, and~$945$. Let $\mathfrak
W_1$ be the one of dimension $35$. Its second exterior power
$\bigwedge^2 \mathfrak W_1$ of dimension $595$ splits into 
four irreducible constituents $\mathfrak X_k$, 
$1 \le k \le 4$, of dimensions 
$7$, $21$, $189$, and~$378$. $H$ has exactly one irreducible
character of dimension $7$ by Table \ref{Co_2ct_H}. Hence
$\mathfrak V_2 = \mathfrak X_1$. Choosing a basis in this
$KH$-module the generators of $H$ are represented by the following
matrices: 
{\renewcommand{\arraystretch}{0.5}
\scriptsize
$$
\mathfrak Nx = \left( \begin{array}{*{7}{c@{\,}}c}
 5 &  8  &  8  &  5  &  9  & 11  &  8\\
 9 &  6  & 11  &  1  &  6  &  3  &  1\\
10 &  3  &  2  & 11  & 10  &  7  & 12\\
 3 &  8  & 10  &  0  & 10  & 12  &  9\\
11 & 11  &  6  &  5  &  3  &  7  & 11\\
 0 &  6  &  6  &  1  &  0  &  8  &  2\\
 0 &  3  &  3  &  7  &  0  &  11 &  0
\end{array} \right),
\quad 
\mathfrak Ny = \left( \begin{array}{*{7}{c@{\,}}c}
 4 & 10 &  4  &  1 & 10 &  8 &  5\\
 8 &  1 &  7  &  6 &  1 &  8 & 12\\
10 &  4 &  5  & 11 & 12 &  4 &  1\\
 5 &  5 &  0  &  4 & 10 & 11 &  0\\
12 &  8 &  0  &  3 & 11 & 12 &  0\\
12 & 10 &  0  &  1 &  5 &  0 &  0\\
 6 &  4 &  0  &  8 &  6 &  6 & 12
 \end{array} 
\right),\quad\!\mbox{and}\quad\!
\mathfrak Nh = \left( \begin{array}{*{7}{c@{\,}}c}
12 & 10 & 11 &  8 &  5  &  9  &  3\\
 4 &  0 & 11 &  1 & 11  &  4  &  5\\
 0 &  3 &  6 &  5 &  6  &  9  &  2\\
 7 &  5 &  4 & 10 &  0  &  3  &  6\\
 1 &  4 &  0 & 11 & 11  &  0  &  0\\
 1 &  3 & 12 &  2 &  2  & 12  &  2\\
 7 &  2 &  0 & 12 & 12  &  7  &  1
\end{array} \right).%
$$}%
Using \textsc{Magma} and the Meataxe 
Algorithm again we see that the
restriction of $\mathfrak V_2$ to $D_H = \langle x, y \rangle$
splits into two irreducible $KD_H$-modules $\mathfrak V_{21}$ and
$\mathfrak V_{22}$ of dimensions~$1$ and~$6$ having respective
bases $B_1$ and $B_2$. Hence their union $B$ is a basis of
$\mathfrak V_2$. Let $\mathcal T_1 \in \GL_7(13)$ be the matrix of
the base change from the canonical basis to the basis $B$ of
$\mathfrak V_2$. Then $\mathfrak Lx = \mathcal T_1\mathfrak
Nx(\mathcal T_1)^{-1}$, $\mathfrak Ly = \mathcal T_1\mathfrak
Ny(\mathcal T_1)^{-1}$ and $\mathfrak Lh = \mathcal T_1\mathfrak
Nh(\mathcal T_1)^{-1}$ represent the 
three generators $x,y,h$ of $H$
with respect to the basis $B$ of $\mathfrak V_2$. It follows that the
blocked diagonal $7 \times 7$ matrices ${\mathfrak L}x$ and
${\mathfrak L}y$ are the $7 \times 7$ matrices in the 
upper left
corner of the matrices $\mathfrak x$ and $\mathfrak y$ given in
the statement. Furthermore, the matrix ${\mathfrak L}h$ of $h \in
H$ is the corresponding $7 \times 7$ block of the matrix
$\mathfrak h$ in the statement.

Let $\mathfrak x$, $\mathfrak y$ and $\mathfrak h$ be the diagonal
joins of the matrices ${\mathfrak L}x$ and ${\mathfrak J}x$,
${\mathfrak L}y$ and ${\mathfrak J}y$ and ${\mathfrak L}h$ and
${\mathfrak J}h$ in $\GL_{23}(13)$, respectively. Then these three
$23 \times 23$ matrices are 
given in the statement. Clearly, they
satisfy the equations $\mathfrak x = \kappa_{\mathfrak V}(x)$,
$\mathfrak y = \kappa_{\mathfrak V}(y)$ and $\mathfrak h =
\kappa_{\mathfrak V}(h)$ where $\kappa_{\mathfrak V}: H
\rightarrow \GL_{23}(13)$ denotes the 
semisimple representation
of $H = \langle x,y,h \rangle$ afforded by $\mathfrak V$.

Now we construct the faithful representation $\mathfrak W$ of
$E_3$ corresponding to the character $\tau = \tau_{2} + \tau_{\bf
6}$. By the character table of $E_3$ we know that $\tau_2$ is the
only 
nontrivial  linear character of $E$. Its corresponding
representation 
$\mathfrak{W}_1\colon E_3\to \GL_1(K)$ is given
by $\mathfrak W_1(x_1) = \mathfrak W_1(y_1) = 12$ and $\mathfrak
W_1(e_1) = 1$. In order to find the irreducible representation
$\mathfrak W_2$ of $\tau_{\bf 6}$ we determine the irreducible
constituents of the faithful permutation representation
$(1_{T_1})^{E_3}$ of $E_3$ of degree $1024$ with stabilizer
$A_{22}$ constructed in Lemma \ref{l. Aut(M22)-extensions}. An
application of \textsc{Magma} shows that it has 
four irreducible
constituents of degrees $1$, $22$, 
$231$, and $770$. Using the
Meataxe 
Algorithm implemented in \textsc{Magma} the matrices $\mathfrak
Jx_1,\mathfrak Jy_1, \mathfrak Je_1 \in \GL_{22}(13)$ of the
generators of $E_3$ in the $22$-dimensional irreducible
constituent $\mathfrak X$ of $(1_{T_1})^E$ over $K = \GF(13)$ and
their traces have been determined. Now it follows from the
character table of $E_3$ that $\tau_{\bf 6}$ is the character of
the tensor product $\mathfrak X \otimes \mathfrak W_1$. Another
application of the Meataxe 
Algorithm and \textsc{Magma} implies that the
restriction of $\mathfrak W_2$ to $D_E = \langle x_1, y_1 \rangle$
splits into two irreducible $KD_E$-modules $\mathfrak W_{21}$ and
$\mathfrak W_{22}$ having respective bases $B_1$ and $B_2$ such
that the generators $x_1$ and $y_1$ of $D$ have the following
matrices 
with respect to $B_1$ and $B_2$, respectively:
{\renewcommand{\arraystretch}{0.5}
\scriptsize%
\begin{alignat*}{1}
\mathfrak W_{21}(x_1) &= \left( \begin{array}{*{6}{c@{\,}}c}
12 & 1  &  1  &  0  &  0  &  0\\
 0 & 1  &  0  &  0  & 12  &  1\\
12 & 0  &  0  &  0  &  0  &  0\\
 0 & 0  & 12  &  0  &  0  &  0\\
 1 & 0  &  0  &  1  &  0  & 12\\
 0 & 0  &  0  &  1  &  1  & 12
\end{array} \right) \quad \mbox{and}\quad
\mathfrak W_{21}(y_1) = \left( \begin{array}{*{6}{c@{\,}}c}
12 &  0 &  1  &  12 &  0  &  0\\
12 &  0 &  0  &  12 &  0  &  1\\
 0 &  0 &  1  &   0 &  1  & 12\\
 0 &  0 &  0  &   1 &  1  & 12\\
12 &  1 &  0  &  12 & 12  &  1\\
 0 &  0 &  0  &   0 &  0  & 12
\end{array} 
\right);\\
\mathfrak W_{22}(x_1) &= \left( \begin{array}{*{16}{c@{\,}}c}
 1 & 12  & 0  & 0  & 0  & 0  & 1 & 12  & 0 & 12 & 12  & 1  & 1  &  1  & 12 & 12\\
 2  & 0  & 1  & 0  & 1 & 12  & 1 & 11 & 12 & 12 & 12  & 0  & 1  & 12  &  1 & 12\\
 3 & 12  & 2  & 1  & 1 & 12  & 0 & 11 & 11 & 12  & 0 & 12  & 1  & 12  &  2 & 11\\
 3 & 11  & 2  & 2  & 2 & 12  & 0 & 10 & 12 & 12  & 1 & 11  & 0  &  0  &  2 & 10\\
 0  & 0  & 0  & 0  & 0 & 12  & 0  & 1 & 12  & 0  & 1  & 0 & 12  & 12  & 1  &  0\\
12  & 1  & 0 & 12 & 12  & 0  & 1  & 0  & 1  & 0 & 11  & 1  & 1  &  1  & 12 &  1\\
 0 & 12  & 0  & 1  & 0  & 0  & 0 & 12  & 0  & 0  & 1 & 12  & 0  &  1  &  0 & 12\\
 2  & 0  & 1  & 1  & 1 & 12 & 12 & 12 & 11  & 0  & 1 & 11  & 0  & 12  &  2 & 12\\
 1 & 11  & 0  & 1  & 1  & 1  & 1 & 12  & 0 & 12  & 0  & 1  & 1  &  1  & 12 & 11\\
 2 & 10  & 1  & 2  & 1  & 0  & 0 & 11 & 12 & 12  & 2 & 12  & 1  &  1  &  0 & 10\\
 2 & 11  & 1  & 1  & 1  & 0  & 1 & 11 & 12 & 12  & 0  & 0  & 1  &  0  &  0 & 11\\
 3 & 12  & 2  & 0  & 1 & 12  & 1 & 11 & 12 & 12 & 12  & 0  & 1  & 12  &  1 & 11\\
 2 & 10  & 1  & 3  & 2  & 0 & 12 & 11 & 11  & 0  & 3 & 11  & 0  &  0  &  1 & 10\\
 1  & 0  & 0  & 0  & 0 & 12  & 0  & 0 & 12  & 0  & 0  & 0  & 0  & 12  &  1 &  0\\
 0  & 0  & 0  & 0  & 0  & 0  & 0  & 0  & 0  & 1  & 0  & 0  & 0  &  0  &  0 &  0\\
 0 & 11  & 0  & 1  & 0  & 0  & 0  & 0  & 0  & 0  & 1  & 0 & 12  &  1  &  0 & 12
 \end{array}
\right)\\
\mbox{and}\quad\mathfrak W_{22}(y_1) &= \left( \begin{array}{*{16}{c@{\,}}c}
 3 & 10  & 2  & 3  & 2 & 12  & 0 & 10 & 11 & 12  & 2 & 11  & 0  & 0  &  2 &  10\\
 2 & 11  & 2  & 2  & 1 & 12  & 0 & 11 & 12 & 12  & 2 & 11 & 12  & 0  &  2 &  11\\
11  & 1 & 12  & 0 & 12  & 0 & 12  & 1  & 0  & 1  & 1  & 0 & 12 &  0  &  0  &  1\\
 2 & 11  & 1  & 2  & 1  & 0 & 12 & 11 & 12  & 0  & 2 & 11  & 0  & 0  &  1 &  11\\
12  & 0  & 0  & 0  & 0  & 0  & 0  & 0  & 0  & 0  & 0  & 1  & 0 &  0  & 12  &  0\\
 0  & 2  & 0 & 12  & 0  & 0  & 0  & 0  & 0  & 0 & 12  & 0  & 0 & 12  &  1  &  1\\
 2  & 0  & 2  & 0  & 1 & 12  & 0 & 11 & 12  & 0  & 0 & 11  & 0 & 12  &  2 &  12\\
12  & 0  & 0  & 1  & 0  & 0 & 12  & 0 & 12  & 1  & 2 & 12 & 12 &  0  &  0  &  0\\
12 & 11 & 12  & 2  & 0  & 1  & 0  & 0  & 1  & 0  & 1  & 0  & 0 &  2 &  12 &  12\\
 1  & 0  & 1  & 1  & 1 & 12 & 12 & 12 & 12  & 1  & 1 & 11 & 12 & 12  &  2  & 12\\
 2 & 11  & 2  & 2  & 1 & 12  & 0 & 11 & 12 & 12  & 1 & 11  & 0  & 0  &  2 &  11\\
 1 & 12  & 1  & 2  & 1 & 12 & 12 & 12 & 12  & 0  & 2 & 11 & 12 & 12  &  2 &  12\\
 1 & 11  & 1  & 2  & 1  & 0 & 12 & 12 & 12  & 0  & 2 & 11  & 0  & 0  &  1 &  11\\
 0  & 0  & 1  & 0  & 0 & 12  & 0  & 0 & 12  & 0  & 1  & 0 & 12 & 12  &  1  &  0\\
12  & 1  & 0  & 0  & 0  & 0 & 12  & 1  & 0  & 1  & 1 & 12 & 12 & 12  &  1  &  1\\
 2 & 11  & 1  & 1  & 1  & 0  & 0 & 11 & 12 & 12  & 0  & 0  & 2  & 0  &  0 &  11
 \end{array} 
 \right).
\end{alignat*}}
Let $B = B_1 \cup B_2$. Then $B$ is a basis of the irreducible
$KE_3$-module $\mathfrak W_2$, and 
with respect to $B$ the third generator
$e_1$ of $E_3$ has the matrix:
{\renewcommand{\arraystretch}{0.5}
\scriptsize
$$
\mathfrak W_2(e_1) = \left( \begin{array}{*{22}{c@{\,}}c}
 7  & 0  & 6  & 0  & 0  & 0  & 7  & 6  & 7  & 7  & 7  & 0  & 6  & 6  & 6  & 0  & 7  & 6  & 0  & 0  & 7  & 6\\
 6  & 7  & 0  & 6  & 6  & 1  & 7  & 0  & 7  & 6  & 0  & 0  & 0  & 6  & 0  & 0  & 6  & 0  & 7  & 0  & 0  & 0\\
 7  & 6  & 0  & 7  & 7 & 12  & 6  & 6  & 0  & 7  & 0  & 0  & 0  & 0  & 0  & 0  & 7  & 0  & 6  & 7  & 0  & 0\\
 6  & 0  & 7  & 0  & 0  & 0 & 12  & 6  & 6  & 1  & 0  & 0  & 6  & 7  & 0  & 7  & 8  & 6 & 12  & 7  & 0  & 0\\
 7  & 0  & 6  & 0  & 0  & 0  & 0  & 0  & 0  & 0  & 0  & 0  & 6  & 0  & 0  & 7  & 7  & 6  & 0  & 0  & 0  & 0\\
 7  & 0  & 6  & 0  & 0  & 0 & 12  & 6  & 6  & 1  & 0  & 0  & 6  & 7  & 0  & 7  & 8  & 6 & 12  & 7  & 0  & 0\\
 6  & 0  & 7  & 0  & 1  & 0 & 11  & 9  & 6 & 11 & 12  & 0  & 7  & 8  & 1  & 7  & 5  & 8  & 0  & 6 & 12  & 2\\
 7  & 6  & 0  & 7  & 7  & 0  & 7 & 12  & 7  & 7  & 1  & 0  & 1  & 6  & 0 & 12  & 7  & 0  & 7  & 0  & 0 & 12\\
 7  & 0  & 7  & 0  & 0  & 0  & 7 & 11  & 0  & 8  & 7  & 1  & 7 & 12  & 7 & 12  & 7  & 7  & 1  & 8  & 5  & 5\\
 7  & 0  & 7  & 0  & 0  & 0  & 8  & 0  & 2  & 7  & 7 & 12  & 6 & 12  & 5  & 0  & 7  & 5  & 0  & 5  & 8  & 6\\
 7  & 6  & 0  & 7  & 7  & 0  & 5  & 0  & 5  & 6 & 12  & 1  & 1  & 7  & 2  & 0  & 5  & 2  & 7  & 2 & 11  & 0\\
 7 & 12  & 6  & 0  & 0  & 0 & 12  & 0 & 12  & 0  & 0  & 1  & 6  & 1  & 0  & 7  & 7  & 7  & 0  & 0 & 12  & 1\\
 6  & 0  & 7  & 0  & 1  & 0  & 1 & 12  & 1  & 1  & 1  & 0  & 6 & 12 & 12  & 6  & 8  & 5  & 0 & 12  & 1 & 12\\
 7  & 0  & 7  & 0  & 0  & 0  & 5  & 7  & 6  & 6  & 6  & 1  & 7  & 7  & 8  & 0  & 6  & 7  & 0  & 1  & 5  & 7\\
 6  & 7  & 1  & 6  & 7  & 0 & 12  & 1  & 6 & 12  & 5  & 0  & 0  & 8  & 7  & 7  & 6  & 1  & 6  & 7  & 6  & 8\\
 6  & 7  & 1  & 6  & 7  & 0  & 5  & 7 & 12  & 6 & 12  & 1  & 0  & 1  & 1  & 0  & 6  & 1  & 7  & 7 & 12  & 1\\
 6  & 0  & 7  & 0  & 1  & 0  & 0  & 0  & 0  & 0  & 0  & 0  & 7  & 0  & 0  & 6  & 7  & 7  & 0  & 0  & 0  & 0\\
 7  & 0  & 7  & 0  & 0  & 0  & 6  & 7  & 6  & 6  & 6  & 0  & 7  & 7  & 7  & 0  & 6  & 8  & 0  & 0  & 6  & 7\\
 6  & 7  & 1  & 6  & 7  & 0  & 7  & 0  & 7  & 7  & 0  & 0  & 0  & 6  & 0  & 0  & 7 & 12  & 6  & 0  & 1  & 0\\
 7  & 6  & 0  & 7  & 7  & 0  & 7  & 6  & 0  & 7  & 1  & 0  & 1 & 12  & 1 & 12  & 6  & 0  & 7  & 7  & 0 & 12\\
 7  & 0  & 7  & 0  & 0  & 0  & 0  & 7  & 6  & 0  & 0  & 0  & 6  & 7  & 0  & 7  & 7  & 6 & 12  & 6  & 1  & 1\\
 6  & 0  & 7  & 0  & 1  & 0  & 1  & 1  & 1  & 0  & 0 & 12  & 6 & 12 & 12  & 7  & 6  & 6  & 0 & 12  & 1  & 0
 \end{array} \right).$$}%

By construction, the $KD$-modules $\mathfrak V_{|D}$ and $\mathfrak
W_{|D}$ described by the 
two pairs $(\mathfrak V(x), 
\mathfrak{V}(y))$ and $(\mathfrak W(x_1),\mathfrak W(y_1))$ of matrices in
$\GL_{23}(13)$ are isomorphic. Let $Y = 
\GL_{23}(13)$. Applying then Parker's isomorphism test of
Proposition 6.1.6 of~\cite{michler} by means of the \textsc{Magma} command
$$\verb"IsIsomorphic(GModule(sub<Y|V(x),V(y)>),GModule(sub<Y|W(x1),W(y1)>))"$$
we obtain the transformation 
matrix:
{\renewcommand{\arraystretch}{0.5} 
\scriptsize
$$
\mathcal T = \left( \begin{array}{*{23}{c@{\,}}c}
 1  & 0  & 0  & 0  & 0  & 0  & 0  & 0  & 0  & 0  & 0  & 0  & 0  & 0  & 0  & 0  & 0  & 0  & 0  & 0  & 0  & 0  & 0\\
 0  & 1  & 1 & 11  & 1  & 9  & 0  & 0  & 0  & 0  & 0  & 0  & 0  & 0  & 0  & 0  & 0  & 0  & 0  & 0  & 0  & 0  & 0\\
 0  & 1  & 1 & 11  & 0  & 8  & 0  & 0  & 0  & 0  & 0  & 0  & 0  & 0  & 0  & 0  & 0  & 0  & 0  & 0  & 0  & 0  & 0\\
 0  & 4  & 2  & 5  & 3 & 12  & 2  & 0  & 0  & 0  & 0  & 0  & 0  & 0  & 0  & 0  & 0  & 0  & 0  & 0  & 0  & 0  & 0\\
 0  & 3  & 1  & 7  & 2  & 3 & 11  & 0  & 0  & 0  & 0  & 0  & 0  & 0  & 0  & 0  & 0  & 0  & 0  & 0  & 0  & 0  & 0\\
 0  & 0  & 0  & 0  & 1 & 12  & 0  & 0  & 0  & 0  & 0  & 0  & 0  & 0  & 0  & 0  & 0  & 0  & 0  & 0  & 0  & 0  & 0\\
 0  & 1  & 1  & 9 & 10  & 2 & 11  & 0  & 0  & 0  & 0  & 0  & 0  & 0  & 0  & 0  & 0  & 0  & 0  & 0  & 0  & 0  & 0\\
 0  & 0  & 0  & 0  & 0  & 0  & 0  & 0 & 11  & 0  & 0  & 1  & 0  & 0  & 1  & 1  & 0  & 2 & 12  & 0 & 12  & 0 & 12\\
 0  & 0  & 0  & 0  & 0  & 0  & 0  & 0  & 0 & 11  & 0  & 3  & 2  & 0  & 0 & 11  & 2  & 2 & 11 & 12 & 11  & 3 & 12\\
 0  & 0  & 0  & 0  & 0  & 0  & 0 & 12 & 12  & 3 & 11 & 11 & 11  & 1  & 1  & 3 & 12  & 0  & 3  & 1  & 2  & 9  & 1\\
 0  & 0  & 0  & 0  & 0  & 0  & 0  & 1  & 4  & 7  & 4  & 4  & 3 & 12 & 12  & 7  & 3  & 1  & 7 & 10 & 10 & 10 & 11\\
 0  & 0  & 0  & 0  & 0  & 0  & 0  & 0  & 0  & 2 & 11 & 12  & 0  & 0  & 0  & 2 & 11  & 0  & 2  & 1  & 2 & 10  & 1\\
 0  & 0  & 0  & 0  & 0  & 0  & 0  & 0 & 11  & 2 & 12 & 12 & 11  & 1  & 1  & 2  & 0  & 1  & 2  & 1  & 0  & 9  & 0\\
 0  & 0  & 0  & 0  & 0  & 0  & 0 & 11  & 1  & 1 & 12 & 10 & 10  & 1 & 12  & 3 & 11 & 11  & 3  & 1  & 3  & 9  & 1\\
 0  & 0  & 0  & 0  & 0  & 0  & 0  & 0  & 0  & 1 & 12  & 0 & 12  & 1  & 1  & 1 & 12  & 0  & 0  & 1  & 0 & 12  & 1\\
 0  & 0  & 0  & 0  & 0  & 0  & 0  & 1 & 11  & 0  & 1  & 2  & 1  & 0  & 2  & 0  & 1  & 2 & 11 & 12 & 10  & 3 & 12\\
 0  & 0  & 0  & 0  & 0  & 0  & 0 & 12 & 12  & 2 & 12 & 10 & 10  & 0  & 0  & 4 & 11  & 0  & 3  & 1  & 2  & 9  & 1\\
 0  & 0  & 0  & 0  & 0  & 0  & 0 & 12 & 11  & 1 & 12  & 1  & 0  & 0  & 1  & 2  & 0  & 1  & 1  & 0  & 0 & 12  & 0\\
 0  & 0  & 0  & 0  & 0  & 0  & 0  & 0  & 0 & 12  & 1  & 2  & 1 & 12  & 1 & 12  & 1  & 2 & 11 & 12 & 11  & 3 & 12\\
 0  & 0  & 0  & 0  & 0  & 0  & 0  & 0  & 0  & 0  & 0  & 1  & 0  & 0  & 0  & 0  & 0  & 0  & 0 & 12  & 0  & 1  & 1\\
 0  & 0  & 0  & 0  & 0  & 0  & 0 & 12 & 12  & 2 & 10 & 12 & 12  & 0  & 0  & 2 & 11  & 0  & 3  & 1  & 2  & 9  & 1\\
 0  & 0  & 0  & 0  & 0  & 0  & 0  & 0 & 12  & 1  & 0  & 1 & 12  & 0  & 1  & 0  & 0  & 1  & 0  & 0 & 11  & 0  & 0\\
 0  & 0  & 0  & 0  & 0  & 0  & 0  & 0  & 1 & 12  & 1  & 1  & 1 & 12 & 12 & 12  & 0  & 0 & 12 & 12  & 1  & 2 & 12
 \end{array} \right)$$}
satisfying $\mathfrak V(x) = (\mathfrak W(x_1))^{\mathcal T}$ and
$\mathfrak V(y) = (\mathfrak W(y_1))^{\mathcal T}$.

Let $\mathfrak e = (\mathfrak W(e_1))^{\mathcal T}$. Then
$\mathfrak E_3 = \langle \mathfrak x, \mathfrak y, \mathfrak e
\rangle \cong E$. Let $\mathfrak G_3 = \langle \mathfrak
h,\mathfrak x, \mathfrak y, \mathfrak e \rangle$. Using the \textsc{Magma}
command $\verb"CosetAction(G, E)"$ we obtain a faithful
permutation representation of $\mathfrak G_3$ of degree $46575$
with stabilizer $\mathfrak E_3$. In particular $|\mathfrak G_3| =
2^{18}\cdot 3^6\cdot 5^3\cdot 7\cdot 11\cdot 23$.
Using the faithful permutation representation of $\mathfrak G_3$
of degree $46575$ and Kratzer's Algorithm 5.3.18 of \cite{michler}
we calculated the representatives of all the conjugacy classes of
$\mathfrak G_3$, see Table \ref{Co_2cc G_3}.
Furthermore, the character table of $\mathfrak G_3$ has been
calculated by means of the above permutation representation and
\textsc{Magma}. It coincides with the one of $\Co_2$ in 
\cite[pp.~154--155]{atlas}. 

(d) Let $\mathfrak z = (\mathfrak x \mathfrak y \mathfrak
h)^{15}$. Then $C_{\mathfrak G_2}(\mathfrak z)$ contains
$\mathfrak H = \langle \mathfrak h,\mathfrak x, \mathfrak y
\rangle$ which is isomorphic to~$H$. By the table of (c)(4)
$|C_{\mathfrak G}(\mathfrak z)| = |H|$. Hence $C_{\mathfrak
G_3}=(\mathfrak z) \cong \mathfrak H$. The character table of
$\mathfrak G_3$ implies that $\mathfrak G_3$ is a simple group.
This completes the proof.
\end{proof}

Praeger and Soicher 
give in \cite{praeger} a nice
presentation for Conway's simple group $\Co_2$ which is 
given in the statement of the following corollary.

\begin{corollary}\label{cor. presentCo_2} Keep the notation of
Theorem \ref{thm. existenceCo_2}. The finite simple group~$\mathfrak
G_3$ is isomorphic to the finitely presented group $G =
\langle a,b,c,d,e,f,g \rangle$ with set $\mathcal R(G)$ of
defining relations:
\begin{eqnarray*}
&&a^2 = b^2 = c^2 = d^2 =  e^2 = f^2 = g^2 = 1,\\
&&(ab)^3 = (bc)^5 = (cd)^3 = (df)^3 = (fe)^6 = (ae)^4 = 1,\\
&&(ec)^3 = (cf)^4 = (fg)^4 = (gb)^4 = 1,\\
&&(ac)^2 = (ad)^2 = (af)^2 = (ag)^2 = (bd)^2 = (be)^2 = 1,\\
&&(bf)^2 = (cg)^2 = (dg)^2 = (eg)^2 = 1,\\
&&a = (cf)^2, \quad e = (bg)^2, \quad b = (ef)^3,\\
&&(aecd)^4 = (baefg)^3 = (cef)^7 = 1.
\end{eqnarray*}
\end{corollary}

\begin{proof}
Using the \textsc{Magma} command \verb+FPGroupStrong(PG)+ and
the faithful permutation representation of degree~$46575$ of the
simple group 
$\mathfrak G_3$ given in Theorem \ref{thm. existenceCo_2}(c),
we obtained a finite presentation of $\mathfrak G$
with 
sixteen generators and too many relations to be stated in this
article. Therefore we 
gave in the statement the presentation of
\cite[p.~106]{praeger}. Using its subgroup $Q = \langle
a,b,c,d,e,g,(gfdc)^4 \rangle$ as a stabilizer we 
obtained a permutation
representation for their group $G = \langle a,b,c,d,e,f,g \rangle$
of degree $46575$. An isomorphism test using \textsc{Magma} 
then showed that $\mathfrak G_3 \cong G$.
\end{proof}

\section{Construction of Fischer's simple group $\Fi_{22}$}

By Lemma \ref{l. H(Fi_22)} the amalgam $H \leftarrow D \rightarrow
E_2$ constructed in 
Sections~$2$ and~$3$ satisfies the main
conditions of Step 5 of Algorithm 2.5 of  \cite{michler1}.
Therefore we can apply Algorithm 7.4.8 of \cite{michler} to give
here a new existence proof for Fischer's sporadic group
$\Fi_{22}$, see \cite{fischer}.

\begin{theorem}\label{thm. existenceFi_22}
Keep the notation of Lemma \ref{l. H(Fi_22)} and Proposition
\ref{prop. D(Fi_22)}. Using the notation of the 
three character tables \ref{Fi_22ct_H}, 
\ref{Fi_22ct_D}, and \ref{Fi_22ct_E} of the
groups $H$, 
$D$, and $E_2$, respectively, the following statements
hold:
\begin{enumerate}
\item[\rm(a)] The smallest degree of a nontrivial pair of compatible
  characters $(\chi,\tau)$ in $m\!f\mathrm{char}_{\mathbb{C}}(H)\times
  m\!f\mathrm{char}_{\mathbb{C}}(E_2)$ is~$78$. 

\item[\rm(b)] There is exactly one compatible pair $(\chi, \tau)
\in 
m\!f \mathrm{char}_{\mathbb{C}}(H) \times m\!f \mathrm{char}_{\mathbb{C}}(E_2)$
of degree $78$ of the groups
$H=\langle D, h \rangle$ and $E_2 = \langle D, e 
\rangle$, namely:
$$ 
\qquad (\chi_{4}+ \chi_{20} + \chi_{\bf 17}, \tau_{1} + \tau_{\bf 6})
$$%
with common restriction
$\tau_{|D} = \chi_{|D} = \psi_{1} + \psi_{8}+\psi_{41} +
\psi_{\bf 33}$, 
where irreducible characters with 
boldface indices denote
faithful irreducible characters.

\item[\rm(c)]Let $\mathfrak V$ and $\mathfrak W$ be the 
uniquely determined (up to isomorphism) faithful 
semisimple 
multiplicity-free $78$-dimensional modules of $H$ and $E_2$ over
$F = \GF(13)$ corresponding to 
$\chi$ and~$\tau$, respectively, where $(\chi,\tau)$
is the compatible pair.
Let $\kappa_\mathfrak V : H \rightarrow \GL_{78}(13)$ and
$\kappa_\mathfrak W : E_2 \rightarrow \GL_{78}(13)$ be the
representations of $H$ and $E_2$ afforded by the modules
$\mathfrak V$ and $\mathfrak W$, respectively.
Let $\mathfrak h = \kappa_\mathfrak V(h)$, $\mathfrak x =
\kappa_\mathfrak V(x)$, $\mathfrak y = \kappa_\mathfrak V(y)$ in $
\kappa_\mathfrak V(H) \le \GL_{78}(13)$. Then the following
assertions hold:
\begin{enumerate}
\item[\rm(1)] $\mathfrak V_{|D} \cong \mathfrak W_{|D}$, and there
is a transformation matrix $\mathcal T \in \GL_{23}(13)$ such that
$$
\mathfrak x = \mathcal T^{-1} \kappa_\mathfrak W (x_1) \mathcal T
\mbox{\ and\ }
\mathfrak y = \mathcal T^{-1} \kappa_\mathfrak W(y_1) \mathcal T.
$$
Let $\mathfrak e = \mathcal T^{-1} \kappa_\mathfrak W(e) \mathcal
T \in \GL_{78}(13)$.

\item[\rm(2)] In $\mathfrak G_2 = \langle \mathfrak h, \mathfrak
x,\mathfrak y, \mathfrak e \rangle$ the subgroup $\mathfrak E =
\langle \mathfrak x, \mathfrak y, \mathfrak e \rangle$ is the
stabilizer of a $1$-dimensional subspace ${\mathfrak U}$ of
$\mathfrak V$ such that the $\mathfrak G_2$-orbit ${\mathfrak
U}^{{\mathfrak G_2}}$ has degree~$142155$.

\item[\rm(3)] The 
four generating matrices of $\mathfrak G_2$ are
stated in Appendix~\ref{sec:genmatrixFiG2}.

\item[\rm(4)]  $\mathfrak G_2 = \langle \mathfrak h, \mathfrak x,
\mathfrak y, \mathfrak e \rangle$, and 
$\mathfrak{G}_2$ has $65$
conjugacy classes 
$\mathfrak{g}_i^{\mathfrak{G}_2}$ with
representatives 
$\mathfrak{g}_i$ and centralizer orders
$|C_{\mathfrak{G}_2}(\mathfrak{g}_i)|$
as given in Table
\ref{Fi_22cc G_2}.

\item[\rm(5)] The character table of $\mathfrak G_2$ coincides
with that of $\Fi_{22}$ in the 
Atlas \cite[pp.~156--157]{atlas}.

\end{enumerate}
\item[\rm(d)] $\mathfrak G_2$ is a finite simple group with
$2$-central involution $\kappa_\mathfrak V(z) =(\mathfrak
x\mathfrak y\mathfrak x \mathfrak h)^{10}$ such that
$C_{\mathfrak G_2} (\kappa_\mathfrak V(z)) = \kappa_\mathfrak
V(H)$, and $|\mathfrak G_2| = 2^{17}\cdot 3^9\cdot 5^2\cdot
7\cdot 11 \cdot 
13$. 
\end{enumerate}
\end{theorem}
\begin{proof}
(a) The character tables of the groups $H$, 
$D$, and $E_2$ are
stated in the 
Appendix; we will follow the notation used there. 
Using the character tables of $H$, $D$, and $E_2$ and the fusion
of the classes of $D(H)$ in $H$ and $D(E_2)$ in $E_2$ an
application of Kratzer's Algorithm 7.3.10 of \cite{michler} yields
the compatible pair stated in assertion (a).

(b) 
Kratzer's Algorithm also shows that the pair $(\chi,\tau)$ of (b) is the unique
compatible pair of degree $78$ with respect to the fusion of the
$D$-classes into the $H$- and 
$E_2$-classes.

(c) In order to construct the 
semisimple faithful representation
$\mathfrak V$ corresponding to the character $\chi = \chi_4 +
\chi_{20} + 
\chi_{\bf 17}$ we determine the irreducible constituents
of the faithful permutation representation $(1_T)^H$ of $H$ with
stabilizer $T = \langle h_1,h_2,h_3,h_4 \rangle$ given in Lemma
\ref{l. H(Fi_22)}. Calculating inner products with the irreducible
characters of $H$ it follows $(1_T)^H$ has 
seven irreducible
constituents of degrees $1$, $32$, $40$, $120$, $135$, 
$216$, and $480$. By the character table of $H$ we know that
$\chi_{\bf 17}$ is
the irreducible character of $H$ of degree $32$ occurring in the
permutation character $(1_T)^H$. Constructing the permutation
matrices of the 
three generators $h$, 
$x$, and $y$ of $H$ over the
field $K = \GF(13)$ and applying the Meataxe 
Algorithm implemented
in \textsc{Magma} we obtain the matrices $\mathcal X_2, \mathcal Y_2,
\mathcal H_2 \in \GL_{32}(13)$ of the generators of $H$ in the
$32$-dimensional representation $\mathfrak V_1$. These 
three 
matrices are 
given in Appendix~\ref{sec:genmatrixFiG2}. Furthermore, it has been
checked that $\mathfrak V_1$ restricted to $D(H) = \langle x,y
\rangle$ is irreducible.

In order to find the irreducible modules $\mathfrak V_2$ and
$\mathfrak V_3$ corresponding to the characters $\chi_4$ of degree
$6$ and $\chi_{20}$ of degree $40$ we applied Theorem 2.5.4 of
\cite{michler} to the faithful character 
$\chi_{\bf 17}$. It follows
that both missing representations are constituents of the 
$6$th tensor power of $\mathfrak V_1$. A character calculation shows
that the irreducible representations $\mathfrak V_2$ corresponding
to $\chi_{20}$ and $\mathfrak V_3$ corresponding to $\chi_4$ can
be constructed by the Meataxe decompositions of the
representations of the tensor products belonging to the following
character products:
the $135$-dimensional irreducible character $\chi_{42}$ occurs in
$\chi_{\bf 17}^2$. 
Their tensor product $\chi_{42}\otimes
\chi_{\bf 17}$ contains the $640$-dimensional character $\chi_{80}$,
and $\chi_{59}$ of degree $270$ is an irreducible constituent of
$\chi_{80}\otimes 
\chi_{\bf 17}$. 
Their tensor product $\chi_{59}\otimes
\chi_{\bf 17}$ has the $192$-dimensional constituent~$\chi_{47}$. The
desired character $\chi_4$ of degree $6$ is then an irreducible
composition factor of~$\chi_{47}\otimes 
\chi_{\bf 17}$. 

The faithful $480$-dimensional irreducible character $\chi_{\bf 74}$
occurs in $\chi_{\bf 17}^3$. 
The tensor product 
$\chi_{\bf 74}\otimes\chi_{\bf 17}$ 
contains the $135$-dimensional
character $\chi_{41}$, and $\chi_{45}$ of degree $160$ is an
irreducible constituent of $\chi_{41}\otimes 
\chi_{\bf 17}$. 
The tensor product $\chi_{45}\otimes 
\chi_{\bf 17}$ has the 
desired character $\chi_{20}$ of degree $40$ as an irreducible composition
factor.

An application of the Meataxe 
Algorithm implemented in \textsc{Magma}
produces then the 
six matrices of the three generators of $H$
with respect to fixed bases in $\mathfrak V_2$ and $\mathfrak V_3$. Their
$46$-dimensional diagonal joins $\mathcal H_1$, 
$\mathcal X_1$, and
$\mathcal Y_1$ are 
given in Appendix~\ref{sec:genmatrixFiG2}.

Now we construct the faithful representation $\mathfrak W$ of $E_2$
corresponding to the character $\tau = \tau_{1} + \tau_{\bf
6}$. Clearly $\tau_{1}$ corresponds to the trivial representation
$\mathfrak W_1$ of~$E_2$. By means of the character table of $E_2$ we
see that the irreducible representation $\mathfrak W_2$ of~$\tau_{\bf
6}$ occurs as an irreducible constituent of the faithful permutation
representation $(1_{T_1})^{E_2}$ of $E_2$ of degree $1024$ with
stabilizer $\M_{22}$ constructed in
Lemma~\ref{l. M22-extensions}. Using the Meataxe 
Algorithm implemented
in \textsc{Magma} the $77$-dimensional irreducible constituent
$\mathfrak W_2$ of $(1_{T_1})^E$ over $K = \GF(13)$ has been
determined. Its restriction $(\mathfrak W_2)_{|D}$ to $D = \langle
x_1,y_1 \rangle$ decomposes into 
three irreducible constituents of
dimensions $5$, 
$40$, and $32$. Taking a basis $B_1$ of the
$1$-dimensional restriction $(\mathfrak W_1)_{|D}$ to $D$ and bases
$B_2$, $B_3$, and $B_4$ in the 
three irreducible constituents of
$(\mathfrak W_2)_{|D}$ we obtain the 
two matrices $\mathfrak W(x_1)$
and $\mathfrak W(y_1)$ of the generators $x_1$ and $y_1$ of $D =
D_{E_2}$ in $\GL_{78}(13)$. Let $B = B_1 \cup B_2 \cup B_3 \cup
B_4$. Then $B$ is a basis of the 
semisimple $KE_2$-module $\mathfrak
W$. Let $\mathfrak W(e_1)$ be the matrix of the third generator $e_1$
of $E_2$ 
with respect to~$B$.
The two $KD$-modules $\mathfrak V_{|D}$ and $\mathfrak W_{|D}$
described by the pairs $(\mathfrak V(x),\mathfrak V(y))$ and
$(\mathfrak W(x_1),\mathfrak W(y_1))$ of matrices in
$\GL_{78}(13)$ are isomorphic by construction. Let $Y = 
\GL_{78}(13)$. Applying then Parker's isomorphism test of
Proposition 6.1.6 of~\cite{michler} by means of the \textsc{Magma} command
$$\verb"IsIsomorphic(GModule(sub<Y|V(x),V(y)>),GModule(sub<Y|W(x1),W(y1)>))"$$
we obtain the transformation matrix $\mathcal T$ satisfying
$\mathfrak V(x) = (\mathfrak W(x_1))^{\mathcal T}$ and $\mathfrak
V(y) = (\mathfrak W(y_1))^{\mathcal T}$. In view of its size this
transformation matrix is decomposed into 
four block matrices
$\mathcal A \in M\!at_{38,38}$, $\mathcal B \in M\!at_{38,40}$,
$\mathcal C \in M\!at_{40,38}$ and, $\mathcal D \in M\!at_{40,40}$ 
such that:
\[ \mathcal{T} = \left(\begin{array}{cc}
\mathcal{A} & \mathcal{B}\\
\mathcal{C}&\mathcal{D}
\end{array}\right)\]
has the following 
four blocks:
\scriptsize

\noindent Matrix $\mathcal A$ :%
{\renewcommand{\arraystretch}{0.5}%
$$
\begin{array}{*{38}{c@{\,}}c}
12 & . & . & . & . & . & . & . & . & . & . & . & . & . & . & . & .
& . & . & . &
. & . & . & . & . & . & . & . & . & . & . & . & . & . & . & . & . & . \\
. & 11 & 2 & 9 & 10 & 1 & . & . & . & . & . & . & . & . & . & . &
. & . & . & .
& . & . & . & . & . & . & . & . & . & . & . & . & . & . & . & . & . & . \\
. & 6 & 12 & 3 & 3 & 7 & . & . & . & . & . & . & . & . & . & . & .
& . & . & . &
. & . & . & . & . & . & . & . & . & . & . & . & . & . & . & . & . & . \\
. & 11 & 10 & 5 & 12 & 4 & . & . & . & . & . & . & . & . & . & . &
. & . & . & .
& . & . & . & . & . & . & . & . & . & . & . & . & . & . & . & . & . & . \\
. & 11 & 7 & 12 & 6 & 11 & . & . & . & . & . & . & . & . & . & . &
. & . & . & .
& . & . & . & . & . & . & . & . & . & . & . & . & . & . & . & . & . & . \\
. & 11 & 5 & 9 & 10 & 1 & . & . & . & . & . & . & . & . & . & . &
. & . & . & .
& . & . & . & . & . & . & . & . & . & . & . & . & . & . & . & . & . & . \\
. & . & . & . & . & . & 2 & 9 & 2 & 5 & 2 & 3 & 5 & 6 & 4 & 12 & 7
& 7 & 6 & 2 &
8 & 8 & 10 & 4 & 11 & 9 & 1 & 6 & 9 & 6 & 1 & 12 & . & . & 10 & 8 & 8 & 1 \\
. & . & . & . & . & . & . & 8 & 7 & . & 5 & 7 & 8 & 3 & 3 & 9 & 5
& 11 & 9 & 1 &
. & 12 & 8 & 1 & 1 & . & . & 11 & 3 & 7 & 6 & 1 & 3 & 12 & 8 & 12 & 2 & 8 \\
\end{array}
$$
\noindent Matrix $\mathcal B$:
$$
\begin{array}{*{40}{c@{\,}}c}
. & . & . & . & . & . & . & . & . & . & . & . & . & . & . & . & .
& . & . & . &
. & . & . & . & . & . & . & . & . & . & . & . & . & . & . & . & . & . & . & . \\
. & . & . & . & . & . & . & . & . & . & . & . & . & . & . & . & .
& . & . & . &
. & . & . & . & . & . & . & . & . & . & . & . & . & . & . & . & . & . & . & . \\
. & . & . & . & . & . & . & . & . & . & . & . & . & . & . & . & .
& . & . & . &
. & . & . & . & . & . & . & . & . & . & . & . & . & . & . & . & . & . & . & . \\
. & . & . & . & . & . & . & . & . & . & . & . & . & . & . & . & .
& . & . & . &
. & . & . & . & . & . & . & . & . & . & . & . & . & . & . & . & . & . & . & . \\
. & . & . & . & . & . & . & . & . & . & . & . & . & . & . & . & .
& . & . & . &
. & . & . & . & . & . & . & . & . & . & . & . & . & . & . & . & . & . & . & . \\
. & . & . & . & . & . & . & . & . & . & . & . & . & . & . & . & .
& . & . & . &
. & . & . & . & . & . & . & . & . & . & . & . & . & . & . & . & . & . & . & . \\
9 & 12 & 6 & 8 & 2 & 5 & 4 & 10 & . & . & . & . & . & . & . & . &
. & . & . & . & . & . & . & . & . & . & . & . & . & . & . & . & .
& . & . & . & . & . & . & .
\\
6 & 11 & 4 & 6 & 1 & 8 & 8 & 10 & . & . & . & . & . & . & . & . &
. & . & . & . & . & . & . & . & . & . & . & . & . & . & . & . & .
& . & . & . & . & . & . & .
\\
\end{array}
$$
\noindent Matrix $\mathcal C$:
$$
\begin{array}{*{38}{c@{\,}}c}
. & . & . & . & . & . & . & 3 & 3 & 6 & 7 & 11 & 6 & 7 & 7 & 1 & 8
& 7 & 3 & 11
& 8 & 2 & 10 & 11 & 4 & . & 2 & . & 6 & 6 & 5 & . & 9 & . & 7 & 8 & 2 & 6 \\
. & . & . & . & . & . & 12 & 11 & . & 6 & 8 & 2 & 6 & 9 & 1 & 8 &
6 & 6 & 4 & 1 & 11 & 7 & 4 & 7 & 10 & 9 & 6 & . & 11 & 12 & 8 & .
& 12 & . & 10 & 12 & 2 & 7
\\
. & . & . & . & . & . & . & 8 & 5 & 2 & 9 & 1 & 6 & 7 & 12 & 10 &
9 & 1 & 10 &
12 & 6 & 7 & 2 & 5 & 6 & 5 & 2 & 8 & . & 2 & 4 & 10 & 9 & 5 & 1 & 3 & 11 & 2 \\
. & . & . & . & . & . & . & 3 & 1 & . & . & 9 & 7 & 6 & 11 & 12 &
7 & 1 & 3 & 5
& 4 & 2 & 7 & 2 & 8 & 1 & 9 & 7 & . & 8 & 3 & . & 9 & 1 & 12 & 9 & 5 & 1 \\
. & . & . & . & . & . & . & 8 & 12 & 12 & 7 & 5 & 3 & 10 & 6 & 8 &
5 & 11 & 6 &
8 & 8 & 12 & 3 & 4 & 9 & 2 & 1 & 6 & 3 & 7 & 11 & 10 & 7 & 6 & 8 & 7 & 6 & 5 \\
. & . & . & . & . & . & . & 8 & 9 & 6 & 9 & 6 & 12 & . & 1 & 1 & .
& 9 & 7 & . &
8 & 6 & 12 & . & 6 & 9 & 9 & 1 & 7 & 8 & 5 & 9 & . & 8 & 2 & 2 & 3 & 7 \\
. & . & . & . & . & . & . & 10 & 12 & 9 & 1 & 12 & 2 & 5 & 5 & 1 &
4 & 1 & 9 & 3
& 9 & 7 & 1 & 10 & 4 & 11 & 5 & . & 6 & . & 1 & 2 & 5 & 4 & . & 4 & 5 & 7 \\
. & . & . & . & . & . & . & . & . & . & . & . & . & . & . & . & .
& . & . & . &
. & 9 & 3 & 3 & 7 & 8 & 5 & 2 & 8 & 10 & 3 & . & 8 & 10 & 9 & 9 & 12 & 10 \\
\end{array}
$$
\noindent Matrix $\mathcal D$ :
$$
\begin{array}{*{40}{c@{\,}}c}
10 & 4 & 9 & 3 & 4 & . & 3 & 11 & . & . & . & . & . & . & . & . &
. & . & . & . & . & . & . & . & . & . & . & . & . & . & . & . & .
& . & . & . & . & . & . & .
\\
1 & 1 & 5 & 7 & 12 & 3 & 5 & 12 & . & . & . & . & . & . & . & . &
. & . & . & . & . & . & . & . & . & . & . & . & . & . & . & . & .
& . & . & . & . & . & . & .
\\
7 & 1 & 9 & 4 & 6 & 7 & 7 & 1 & . & . & . & . & . & . & . & . & .
& . & . & . &
. & . & . & . & . & . & . & . & . & . & . & . & . & . & . & . & . & . & . & . \\
12 & 1 & 12 & . & 9 & 3 & 2 & 12 & . & . & . & . & . & . & . & . &
. & . & . & . & . & . & . & . & . & . & . & . & . & . & . & . & .
& . & . & . & . & . & . & .
\\
. & 8 & 3 & 4 & 2 & 9 & 11 & 4 & . & . & . & . & . & . & . & . & .
& . & . & . &
. & . & . & . & . & . & . & . & . & . & . & . & . & . & . & . & . & . & . & . \\
6 & 1 & 8 & 7 & 2 & 3 & 2 & 11 & . & . & . & . & . & . & . & . & .
& . & . & . &
. & . & . & . & . & . & . & . & . & . & . & . & . & . & . & . & . & . & . & . \\
2 & 8 & 1 & 3 & 6 & 2 & 7 & 9 & . & . & . & . & . & . & . & . & .
& . & . & . &
. & . & . & . & . & . & . & . & . & . & . & . & . & . & . & . & . & . & . & . \\
5 & 5 & 6 & 10 & 2 & 3 & 7 & 6 & . & . & . & . & . & . & . & . & .
& . & . & . &
. & . & . & . & . & . & . & . & . & . & . & . & . & . & . & . & . & . & . & . \\
\end{array}
$$}%
\normalsize%
Let $\mathfrak e = (\mathfrak W(e_1))^{\mathcal T}$. Then
$\mathfrak E_2 = \langle \mathfrak x, \mathfrak y, \mathfrak e
\rangle \cong E_2$. Let $\mathfrak G_2 = \langle \mathfrak
h,\mathfrak x, \mathfrak y, \mathfrak e \rangle$. The 
four 
generating matrices of $\mathfrak G_2$ are 
given in Appendix~\ref{sec:genmatrixFiG2}.
Using the \textsc{Magma} command $\verb"CosetAction(G, E)"$ we obtain a
faithful permutation representation of $\mathfrak G_2$ of degree
$142155$ with stabilizer $\mathfrak E_2$. In particular $|\mathfrak
G_2| = 2^{17}\cdot 3^9\cdot 5^2\cdot 7\cdot 11\cdot 13$.

Using the faithful permutation representation of $\mathfrak G_2$
of degree $142155$ and Kratzer's Algorithm 5.3.18 of
\cite{michler} we calculated the representatives of the $65$
conjugacy classes of $\mathfrak G_2$, see Table \ref{Fi_22cc G_2}.
Furthermore, the character table of $\mathfrak G_2$ has been
calculated by means of the above permutation representation and
\textsc{Magma}. It coincides with the one of $\Fi_{22}$ in 
\cite[pp.~156--157]{atlas}. 

(d) Let $\mathfrak z = 
(\mathfrak{xyxh})^{10}$. Then 
$C_{\mathfrak{G}_2}(\mathfrak z)$ contains $\mathfrak H = \langle \mathfrak
h,\mathfrak x, \mathfrak y \rangle$ which is isomorphic to $H$.
Now 
$|C_{\mathfrak{G}_2}(\mathfrak z)| = |H|$ by Table \ref{Fi_22cc
G_2}. Hence 
$C_{\mathfrak{C}_2}(\mathfrak z) \cong \mathfrak H$.
The character table of $\mathfrak G_2$ implies that $\mathfrak
G_2$ is a simple group. This completes the proof.
\end{proof}

Praeger and Soicher have given in \cite{praeger} a nice
presentation for Fischer's
simple group $\Fi_{22}$ which 
we now also verify for $\mathfrak{G}_2$. 

\begin{corollary}\label{cor. presentFi_22} Keep the notation of
Theorem \ref{thm. existenceFi_22}. The finite simple group~$\mathfrak
G_2$ is isomorphic to the finitely presented group $G =
\langle a,b,c,d,e,f,g,h,i \rangle$ with the following set
$\mathcal R(G)$ of defining relations:
\begin{eqnarray*}
&&a^2 = b^2 = c^2 = d^2 =  e^2 = f^2 = g^2 = h^2 = i^2 = 1,\\
&&(ab)^3 = (bc)^3 = (cd)^3 = (de)^3 = (ef)^3 = (fg)^3 = (dh)^3 = (hi)^3 = 1,\\
&&(ac)^2 = (ad)^2 = (ae)^2 = (af)^2 = (ag)^2 = (ah)^2 = (ai)^2 = 1,\\
&&(bd)^2 = (be)^2 = (bf)^2 = (bg)^2 = (bh)^2 = (bi)^2 = 1,\\
&&(ce)^2 = (cf)^2 = (cg)^2 = (ch)^2 = (ci)^2 = (df)^2 = (dg)^2 = (di)^2 = 1,\\
&&(eg)^2 = (eh)^2 = (ei)^2 = (fh)^2 = (fi)^2 = (gh)^2 = (gi)^2 = 1,\\
&&(dcbdefdhi)^{10} = (abcdefh)^9 = (bcdefgh)^9 = 1. 
\end{eqnarray*}
\end{corollary}
\begin{proof}
Using the \textsc{Magma} command \verb+FPGroupStrong(PG)+ and the
faithful permutation representation of degree~$142155$ of the simple group
$\mathfrak{G}_2$ given in Theorem~\ref{thm. existenceFi_22}(c) 
we obtained a finite presentation of $\mathfrak G_2$ with 16
generators and too many relations to be stated in this article.
Therefore we stated in the assertion the presentation of
\cite[p. 110]{praeger}. Using its subgroup 
\[ Q = \bigl\langle a,c,e,g,h,bacb,dced,fegf,dchd,dehd,(cdehi)^4
\bigr\rangle\] 
as a stabilizer we 
obtained a faithful permutation representation of degree
$142155$ for the group $G=\langle a,b,c,d,e,f,g\rangle$. An
isomorphism test in \textsc{Magma} established that
$\mathfrak{G}_2\cong G$.
\end{proof}

\section{The remaining cases: $E_1$, $E_4$, and $E_5$}

In this section we summarize our results on the remaining cases.
Using Theorem~\ref{thm. existenceFi_22} we prove that the
application of Algorithm 2.5 of  \cite{michler1} to the extension
group~$E_4$ constructs the automorphism group 
${\rm Aut}(\Fi_{22})$. In
the cases of $E_1$ and $E_5$ the algorithm terminates without
success.

\begin{corollary}\label{cor. D(Aut(Fi_22))} Keep the notation of Lemma
\ref{l. Aut(M22)-extensions}. Let $E_4 = \langle
p_1,q_1,v_1\rangle$ be the split extension of $A_{22} =
{\rm Aut}(\M_{22})$ by its simple module $V_4$ of dimension $10$ 
over $F = 
\GF(2)$. Let $\mathfrak A_2$ be the automorphism group of the
simple group $\mathfrak G_2$, and let $\mathfrak H = C_{\mathfrak
G_2}(\mathfrak z)$ be the centralizer of a $2$-central involution
$\mathfrak z$ of $\mathfrak G_2$. Then there is a unique maximal
elementary abelian normal subgroup $\mathfrak B$ of a Sylow 
$2$-subgroup $\mathfrak S$ of~$\mathfrak H$ such that $N_{\mathfrak
G_2}(\mathfrak B) \cong E_4$.
\end{corollary}

\begin{proof}
In order to simplify the notation in the proof we replace Gothic
letters by Roman letters and $\mathfrak G_2$, $\mathfrak A_2$ by
$G$, $A$, respectively. Then $G = \langle x,y,h,e \rangle$ by
Theorem \ref{thm. existenceFi_22} and $t = (xye)^7$ is an
involution of $G$ with a centralizer $C_G(t)$ of index $|G :
C_G(t)| = 3510$ by Table \ref{Fi_22cc G_2}. Since $G$ is simple it
has a faithful permutation representation $PG$ of degree $3510$.
Using 
it, \textsc{Magma} has been able to calculate the automorphism group
$A$ of $G$ and its order $|A| = 2^{18}\cdot 3^9\cdot 5^2\cdot
7\cdot 11\cdot 13$. Using the command
$\verb"PermutationRepresentation(AutG)"$,
\textsc{Magma} produces a faithful
permutation representation $PA$ of $A$ of degree $3510$. Let $PU$
be the derived subgroup of $PA$. Then \textsc{Magma} 
verifies that $PU \cong
PG$ and that $|PA : PU| = 2$. Let $PS$ be a Sylow $2$-subgroup of
$PU$ and $pz$ an involution in the center of $PS$. Let $PH =
C_{PU}(pz)$. An isomorphism test shows that $PH \cong C_G(z)$
where $z$ is a $2$-central involution of the simple group $G$. Let
$NH = N_{PA}(PH)$. Then $|NH : PH| = 2$. 
Furthermore, $PH$ has a complement $PC = \langle pu\rangle$ of order
$2$ in~$NH$, as has been checked
by means of \textsc{Magma}. Now Sylow's Theorem asserts that $pu$
can be chosen so that $PS = PS^{pu}$. Another application of \textsc{Magma}
shows that $PS$ has a unique elementary abelian normal subgroup
$PB$ of order $2^{10}$. Since $PB$ is a characteristic subgroup of
$PS$ it is normalized by $pu$. Furthermore, $PA = \langle PU, pu
\rangle$. \textsc{Magma} 
also shows that $PD = N_{PA}(PB)$ has order
$2^{18}\cdot 3\cdot 5$. Another isomorphism test verifies that $PD
\cong C_{E_4}(y)$ where $y$ is the $2$-central involution $y =
(p_1^2q_1v)^{10}$ of $E_4$ given in Lemma \ref{l. classes}.
Therefore Algorithm 2.5 of  \cite{michler1} and Theorem \ref{thm.
existenceFi_22} imply that the application of that algorithm to
the extension group $E_4$ returns the automorphism group $A$ of
the simple group $G$.
\end{proof}

\begin{remark} Let $E_1 = V_1\rtimes \M_{22} = \langle a,b,c,d,t,g,h,i,v_1\rangle$ be
the split extension of $\M_{22}$ by its simple module $V_1$ of
dimension $10$ over $F = 
\GF(2)$ defined in
Lemma~\ref{l. M22-extensions}. By Lemma \ref{l. classes} $E_1$ has a
unique class of $2$-central involutions represented by $z = (tv_1)^3$.
Its centralizer $D = C_{E_1}(z)$ has a uniquely determined 
nonabelian
normal subgroup $Q$ of order $512$ such that $V = Q/Z(Q)$ is
elementary abelian of order $2^8$, where $Z(Q) = \langle z \rangle$
denotes the center of $Q$. Furthermore, $Q$ has a complement $W$ in
$D$. The Fitting subgroup $B$ of $W$ has order $16$ and its complement
$L$ in $W$ is isomorphic to the alternating group $A_6$ and not to
$S_5$ as in Proposition \ref{prop. D(Fi_22)}. Since the center of $W$
has order $2$ we applied Algorithm 7.4.8 of~\cite{michler} to
construct a subgroup $K$ of $\GL_8(2)$ such that $|K : W|$ is
odd. However, this application was not successful.
\end{remark}

\begin{remark} Let 
$E_5$ be the 
non-split extension of $A_{22}$ by its simple module $V_3$ of
dimension $10$ over $F = 
\GF(2)$ defined in Lemma \ref{l.
Aut(M22)-extensions}. By 
Lemma~\ref{l. classes}, $E_5$ has a unique
class of $2$-central involutions represented by $z = p_2^4$. Its
centralizer $D = C_{E_5}(z)$ has a uniquely determined 
nonabelian
normal subgroup $Q$ of order~$512$ such that $V = Q/Z(Q)$ is
elementary abelian of order $2^8$, where $Z(Q) = \langle z
\rangle$ denotes the center of $Q$. This time $Q$ does not have a
complement in $D$. The Fitting subgroup $B$ of $W = D/Q$ has order
$32$ and its factor group $L = W/B$ is isomorphic to the symmetric
group $S_6$. Since the center of $W$ has order $2$ we applied
Algorithm 7.4.8 of \cite{michler} to construct a subgroup $K$ of
$\GL_8(2)$ such that $|K : W|$ is odd. However, that application
was not successful.
\end{remark}

\filbreak
\begin{appendix}


\section{Representatives of conjugacy classes}\label{CFcc}

\begin{cclass}\label{Co_2cc H} Conjugacy classes of $H(\Co_2) = \langle x, y, h \rangle$

\bigskip
{ \setlength{\arraycolsep}{1mm}
\renewcommand{\baselinestretch}{0.5}
{\scriptsize
 $$
\begin{array}{|c|r|r|r|c|c|c|c|}
\hline
\rule[-3pt]{0pt}{10pt} \mbox{Class} &
 \multicolumn{1}{c|}{\mbox{Representative}} &
 \multicolumn{1}{c|}{|\mbox{Class}|} &
 \multicolumn{1}{c|}{|\mbox{Centralizer}|}
 &
 2{\rm P}
 & 3{\rm P}
 & 5{\rm P}
 & 7{\rm P}
 \\\hline
\rule[-1pt]{0pt}{8pt} 1 & 1 & 1 & 2^{18}\cdot3^{4}\cdot5\cdot7
 & 1 & 1 & 1 & 1 \\\hline

\rule[-1pt]{0pt}{8pt} 2_{1} & (xyh)^{15} & 1 & 2^{18}\cdot3^{4}\cdot5\cdot7
 & 1 & 2_{1} & 2_{1} & 2_{1} \\\hline

\rule[-1pt]{0pt}{8pt} 2_{2} & (yh)^{7} & 270 & 2^{17}\cdot3\cdot7
 & 1 & 2_{2} & 2_{2} & 2_{2} \\\hline

\rule[-1pt]{0pt}{8pt} 2_{3} & (y)^{3} & 1008 & 2^{14}\cdot3^{2}\cdot5
 & 1 & 2_{3} & 2_{3} & 2_{3} \\\hline

\rule[-1pt]{0pt}{8pt} 2_{4} & (xyhy)^{5} & 1008 & 2^{14}\cdot3^{2}\cdot5
 & 1 & 2_{4} & 2_{4} & 2_{4} \\\hline

\rule[-1pt]{0pt}{8pt} 2_{5} & (xy)^{4} & 1260 & 2^{16}\cdot3^{2}
 & 1 & 2_{5} & 2_{5} & 2_{5} \\\hline

\rule[-1pt]{0pt}{8pt} 2_{6} & (xyxh^{2})^{6} & 1260 & 2^{16}\cdot3^{2}
 & 1 & 2_{6} & 2_{6} & 2_{6} \\
\hline
 \end{array}
 $$}
%
\filbreak
%
\noindent  Conjugacy classes of $H(\Co_2) = \langle x, y, h \rangle$ (continued)

\bigskip
{  \scriptsize%
 $$
 \begin{array}{|c|r|r|r|c|c|c|c|
}\hline
\rule[-3pt]{0pt}{10pt}\mbox{Class} &
 \multicolumn{1}{c|}{\mbox{Representative}} &
 \multicolumn{1}{c|}{|\mbox{Class}|} &
 \multicolumn{1}{c|}{|\mbox{Centralizer}|}
 &
 2{\rm P}
 & 3{\rm P}
 & 5{\rm P}
 & 7{\rm P}
 \\\hline
\rule[-1pt]{0pt}{8pt} 8_{1} & (xy^{3}h)^{3} & 967680 & 2^{8}\cdot3
 & 4_{1} & 8_{1} & 8_{1} & 8_{1} \\\hline

\rule[-1pt]{0pt}{8pt} 8_{2} & (x^{2}yh)^{2} & 1451520 & 2^{9}
 & 4_{2} & 8_{2} & 8_{2} & 8_{2} \\\hline

\rule[-1pt]{0pt}{8pt} 8_{3} & (xhyh)^{2} & 1451520 & 2^{9}
 & 4_{2} & 8_{3} & 8_{3} & 8_{3} \\\hline

\rule[-1pt]{0pt}{8pt} 8_{4} & x^{4}y^{3} & 2903040 & 2^{8}
 & 4_{8} & 8_{4} & 8_{4} & 8_{4} \\\hline

\rule[-1pt]{0pt}{8pt} 8_{5} & x^{2}h^{2}xy^{2} & 2903040 & 2^{8}
 & 4_{8} & 8_{5} & 8_{5} & 8_{5} \\\hline

\rule[-1pt]{0pt}{8pt} 8_{6} & x^{2}y & 5806080 & 2^{7}
 & 4_{9} & 8_{6} & 8_{6} & 8_{6} \\\hline

\rule[-1pt]{0pt}{8pt} 8_{7} & x^{3}yhxy & 5806080 & 2^{7}
 & 4_{15} & 8_{7} & 8_{7} & 8_{7} \\
\hline

\end{array}
 $$}}
\end{cclass}

\filbreak
\begin{cclass}\label{Co_2cc D} 
\noindent Conjugacy classes of $D(\Co_2) = \langle x, y \rangle$

\bigskip

{ \setlength{\arraycolsep}{1mm}
\renewcommand{\baselinestretch}{0.5}
{\scriptsize
 $$
 \begin{array}{|c|r|r|r|c|c|c|
}\hline
\rule[-3pt]{0pt}{10pt}\mbox{Class} &
 \multicolumn{1}{c|}{\mbox{Representative}} &
 \multicolumn{1}{c|}{|\mbox{Class}|} &
 \multicolumn{1}{c|}{|\mbox{Centralizer}|}
 &
 2{\rm P}
 & 3{\rm P}
 & 5{\rm P}
 \\\hline
\rule[-1pt]{0pt}{8pt} 1 & 1 & 1 & 2^{18}\cdot3^{2}\cdot5
 & 1 & 1 & 1 \\\hline

\rule[-1pt]{0pt}{8pt} 2_{1} & (x^{3}y)^{6} & 1 & 2^{18}\cdot3^{2}\cdot5
 & 1 & 2_{1} & 2_{1} \\\hline

\rule[-1pt]{0pt}{8pt} 2_{2} & (x^{5}y)^{5} & 16 & 2^{14}\cdot3^{2}\cdot5
 & 1 & 2_{2} & 2_{2} \\\hline

\rule[-1pt]{0pt}{8pt} 2_{3} & (x^{4}yxy^{2})^{5} & 16 & 2^{14}\cdot3^{2}\cdot5
 & 1 & 2_{3} & 2_{3} \\\hline

\rule[-1pt]{0pt}{8pt} 2_{4} & (x^{4}y^{3})^{4} & 30 & 2^{17}\cdot3
 & 1 & 2_{4} & 2_{4} \\\hline

\rule[-1pt]{0pt}{8pt} 2_{5} & (xy)^{4} & 60 & 2^{16}\cdot3
 & 1 & 2_{5} & 2_{5} \\\hline

\rule[-1pt]{0pt}{8pt} 2_{6} & (x^{4}yxy)^{4} & 60 & 2^{16}\cdot3
 & 1 & 2_{6} & 2_{6} \\
\hline
 \end{array}
 $$}
\filbreak

\noindent Conjugacy classes of $D(\Co_2) = \langle x, y \rangle$
 (continued)

\bigskip
{ \scriptsize
 $$
 \begin{array}{|c|r|r|r|c|c|c|
}\hline
\rule[-3pt]{0pt}{10pt} \mbox{Class} &
 \multicolumn{1}{c|}{\mbox{Representative}} &
 \multicolumn{1}{c|}{|\mbox{Class}|} &
 \multicolumn{1}{c|}{|\mbox{Centralizer}|}
 &
 2{\rm P}
 & 3{\rm P}
 & 5{\rm P}
 \\\hline

\rule[-1pt]{0pt}{8pt} 4_{31} & x^{4}y^{4} & 11520 & 2^{10}
 & 2_{8} & 4_{31} & 4_{31} \\\hline

\rule[-1pt]{0pt}{8pt} 4_{32} & (x^{3}y^{2}xy^{3})^{2} & 11520 & 2^{10}
 & 2_{8} & 4_{32} & 4_{32} \\\hline

\rule[-1pt]{0pt}{8pt} 4_{33} & x^{2}y^{2}xy^{4} & 11520 & 2^{10}
 & 2_{5} & 4_{33} & 4_{33} \\\hline

\rule[-1pt]{0pt}{8pt} 4_{34} & x^{2}y^{2}x^{2}y^{5} & 11520 & 2^{10}
 & 2_{8} & 4_{34} & 4_{34} \\\hline

\rule[-1pt]{0pt}{8pt} 4_{35} & x^{4}yx^{4}yxy & 11520 & 2^{10}
 & 2_{5} & 4_{35} & 4_{35} \\\hline

\rule[-1pt]{0pt}{8pt} 4_{36} & x^{4}yxyx^{2}yxy & 11520 & 2^{10}
 & 2_{5} & 4_{36} & 4_{36} \\\hline

\rule[-1pt]{0pt}{8pt} 4_{37} & x^{2}yxy^{2}x^{2}y^{2}xy & 11520 & 2^{10}
 & 2_{4} & 4_{37} & 4_{37} \\
\hline
 \end{array}
 $$}
\filbreak
%
\noindent Conjugacy classes of $D(\Co_2) = \langle x, y \rangle$
 (continued)

\bigskip
{ \scriptsize
 $$
 \begin{array}{|c|r|r|r|c|c|c|
}\hline
\rule[-3pt]{0pt}{10pt} \mbox{Class} &
 \multicolumn{1}{c|}{\mbox{Representative}} &
 \multicolumn{1}{c|}{|\mbox{Class}|} &
 \multicolumn{1}{c|}{|\mbox{Centralizer}|}
 &
 2{\rm P}
 & 3{\rm P}
 & 5{\rm P}
 \\\hline

\rule[-1pt]{0pt}{8pt} 8_{1} & (x^{3}y^{2}xyxy)^{3} & 15360 & 2^{8}\cdot3
 & 4_{1} & 8_{1} & 8_{1} \\\hline

\rule[-1pt]{0pt}{8pt} 8_{2} & (x^{2}yxy^{2})^{2} & 23040 & 2^{9}
 & 4_{3} & 8_{2} & 8_{2} \\\hline

\rule[-1pt]{0pt}{8pt} 8_{3} & (x^{3}yxy^{2})^{2} & 23040 & 2^{9}
 & 4_{3} & 8_{3} & 8_{3} \\\hline

\rule[-1pt]{0pt}{8pt} 8_{4} & x^{3}yx^{3}y^{2} & 46080 & 2^{8}
 & 4_{6} & 8_{4} & 8_{4} \\\hline

\rule[-1pt]{0pt}{8pt} 8_{5} & x^{3}yx^{2}yxyxy^{2} & 46080 & 2^{8}
 & 4_{3} & 8_{5} & 8_{5} \\\hline

\rule[-1pt]{0pt}{8pt} 8_{6} & x^{2}yx^{2}y^{2}x^{2}y^{3} & 46080 & 2^{8}
 & 4_{1} & 8_{6} & 8_{6} \\\hline

\rule[-1pt]{0pt}{8pt} 8_{7} & x^{5}yx^{2}y^{2}xy^{2} & 46080 & 2^{8}
 & 4_{1} & 8_{7} & 8_{7} \\
\hline

 \end{array}
 $$}}
\end{cclass}
%
\filbreak
\begin{cclass}\label{Co_2cc E}
\noindent Conjugacy classes of $E_3 = E(\Co_2) = \langle x, y, e\rangle$

\bigskip
{ \setlength{\arraycolsep}{1mm}
\renewcommand{\baselinestretch}{0.5}
{\scriptsize
 $$
 \begin{array}{|c|r|r|r|c|c|c|c|c|
}\hline
\rule[-3pt]{0pt}{10pt} \mbox{Class} &
 \multicolumn{1}{c|}{\mbox{Representative}} &
 \multicolumn{1}{c|}{|\mbox{Class}|} &
 \multicolumn{1}{c|}{|\mbox{Centralizer}|}
 &
 2{\rm P}
 & 3{\rm P}
 & 5{\rm P}
 & 7{\rm P}
 & 11{\rm P}
 \\\hline
\rule[-1pt]{0pt}{8pt} 1 & 1 & 1 & 2^{18}\cdot3^{2}\cdot5\cdot7\cdot11
 & 1 & 1 & 1 & 1 & 1 \\\hline

\rule[-1pt]{0pt}{8pt} 2_{1} & (xye)^{5} & 77 & 2^{18}\cdot3^{2}\cdot5
 & 1 & 2_{1} & 2_{1} & 2_{1} & 2_{1} \\\hline

\rule[-1pt]{0pt}{8pt} 2_{2} & (x)^{6} & 330 & 2^{17}\cdot3\cdot7
 & 1 & 2_{2} & 2_{2} & 2_{2} & 2_{2} \\\hline

\rule[-1pt]{0pt}{8pt} 2_{3} & (xe)^{10} & 616 & 2^{15}\cdot3^{2}\cdot5
 & 1 & 2_{3} & 2_{3} & 2_{3} & 2_{3} \\\hline

\rule[-1pt]{0pt}{8pt} 2_{4} & (x^{2}ey)^{7} & 2640 & 2^{14}\cdot3\cdot7
 & 1 & 2_{4} & 2_{4} & 2_{4} & 2_{4} \\\hline

\rule[-1pt]{0pt}{8pt} 2_{5} & (xy^{2}e)^{7} & 2640 & 2^{14}\cdot3\cdot7
 & 1 & 2_{5} & 2_{5} & 2_{5} & 2_{5} \\\hline

\rule[-1pt]{0pt}{8pt} 2_{6} & (e)^{2} & 18480 & 2^{14}\cdot3
 & 1 & 2_{6} & 2_{6} & 2_{6} & 2_{6} \\
\hline
 \end{array}
 $$}%
\filbreak%
%
\noindent Conjugacy classes of $E_3 = E(\Co_2) = \langle x, y, e\rangle$
 (continued)

\bigskip
{ \scriptsize
 $$
 \begin{array}{|c|r|r|r|c|c|c|c|c|
}\hline
\rule[-3pt]{0pt}{10pt} \mbox{Class} &
 \multicolumn{1}{c|}{\mbox{Representative}} &
 \multicolumn{1}{c|}{|\mbox{Class}|} &
 \multicolumn{1}{c|}{|\mbox{Centralizer}|}
 &
 2{\rm P}
 & 3{\rm P}
 & 5{\rm P}
 & 7{\rm P}
 & 11{\rm P}
 \\\hline

\rule[-1pt]{0pt}{8pt} 8_{1} & (x^{2}exe^{2})^{3} & 1182720 & 2^{8}\cdot3
 & 4_{1} & 8_{1} & 8_{1} & 8_{1} & 8_{1} \\\hline

\rule[-1pt]{0pt}{8pt} 8_{2} & (xe^{2})^{2} & 1774080 & 2^{9}
 & 4_{3} & 8_{2} & 8_{2} & 8_{2} & 8_{2} \\\hline

\rule[-1pt]{0pt}{8pt} 8_{3} & x^{3}ye & 1774080 & 2^{9}
 & 4_{3} & 8_{3} & 8_{3} & 8_{3} & 8_{3} \\\hline

\rule[-1pt]{0pt}{8pt} 8_{4} & x^{2}y & 3548160 & 2^{8}
 & 4_{6} & 8_{4} & 8_{4} & 8_{4} & 8_{4} \\\hline

\rule[-1pt]{0pt}{8pt} 8_{5} & x^{3}e^{2} & 3548160 & 2^{8}
 & 4_{6} & 8_{5} & 8_{5} & 8_{5} & 8_{5} \\\hline

\rule[-1pt]{0pt}{8pt} 8_{6} & x^{2}yey^{2} & 3548160 & 2^{8}
 & 4_{1} & 8_{6} & 8_{6} & 8_{6} & 8_{6} \\\hline

\rule[-1pt]{0pt}{8pt} 8_{7} & x^{2}eyxe & 3548160 & 2^{8}
 & 4_{3} & 8_{7} & 8_{7} & 8_{7} & 8_{7} \\
\hline
 \end{array}
 $$}
}
\end{cclass}
\filbreak
%
\begin{cclass}\label{Co_2cc G_3} 
\noindent Conjugacy classes of $\mathfrak G_3 = \langle h, x, y, e \rangle \cong \Co_2$
{ \setlength{\arraycolsep}{1mm}
\renewcommand{\baselinestretch}{0.5}
{\scriptsize
$$
 \begin{array}{|c|r|r|r|c|c|c|c|c|c|
}\hline
\rule[-3pt]{0pt}{10pt} \mbox{Class} &
 \multicolumn{1}{c|}{\mbox{Representative}} &
 \multicolumn{1}{c|}{|\mbox{Class}|} &
 \multicolumn{1}{c|}{|\mbox{Centralizer}|}
 &
 2{\rm P}
 & 3{\rm P}
 & 5{\rm P}
 & 7{\rm P}
 & 11{\rm P}
 & 23{\rm P}
 \\\hline
\rule[-1pt]{0pt}{8pt} 1 & 1 & 1 & 2^{18}\cdot3^{6}\cdot5^{3}\cdot7\cdot11\cdot23
 & 1 & 1 & 1 & 1 & 1 & 1 \\\hline

\rule[-1pt]{0pt}{8pt} 2_{1} & (xyh)^{15} & 56925 & 2^{18}\cdot3^{4}\cdot5\cdot7
 & 1 & 2_{1} & 2_{1} & 2_{1} & 2_{1} & 2_{1} \\\hline

\rule[-1pt]{0pt}{8pt} 2_{2} & (x)^{6} & 1024650 & 2^{17}\cdot3^{2}\cdot5\cdot7
 & 1 & 2_{2} & 2_{2} & 2_{2} & 2_{2} & 2_{2} \\\hline

\rule[-1pt]{0pt}{8pt} 2_{3} & (y)^{3} & 28690200 & 2^{15}\cdot3^{2}\cdot5
 & 1 & 2_{3} & 2_{3} & 2_{3} & 2_{3} & 2_{3} \\\hline

\rule[-1pt]{0pt}{8pt} 3_{1} & (x^{2}h)^{4} & 90675200 & 2^{7}\cdot3^{6}\cdot5
 & 3_{1} & 1 & 3_{1} & 3_{1} & 3_{1} & 3_{1} \\\hline

\rule[-1pt]{0pt}{8pt} 3_{2} & (x)^{4} & 272025600 & 2^{7}\cdot3^{5}\cdot5
 & 3_{2} & 1 & 3_{2} & 3_{2} & 3_{2} & 3_{2} \\\hline

\rule[-1pt]{0pt}{8pt} 4_{1} & (yh^{2})^{7} & 13662000 & 2^{14}\cdot3^{3}\cdot7
 & 2_{1} & 4_{1} & 4_{1} & 4_{1} & 4_{1} & 4_{1} \\
\hline
 \end{array}
 $$}
}
\end{cclass}
%
\filbreak
\begin{cclass}\label{Fi_22cc H}
\noindent Conjugacy classes of $H(\Fi_{22}) = \langle x, y, h \rangle$
{ \setlength{\arraycolsep}{1mm}
\renewcommand{\baselinestretch}{0.5}
{\scriptsize
 $$
 \begin{array}{|c|r|r|r|c|c|c|
}\hline
\rule[-2pt]{0pt}{10pt} \mbox{Class} &
 \multicolumn{1}{c|}{\mbox{Representative}} &
 \multicolumn{1}{c|}{|\mbox{Class}|} &
 \multicolumn{1}{c|}{|\mbox{Centralizer}|}
 &
 2{\rm P}
 & 3{\rm P}
 & 5{\rm P}
 \\\hline
\rule[-1pt]{0pt}{8pt} 1 & 1 & 1 & 2^{17}\cdot3^{4}\cdot5
 & 1 & 1 & 1 \\\hline

\rule[-1pt]{0pt}{8pt} 2_{1} & (xyxh)^{10} & 1 & 2^{17}\cdot3^{4}\cdot5
 & 1 & 2_{1} & 2_{1} \\\hline

\rule[-1pt]{0pt}{8pt} 2_{2} & (xhy^{2})^{9} & 2 & 2^{16}\cdot3^{4}\cdot5
 & 1 & 2_{2} & 2_{2} \\\hline

\rule[-1pt]{0pt}{8pt} 2_{3} & (xh)^{6} & 180 & 2^{15}\cdot3^{2}
 & 1 & 2_{3} & 2_{3} \\\hline

\rule[-1pt]{0pt}{8pt} 2_{4} & (xyhyxh^{2})^{3} & 180 & 2^{15}\cdot3^{2}
 & 1 & 2_{4} & 2_{4} \\\hline

\rule[-1pt]{0pt}{8pt} 2_{5} & (yh)^{6} & 270 & 2^{16}\cdot3
 & 1 & 2_{5} & 2_{5} \\\hline

\rule[-1pt]{0pt}{8pt} 2_{6} & (xy^{2}xy^{2}xh)^{4} & 270 & 2^{16}\cdot3
 & 1 & 2_{6} & 2_{6} \\
\hline

 \end{array}
 $$}
\filbreak
\noindent Conjugacy classes of $H(\Fi_{22}) = \langle x, y, h \rangle$
  (continued)

\bigskip
{  \scriptsize
 $$
 \begin{array}{|c|r|r|r|c|c|c|
}\hline
\rule[-3pt]{0pt}{10pt} \mbox{Class} &
 \multicolumn{1}{c|}{\mbox{Representative}} &
 \multicolumn{1}{c|}{|\mbox{Class}|} &
 \multicolumn{1}{c|}{|\mbox{Centralizer}|}
 &
 2{\rm P}
 & 3{\rm P}
 & 5{\rm P}
 \\\hline

\rule[-1pt]{0pt}{8pt} 6_{14} & (yh)^{2} & 138240 & 2^{7}\cdot3
 & 3_{2} & 2_{5} & 6_{14} \\\hline

\rule[-1pt]{0pt}{8pt} 6_{15} & xyxyhxy^{2}h & 138240 & 2^{7}\cdot3
 & 3_{2} & 2_{6} & 6_{15} \\\hline

\rule[-1pt]{0pt}{8pt} 6_{16} & y^{2}hyh & 184320 & 2^{5}\cdot3^{2}
 & 3_{2} & 2_{8} & 6_{16} \\\hline

\rule[-1pt]{0pt}{8pt} 6_{17} & xyxyxh & 184320 & 2^{5}\cdot3^{2}
 & 3_{2} & 2_{7} & 6_{17} \\\hline

\rule[-1pt]{0pt}{8pt} 6_{18} & xyhxhy & 368640 & 2^{4}\cdot3^{2}
 & 3_{3} & 2_{7} & 6_{19} \\\hline

\rule[-1pt]{0pt}{8pt} 6_{19} & (xyhxhy)^{5} & 368640 & 2^{4}\cdot3^{2}
 & 3_{3} & 2_{7} & 6_{18} \\\hline

\rule[-1pt]{0pt}{8pt} 6_{20} & xyhyxh^{2} & 368640 & 2^{4}\cdot3^{2}
 & 3_{3} & 2_{4} & 6_{20} \\
\hline
 \end{array}
 $$}
 }
\end{cclass}
%
\filbreak
\begin{cclass}\label{Fi_22cc D}\noindent Conjugacy classes of $D(\Fi_{22}) = \langle x, y \rangle$

\bigskip
{ \setlength{\arraycolsep}{1mm}
\renewcommand{\baselinestretch}{0.5}
{\scriptsize
 $$
 \begin{array}{|c|r|r|r|c|c|c|
}\hline
\rule[-3pt]{0pt}{10pt} \mbox{Class} &
 \multicolumn{1}{c|}{\mbox{Representative}} &
 \multicolumn{1}{c|}{|\mbox{Class}|} &
 \multicolumn{1}{c|}{|\mbox{Centralizer}|}
 &
 2{\rm P}
 & 3{\rm P}
 & 5{\rm P}
 \\\hline
\rule[-1pt]{0pt}{8pt} 1 & 1 & 1 & 2^{17}\cdot3\cdot5
 & 1 & 1 & 1 \\\hline

\rule[-1pt]{0pt}{8pt} 2_{1} & (xyxyxy^{2}xy^{3})^{8} & 1 & 2^{17}\cdot3\cdot5
 & 1 & 2_{1} & 2_{1} \\\hline

\rule[-1pt]{0pt}{8pt} 2_{2} & (xyxyxyxy^{2}xy^{5})^{3} & 2 & 2^{16}\cdot3\cdot5
 & 1 & 2_{2} & 2_{2} \\\hline

\rule[-1pt]{0pt}{8pt} 2_{3} & (xyxy^{2})^{4} & 20 & 2^{15}\cdot3
 & 1 & 2_{3} & 2_{3} \\\hline

\rule[-1pt]{0pt}{8pt} 2_{4} & (xyxyxyxy^{3})^{3} & 20 & 2^{15}\cdot3
 & 1 & 2_{4} & 2_{4} \\\hline

\rule[-1pt]{0pt}{8pt} 2_{5} & (xyxy^{3})^{4} & 30 & 2^{16}
 & 1 & 2_{5} & 2_{5} \\\hline

\rule[-1pt]{0pt}{8pt} 2_{6} & (xyxy^{3}xy^{4})^{4} & 30 & 2^{16}
 & 1 & 2_{6} & 2_{6} \\
\hline
 \end{array}
 $$}
\filbreak
\noindent Conjugacy classes of $D(\Fi_{22}) = \langle x, y \rangle$
 (continued)

\bigskip
{\scriptsize
 $$
 \begin{array}{|c|r|r|r|c|c|c|
}\hline
\rule[-3pt]{0pt}{10pt} \mbox{Class} &
 \multicolumn{1}{c|}{\mbox{Representative}} &
 \multicolumn{1}{c|}{|\mbox{Class}|} &
 \multicolumn{1}{c|}{|\mbox{Centralizer}|}
 &
 2{\rm P}
 & 3{\rm P}
 & 5{\rm P}
 \\\hline
\rule[-1pt]{0pt}{8pt} 4_{31} & xyxy^{2}xy^{3}xyxy^{3}xy^{2} & 7680 & 2^{8}
 & 2_{8} & 4_{31} & 4_{31} \\\hline

\rule[-1pt]{0pt}{8pt} 4_{32} & xyxy^{2}xyxy^{2}xy^{4}xy^{3} & 7680 & 2^{8}
 & 2_{10} & 4_{32} & 4_{32} \\\hline

\rule[-1pt]{0pt}{8pt} 4_{33} & xyxyxy^{2}xy^{2}xyxy^{2}xyxy^{3} & 7680 & 2^{8}
 & 2_{10} & 4_{33} & 4_{33} \\\hline

\rule[-1pt]{0pt}{8pt} 4_{34} & xyxyxyxyxy^{2}xyxy^{5}xy^{2} & 7680 & 2^{8}
 & 2_{10} & 4_{34} & 4_{34} \\\hline

\rule[-1pt]{0pt}{8pt} 4_{35} & xyxyxyxy^{2}xy^{5}xy^{2}xy^{3} & 7680 & 2^{8}
 & 2_{9} & 4_{35} & 4_{35} \\\hline

\rule[-1pt]{0pt}{8pt} 4_{36} & xyxyxyxy^{5}xy^{2}xyxy^{2}xy^{5} & 7680 & 2^{8}
 & 2_{8} & 4_{36} & 4_{36} \\\hline

\rule[-1pt]{0pt}{8pt} 4_{37} & xyxyxyxy^{3}xy^{2}xy^{4} & 15360 & 2^{7}
 & 2_{11} & 4_{37} & 4_{37} \\
\hline
 \end{array}
 $$}}\end{cclass}
%
%
\filbreak
\begin{cclass}\label{Fi_22cc E} Conjugacy classes of $E_2 = E(\Fi_{22}) = \langle x, y, e
\rangle$ 

\bigskip
{ \setlength{\arraycolsep}{1mm}
\renewcommand{\baselinestretch}{0.5}
{\scriptsize
 $$
 \begin{array}{|c|r|r|r|c|c|c|c|c|
}\hline
\rule[-3pt]{0pt}{10pt} \mbox{Class} &
 \multicolumn{1}{c|}{\mbox{Representative}} &
 \multicolumn{1}{c|}{|\mbox{Class}|} &
 \multicolumn{1}{c|}{|\mbox{Centralizer}|}
 &
 2{\rm P}
 & 3{\rm P}
 & 5{\rm P}
 & 7{\rm P}
 & 11{\rm P}
 \\\hline
\rule[-1pt]{0pt}{8pt} 1 & 1 & 1 & 2^{17}\cdot3^{2}\cdot5\cdot7\cdot11
 & 1 & 1 & 1 & 1 & 1 \\\hline

\rule[-1pt]{0pt}{8pt} 2_{1} & (xye)^{7} & 22 & 2^{16}\cdot3^{2}\cdot5\cdot7
 & 1 & 2_{1} & 2_{1} & 2_{1} & 2_{1} \\\hline

\rule[-1pt]{0pt}{8pt} 2_{2} & (xy)^{6} & 231 & 2^{17}\cdot3\cdot5
 & 1 & 2_{2} & 2_{2} & 2_{2} & 2_{2} \\\hline

\rule[-1pt]{0pt}{8pt} 2_{3} & (ye^{2})^{6} & 770 & 2^{16}\cdot3^{2}
 & 1 & 2_{3} & 2_{3} & 2_{3} & 2_{3} \\\hline

\rule[-1pt]{0pt}{8pt} 2_{4} & x & 18480 & 2^{13}\cdot3
 & 1 & 2_{4} & 2_{4} & 2_{4} & 2_{4} \\\hline

\rule[-1pt]{0pt}{8pt} 2_{5} & (xeyxe^{2})^{2} & 55440 & 2^{13}
 & 1 & 2_{5} & 2_{5} & 2_{5} & 2_{5} \\\hline

\rule[-1pt]{0pt}{8pt} 3 & (y)^{2} & 788480 & 2^{6}\cdot3^{2}
 & 3 & 1 & 3 & 3 & 3 \\
\hline

 \end{array}
 $$}}\end{cclass}
%
\filbreak
\begin{cclass}\label{Fi_22cc G_2} Conjugacy classes of $\mathfrak G_2 = \langle h, x, y, e \rangle \cong
\Fi_{22}$ 

\bigskip
{ \setlength{\arraycolsep}{1mm}
\renewcommand{\baselinestretch}{0.5}
{\scriptsize
 $$
 \begin{array}{|c|r|r|r|c|c|c|c|c|c|
}\hline \rule[-3pt]{0pt}{10pt} \mbox{Class} &
\multicolumn{1}{c|}{\mbox{Representative}} &
\multicolumn{1}{c|}{|\mbox{Class}|} &
\multicolumn{1}{c|}{|\mbox{Centralizer}|} & 2{\rm P} & 3{\rm P} &
5{\rm P} & 7{\rm P} & 11{\rm P} & 13{\rm P} \\\hline 

\rule[-1pt]{0pt}{8pt} 1 & 1 & 1 &
2^{17}\cdot3^{9}\cdot5^{2}\cdot7\cdot11\cdot13 
   & 1 & 1 & 1 & 1 & 1 & 1 \\\hline

\rule[-1pt]{0pt}{8pt} 2_{1} & (xye)^{7} & 3510 &
 2^{16}\cdot3^{6}\cdot5\cdot7\cdot11 & 1 & 2_{1} & 2_{1} & 2_{1} &
 2_{1} & 2_{1} \\\hline

\rule[-1pt]{0pt}{8pt} 2_{2} & x & 1216215 & 2^{17}\cdot3^{4}\cdot5
 & 1 & 2_{2} & 2_{2} & 2_{2} & 2_{2} & 2_{2} \\\hline

\rule[-1pt]{0pt}{8pt} 2_{3} & (xh)^{6} & 36486450 & 2^{16}\cdot3^{3}
 & 1 & 2_{3} & 2_{3} & 2_{3} & 2_{3} & 2_{3} \\\hline

\rule[-1pt]{0pt}{8pt} 3_{1} & (yh)^{4} & 3294720 & 2^{8}\cdot3^{7}\cdot5\cdot7
 & 3_{1} & 1 & 3_{1} & 3_{1} & 3_{1} & 3_{1} \\\hline

\rule[-1pt]{0pt}{8pt} 3_{2} & h & 25625600 & 2^{7}\cdot3^{9}
 & 3_{2} & 1 & 3_{2} & 3_{2} & 3_{2} & 3_{2} \\\hline

\rule[-1pt]{0pt}{8pt} 3_{3} & (y)^{2} & 461260800 & 2^{6}\cdot3^{7}
 & 3_{3} & 1 & 3_{3} & 3_{3} & 3_{3} & 3_{3} \\
\hline
 \end{array}
 $$}}
\filbreak
\noindent Conjugacy classes of $\mathfrak G_2 = \langle h, x, y, e \rangle \cong
\Fi_{22}$ (continued) 

\bigskip
{ \setlength{\arraycolsep}{1mm}
\renewcommand{\baselinestretch}{0.5}
{\scriptsize
 $$
 \begin{array}{|c|r|r|r|c|c|c|c|c|c|
}\hline \rule[-3pt]{0pt}{10pt} \mbox{Class} &
\multicolumn{1}{c|}{\mbox{Representative}} &
\multicolumn{1}{c|}{|\mbox{Class}|} &
\multicolumn{1}{c|}{|\mbox{Centralizer}|} & 2{\rm P} & 3{\rm P} &
5{\rm P} & 7{\rm P} & 11{\rm P} & 13{\rm P} \\\hline 

\rule[-1pt]{0pt}{8pt} 9_{1} & (xy^{2}h)^{2} & 199264665600 & 2^{2}\cdot3^{4}
 & 9_{1} & 3_{2} & 9_{1} & 9_{1} & 9_{1} & 9_{1} \\\hline

\rule[-1pt]{0pt}{8pt} 9_{2} & xheye & 398529331200 & 2\cdot3^{4}
 & 9_{2} & 3_{2} & 9_{2} & 9_{2} & 9_{2} & 9_{2} \\\hline

\rule[-1pt]{0pt}{8pt} 9_{3} & xy^{2}eh & 2391175987200 & 3^{3}
 & 9_{3} & 3_{4} & 9_{3} & 9_{3} & 9_{3} & 9_{3} \\\hline

\rule[-1pt]{0pt}{8pt} 10_{1} & y^{2}e & 1076029194240 & 2^{2}\cdot3\cdot5
 & 5 & 10_{1} & 2_{1} & 10_{1} & 10_{1} & 10_{1} \\\hline

\rule[-1pt]{0pt}{8pt} 10_{2} & xhe & 1614043791360 & 2^{3}\cdot5
 & 5 & 10_{2} & 2_{2} & 10_{2} & 10_{2} & 10_{2} \\\hline

\rule[-1pt]{0pt}{8pt} 11_{1} & xe & 2934625075200 & 2\cdot11
 & 11_{2} & 11_{1} & 11_{1} & 11_{2} & 1 & 11_{2} \\\hline

\rule[-1pt]{0pt}{8pt} 11_{2} & (xe)^{2} & 2934625075200 & 2\cdot11
 & 11_{1} & 11_{2} & 11_{2} & 11_{1} & 1 & 11_{1} \\
\hline
\end{array}
 $$}}
\end{cclass}
%
\filbreak
\section{Character Tables of Local Subgroups of $\Co_2$ and $\Fi_{22}$}
\begin{ctab}\label{Co_2ct_H}Character table of $H(\Co_2) = \langle x,y,h \rangle$
{ \setlength{\arraycolsep}{1mm}
\renewcommand{\baselinestretch}{0.5}
{\tiny \setlength{\arraycolsep}{0,3mm}
$$
\begin{array}{r|rrrrrrrrrrrrrrrrrrrrrr}
2&18&18&17&14&14&16&16&14&14&15&11&6&5&7&14&14&13&12&12&12&12\\
3&4&4&1&2&2&2&2&1&1&.&1&4&3&3&3&1&1&1&1&1&1\\
5&1&1&.&1&1&.&.&.&.&.&.&.&1&.&.&.&.&.&.&.&.\\
7&1&1&1&.&.&.&.&.&.&.&.&.&.&.&1&.&.&.&.&.&.\\
\hline
\rule[-1pt]{0pt}{6pt}&1a&2a&2b&2c&2d&2e&2f&2g&2h&2i&2j&3a&3b&3c&4a&4b&4c&4d&4e&4f&4g\\
\hline
\rule{0pt}{5pt}2P&1a&1a&1a&1a&1a&1a&1a&1a&1a&1a&1a&3a&3b&3c&2a&2a&2b&2f&2b&2f&2b\\
3P&1a&2a&2b&2c&2d&2e&2f&2g&2h&2i&2j&1a&1a&1a&4a&4b&4c&4d&4e&4f&4g\\
5P&1a&2a&2b&2c&2d&2e&2f&2g&2h&2i&2j&3a&3b&3c&4a&4b&4c&4d&4e&4f&4g\\
\end{array}
$$}}
%
\filbreak
\noindent Character table of $H(\Co_2)$ (continued){
\setlength{\arraycolsep}{1mm}
\renewcommand{\baselinestretch}{0.5}%
{\tiny \setlength{\arraycolsep }{0,3mm}
$$
\begin{array}{r|rrrrrrrrrrrrrrrrrrrrrrrrrr}
2&13&11&11&12&12&10&11&11&9&9&9&9&10&10&8&8&8&2&6&5&7&6&6&5&5\\
3&.&1&1&.&.&1&.&.&1&1&1&1&.&.&.&.&.&1&4&3&3&2&2&2&2\\
5&.&.&.&.&.&.&.&.&.&.&.&.&.&.&.&.&.&1&.&1&.&.&.&.&.\\
7&.&.&.&.&.&.&.&.&.&.&.&.&.&.&.&.&.&.&.&.&.&.&.&.&.\\
\hline
\rule[-1pt]{0pt}{6pt}&4h&4i&4j&4k&4l&4m&4n&4o&4p&4q&4r&4s&4t&4u&4v&4w&4x&5a&6a&6b&6c&6d&6e&6f&6g\\
\hline
\rule{0pt}{5pt}2P&2b&2e&2b&2e&2e&2e&2b&2f&2g&2g&2g&2g&2b&2e&2i&2g&2g&5a&3a&3b&3c&3a&3a&3b&3b\\
3P&4h&4i&4j&4k&4l&4m&4n&4o&4p&4q&4r&4s&4t&4u&4v&4w&4x&5a&2a&2a&2a&2e&2f&2c&2d\\
5P&4h&4i&4j&4k&4l&4m&4n&4o&4p&4q&4r&4s&4t&4u&4v&4w&4x&1a&6a&6b&6c&6d&6e&6f&6g\\
\end{array}
$$}}%
\filbreak
\nobreak
Character table of $H(\Co_2)$ (continued){
\setlength{\arraycolsep}{1mm}
\renewcommand{\baselinestretch}{0.5}
{\tiny \setlength{\arraycolsep }{0,3mm}
$$
\begin{array}{r|rrrrrrrrrrrrrrrrrrrrrrrrrrrrrr}
2&5&5&5&5&5&5&6&5&5&4&3&8&9&9&8&8&7&7&7&7&7&6&6&6&1&2&2&2&5\\
3&2&2&2&2&2&2&1&1&1&1&.&1&.&.&.&.&.&.&.&.&.&.&.&.&2&1&.&.&3\\
5&.&.&.&.&.&.&.&.&.&.&.&.&.&.&.&.&.&.&.&.&.&.&.&.&.&1&1&1&.\\
7&.&.&.&.&.&.&.&.&.&.&1&.&.&.&.&.&.&.&.&.&.&.&.&.&.&.&.&.&.\\
\hline
\rule[-1pt]{0pt}{6pt}&6h&6i&6j&6k&6l&6m&6n&6o&6p&6q&7a&8a&8b&8c&8d&8e&8f&8g&8h&8i&8j&8k&8l&8m&9a&10a
&10b&10c&12a\\
\hline
\rule{0pt}{5pt}2P&3c&3c&3c&3b&3b&3c&3c&3b&3b&3c&7a&4a&4b&4b&4h&4h&4i&4o&4h&4i&4o&4k&4n&4d&9a&5a&5a&5a&6a\\
3P&2e&2d&2f&2e&2f&2c&2b&2g&2h&2j&7a&8a&8b&8c&8d&8e&8f&8g&8h&8i&8j&8k&8l&8m&3a&10a&10b&10c&4a\\
5P&6h&6i&6j&6k&6l&6m&6n&6o&6p&6q&7a&8a&8b&8c&8d&8e&8f&8g&8h&8i&8j&8k&8l&8m&9a&2a&2d&2c&12a\\
\end{array}
$$}}%
\filbreak
\noindent Character table of $H(\Co_2)$ (continued)
{\setlength{\arraycolsep}{1mm}
\renewcommand{\baselinestretch}{0.5}
{\tiny \setlength{\arraycolsep }{0,3mm}
$$
\begin{array}{r|rrrrrrrrrrrrrrrrrrrrrrrrrr}
2&5&5&5&5&5&5&4&4&4&4&4&4&4&4&4&3&2&2&1&5&5&1&3&2&1\\
3&2&1&1&1&1&1&1&1&1&1&1&1&1&1&1&.&.&.&1&.&.&2&1&.&1\\
5&.&.&.&.&.&.&.&.&.&.&.&.&.&.&.&.&.&.&1&.&.&.&.&.&1\\
7&.&.&.&.&.&.&.&.&.&.&.&.&.&.&.&1&1&1&.&.&.&.&.&1&.\\
\hline
\rule[-1pt]{0pt}{6pt}&12b&12c&12d&12e&12f&12g&12h&12i&12j&12k&12l&12m&12n&12o&
 12p&14a&14b&14c&15a&16a&16b&18a&24a&28a&30a\\
\hline
\rule{0pt}{5pt}2P&6c&6e&6n&6a&6e&6n&6o&6d&6o&6n&6n&6d&6d&6o&6o&7a&7a&7a&15a&8b&8c&9a&12b&14a&15a\\
3P&4a&4d&4e&4b&4f&4g&4p&4i&4q&4j&4c&4m&4m&4s&4r&14a&14c&14b&5a&16a&16b&6a&8a&28a&10a\\
5P&12b&12c&12d&12e&12f&12g&12h&12i&12j&12k&12l&12n&12m&12o&12p&14a&14c&14b&3b&16a&16b&18a&24a&28a&6b\\
\end{array}
$$}}
\filbreak
\noindent Character table of $H(\Co_2)$ (continued){
\setlength{\arraycolsep}{1mm}
\renewcommand{\baselinestretch}{0.5}
{\tiny \setlength{\arraycolsep }{0,3mm}
\begin{alignat*}{1}
&\begin{array}{r|rrrrrrrrrrrrrrrrrrrrrr}
2&18&18&17&14&14&16&16&14&14&15&11&6&5&7&14&14&13&12&12&12&12\\
3&4&4&1&2&2&2&2&1&1&.&1&4&3&3&3&1&1&1&1&1&1\\
5&1&1&.&1&1&.&.&.&.&.&.&.&1&.&.&.&.&.&.&.&.\\
7&1&1&1&.&.&.&.&.&.&.&.&.&.&.&1&.&.&.&.&.&.\\
\hline
\rule[-1pt]{0pt}{6pt}&1a&2a&2b&2c&2d&2e&2f&2g&2h&2i&2j&3a&3b&3c&4a&4b&4c&4d&4e&4f&4g\\
\hline
\rule{0pt}{5pt}2P&1a&1a&1a&1a&1a&1a&1a&1a&1a&1a&1a&3a&3b&3c&2a&2a&2b&2f&2b&2f&2b\\
3P&1a&2a&2b&2c&2d&2e&2f&2g&2h&2i&2j&1a&1a&1a&4a&4b&4c&4d&4e&4f&4g\\
5P&1a&2a&2b&2c&2d&2e&2f&2g&2h&2i&2j&3a&3b&3c&4a&4b&4c&4d&4e&4f&4g\\
7P&1a&2a&2b&2c&2d&2e&2f&2g&2h&2i&2j&3a&3b&3c&4a&4b&4c&4d&4e&4f&4g\\
\hline
\rule{0pt}{5pt}X.83&3360&-3360&.&40&-40&-112&112&40&-40&.&.&12&-15&12&.&.&.&24&-8&-24&8\\
X.84&3360&-3360&.&200&-200&16&-16&8&-8&.&.&-6&15&.&.&.&.&-8&24&8&-24\\
X.85&3456&-3456&.&-96&96&192&-192&32&-32&.&.&.&-9&.&.&.&.&.&.&.&.\\
X.86&3780&3780&196&-60&-60&132&132&-60&-60&4&4&.&.&9&-252&-60&4&12&4&12&4\\
X.87&3780&3780&196&-60&-60&-156&-156&36&36&-28&4&.&.&9&-252&36&4&-12&4&-12&4\\
X.88&4480&-4480&.&-160&160&-64&64&-32&32&.&.&-20&10&4&.&.&.&.&-32&.&32\\
X.89&4480&-4480&.&160&-160&192&-192&32&-32&.&.&16&-5&-8&.&.&.&.&.&.&.\\
X.90&5040&5040&-336&.&.&48&48&.&.&-16&.&18&.&.&336&16&.&-24&.&-24&.\\
X.91&5040&-5040&.&-180&180&-168&168&12&-12&.&.&18&.&.&.&.&.&-20&-12&20&12\\
X.92&5376&-5376&.&-64&64&128&-128&-64&64&.&.&12&6&.&.&.&.&.&.&.&.\\
X.93&5670&5670&294&270&270&-90&-90&-18&-18&-26&6&.&.&.&-378&6&-18&-6&6&-6&6\\
X.94&6048&-6048&.&-120&120&-48&48&8&-8&.&.&.&-9&.&.&.&.&24&24&-24&-24\\
X.95&6480&-6480&.&180&-180&-216&216&-12&12&.&.&.&.&.&.&.&.&-12&12&12&-12\\
X.96&6720&6720&-448&.&.&-192&-192&.&.&64&.&6&.&12&448&-64&.&.&.&.&.\\
X.97&6720&-6720&.&80&-80&32&-32&-48&48&.&.&6&.&12&.&.&.&-16&-16&16&16\\
X.98&7560&7560&392&-120&-120&-24&-24&-24&-24&-24&8&.&.&-9&-504&-24&8&.&8&.&8\\
X.99&7680&7680&-512&.&.&.&.&.&.&.&.&-24&.&-12&512&.&.&.&.&.&.\\
X.100&\,8192&-8192&.&.&.&.&.&.&.&.&.&-16&-16&-16&.&.&.&.&.&.&.\\
\end{array}\\
\intertext{\normalsize\noindent Character table of $H(\Co_2)$ (continued)}
&\begin{array}{r|rrrrrrrrrrrrrrrrrrrrrrrrrr}
2&13&11&11&12&12&10&11&11&9&9&9&9&10&10&8&8&8&2&6&5&7&6&6&5&5\\
3&.&1&1&.&.&1&.&.&1&1&1&1&.&.&.&.&.&1&4&3&3&2&2&2&2\\
5&.&.&.&.&.&.&.&.&.&.&.&.&.&.&.&.&.&1&.&1&.&.&.&.&.\\
7&.&.&.&.&.&.&.&.&.&.&.&.&.&.&.&.&.&.&.&.&.&.&.&.&.\\
\hline
\rule[-1pt]{0pt}{6pt}&4h&4i&4j&4k&4l&4m&4n&4o&4p&4q&4r&4s&4t&4u&4v&4w&4x&5a&6a&6b&6c&6d&6e&6f&6g\\
\hline
\rule{0pt}{5pt}2P&2b&2e&2b&2e&2e&2e&2b&2f&2g&2g&2g&2g&2b&2e&2i&2g&2g&5a&3a&3b&3c&3a&3a&3b&3b\\
3P&4h&4i&4j&4k&4l&4m&4n&4o&4p&4q&4r&4s&4t&4u&4v&4w&4x&5a&2a&2a&2a&2e&2f&2c&2d\\
5P&4h&4i&4j&4k&4l&4m&4n&4o&4p&4q&4r&4s&4t&4u&4v&4w&4x&1a&6a&6b&6c&6d&6e&6f&6g\\
7P&4h&4i&4j&4k&4l&4m&4n&4o&4p&4q&4r&4s&4t&4u&4v&4w&4x&5a&6a&6b&6c&6d&6e&6f&6g\\
\hline
\rule{0pt}{5pt}X.83&.&.&.&-8&8&.&.&.&4&4&-4&-4&.&.&.&-4&4&.&-12&15&-12&-4&4&1&-1\\
X.84&.&.&.&-8&8&.&.&.&-4&-4&4&4&.&.&.&-4&4&.&6&-15&.&-2&2&-1&1\\
X.85&.&.&.&.&.&.&.&.&8&-8&-8&8&.&.&.&.&.&1&.&9&.&.&.&-3&3\\
X.86&4&12&4&-4&-4&12&4&-4&.&.&.&.&4&-4&4&.&.&.&.&.&9&.&.&.&.\\
X.87&-28&-12&4&4&4&-12&4&4&.&.&.&.&4&4&-4&.&.&.&.&.&9&.&.&.&.\\
X.88&.&.&.&.&.&.&.&.&.&.&.&.&.&.&.&.&.&.&20&-10&-4&-4&4&2&-2\\
X.89&.&.&.&.&.&.&.&.&-8&8&8&-8&.&.&.&.&.&.&-16&5&8&.&.&1&-1\\
X.90&.&-24&.&-8&-8&24&.&-8&.&.&.&.&.&8&.&.&.&.&18&.&.&6&6&.&.\\
X.91&.&.&.&12&-12&.&.&.&-6&-6&6&6&.&.&.&-2&2&.&-18&.&.&6&-6&.&.\\
X.92&.&.&.&.&.&.&.&.&.&.&.&.&.&.&.&.&.&1&-12&-6&.&-4&4&2&-2\\
X.93&14&-6&6&2&2&-6&-2&2&-6&-6&-6&-6&-10&2&2&2&2&.&.&.&.&.&.&.&.\\
X.94&.&.&.&-8&8&.&.&.&-4&-4&4&4&.&.&.&4&-4&-2&.&9&.&.&.&3&-3\\
X.95&.&.&.&20&-20&.&.&.&6&6&-6&-6&.&.&.&2&-2&.&.&.&.&.&.&.&.\\
X.96&.&.&.&.&.&.&.&.&.&.&.&.&.&.&.&.&.&.&6&.&12&-6&-6&.&.\\
X.97&.&.&.&-16&16&.&.&.&.&.&.&.&.&.&.&.&.&.&-6&.&-12&2&-2&-4&4\\
X.98&-24&.&8&.&.&.&8&.&.&.&.&.&8&.&.&.&.&.&.&.&-9&.&.&.&.\\
X.99&.&.&.&.&.&.&.&.&.&.&.&.&.&.&.&.&.&.&-24&.&-12&.&.&.&.\\
X.100&.&.&.&.&.&.&.&.&.&.&.&.&.&.&.&.&.&2&16&16&16&.&.&.&.\\
\end{array}\\
\intertext{\normalsize\noindent Character table of $H(\Co_2)$ (continued)}
&\begin{array}{r|rrrrrrrrrrrrrrrrrrrrrrrrrrrrrr}
2&5&5&5&5&5&5&6&5&5&4&3&8&9&9&8&8&7&7&7&7&7&6&6&6&1&2&2&2&5\\
3&2&2&2&2&2&2&1&1&1&1&.&1&.&.&.&.&.&.&.&.&.&.&.&.&2&1&.&.&3\\
5&.&.&.&.&.&.&.&.&.&.&.&.&.&.&.&.&.&.&.&.&.&.&.&.&.&1&1&1&.\\
7&.&.&.&.&.&.&.&.&.&.&1&.&.&.&.&.&.&.&.&.&.&.&.&.&.&.&.&.&.\\
\hline
\rule[-1pt]{0pt}{6pt}&6h&6i&6j&6k&6l&6m&6n&6o&6p&6q&7a&8a&8b&8c&8d&8e&8f&8g&8h&8i&8j&8k&8l&8m&9a&10a
&10b&10c&12a\\
\hline
\rule{0pt}{5pt}2P&3c&3c&3c&3b&3b&3c&3c&3b&3b&3c&7a&4a&4b&4b&4h&4h&4i&4o&4h&4i&4o&4k&4n&4d&9a&5a&5a&5a&6a\\
3P&2e&2d&2f&2e&2f&2c&2b&2g&2h&2j&7a&8a&8b&8c&8d&8e&8f&8g&8h&8i&8j&8k&8l&8m&3a&10a&10b&10c&4a\\
5P&6h&6i&6j&6k&6l&6m&6n&6o&6p&6q&7a&8a&8b&8c&8d&8e&8f&8g&8h&8i&8j&8k&8l&8m&9a&2a&2d&2c&12a\\
\end{array}
\end{alignat*}}}
\filbreak
\noindent Character table of $H(\Co_2)$ (continued){
\setlength{\arraycolsep}{1mm}
\renewcommand{\baselinestretch}{0.5}
{\tiny \setlength{\arraycolsep }{0,3mm}
$$
\begin{array}{r|rrrrrrrrrrrrrrrrrrrrrrrrrr}
2&5&5&5&5&5&5&4&4&4&4&4&4&4&4&4&3&2&2&1&5&5&1&3&2&1\\
3&2&1&1&1&1&1&1&1&1&1&1&1&1&1&1&.&.&.&1&.&.&2&1&.&1\\
5&.&.&.&.&.&.&.&.&.&.&.&.&.&.&.&.&.&.&1&.&.&.&.&.&1\\
7&.&.&.&.&.&.&.&.&.&.&.&.&.&.&.&1&1&1&.&.&.&.&.&1&.\\
\hline
\rule[-1pt]{0pt}{6pt}&12b&12c&12d&12e&12f&12g&12h&12i&12j&12k&12l&12m&12n&12o&12p&14a&14b&14c&15a&16a&16b&18a&24a&28a&30a\\
\hline
\rule{0pt}{5pt}2P&6c&6e&6n&6a&6e&6n&6o&6d&6o&6n&6n&6d&6d&6o&6o&7a&7a&7a&15a&8b&8c&9a&12b&14a&15a\\
3P&4a&4d&4e&4b&4f&4g&4p&4i&4q&4j&4c&4m&4m&4s&4r&14a&14c&14b&5a&16a&16b&6a&8a&28a&10a\\
5P&12b&12c&12d&12e&12f&12g&12h&12i&12j&12k&12l&12n&12m&12o&12p&14a&14c&14b&3b&16a&16b&18a&24a&28a&6b\\
7P&12b&12c&12d&12e&12f&12g&12h&12i&12j&12k&12l&12m&12n&12o&12p&2a&2b&2b&15a&16a&16b&18a&24a&4a&30a\\
\hline
\rule{0pt}{5pt}X.83&.&.&-2&.&.&2&1&.&1&.&.&.&.&-1&-1&.&.&.&.&.&.&.&.&.&.\\
X.84&.&-2&.&.&2&.&-1&.&-1&.&.&.&.&1&1&.&.&.&.&.&.&.&.&.&.\\
X.85&.&.&.&.&.&.&-1&.&1&.&.&.&.&-1&1&2&.&.&1&.&.&.&.&.&-1\\
X.86&-3&.&1&.&.&1&.&.&.&1&1&.&.&.&.&.&.&.&.&.&.&.&-1&.&.\\
X.87&-3&.&1&.&.&1&.&.&.&1&1&.&.&.&.&.&.&.&.&.&.&.&-1&.&.\\
X.88&.&.&-2&.&.&2&.&.&.&.&.&.&.&.&.&.&.&.&.&.&.&-1&.&.&.\\
X.89&.&.&.&.&.&.&1&.&-1&.&.&.&.&1&-1&.&.&.&.&.&.&-1&.&.&.\\
X.90&.&.&.&-2&.&.&.&.&.&.&.&.&.&.&.&.&.&.&.&.&.&.&.&.&.\\
X.91&.&-2&.&.&2&.&.&.&.&.&.&.&.&.&.&.&.&.&.&.&.&.&.&.&.\\
X.92&.&.&.&.&.&.&.&.&.&.&.&.&.&.&.&.&.&.&1&.&.&.&.&.&-1\\
X.93&.&.&.&.&.&.&.&.&.&.&.&.&.&.&.&.&.&.&.&.&.&.&.&.&.\\
X.94&.&.&.&.&.&.&-1&.&-1&.&.&.&.&1&1&.&.&.&1&.&.&.&.&.&-1\\
X.95&.&.&.&.&.&.&.&.&.&.&.&.&.&.&.&2&.&.&.&.&.&.&.&.&.\\
X.96&4&.&.&2&.&.&.&.&.&.&.&.&.&.&.&.&.&.&.&.&.&.&.&.&.\\
X.97&.&2&2&.&-2&-2&.&.&.&.&.&.&.&.&.&.&.&.&.&.&.&.&.&.&.\\
X.98&3&.&-1&.&.&-1&.&.&.&-1&-1&.&.&.&.&.&.&.&.&.&.&.&1&.&.\\
X.99&-4&.&.&.&.&.&.&.&.&.&.&.&.&.&.&1&-1&-1&.&.&.&.&.&1&.\\
X.100&.&.&.&.&.&.&.&.&.&.&.&.&.&.&.&-2&.&.&-1&.&.&1&.&.&1\\
\end{array}
$$}}%
\noindent where $A = i\sqrt{7},\quad  B = -1 - i\sqrt{3}$.
\end{ctab}
%
\filbreak
\begin{ctab}\label{Co_2ct_D}Character table of $D(\Co_2) = \langle x,y \rangle$
{ \setlength{\arraycolsep}{1mm}
\renewcommand{\baselinestretch}{0.5}
{\tiny \setlength{\arraycolsep}{0,3mm}
$$
\begin{array}{r|rrrrrrrrrrrrrrrrrrrrrrrr}
2&18&18&14&14&17&16&16&14&14&14&14&14&15&13&13&13&13&14&13&11&12&12&12\\
3&2&2&2&2&1&1&1&1&1&1&1&1&.&1&1&1&1&.&.&1&.&.&.\\
5&1&1&1&1&.&.&.&.&.&.&.&.&.&.&.&.&.&.&.&.&.&.&.\\
\hline
\rule[-1pt]{0pt}{6pt}&1a&2a&2b&2c&2d&2e&2f&2g&2h&2i&2j&2k&2l&2m&2n&2o&2p&2q&2r&2s&2t&2u&2v\\
\hline
\rule{0pt}{5pt}2P&1a&1a&1a&1a&1a&1a&1a&1a&1a&1a&1a&1a&1a&1a&1a&1a&1a&1a&1a&1a&1a&1a&1a\\
3P&1a&2a&2b&2c&2d&2e&2f&2g&2h&2i&2j&2k&2l&2m&2n&2o&2p&2q&2r&2s&2t&2u&2v\\
5P&1a&2a&2b&2c&2d&2e&2f&2g&2h&2i&2j&2k&2l&2m&2n&2o&2p&2q&2r&2s&2t&2u&2v\\
\hline
\rule{0pt}{5pt}X.1&1&1&1&1&1&1&1&1&1&1&1&1&1&1&1&1&1&1&1&1&1&1&1\\
\end{array}
$$}}
\filbreak
\noindent Character table of $D(\Co_2)$ (continued){
\setlength{\arraycolsep}{1mm}
\renewcommand{\baselinestretch}{0.5}
{\tiny \setlength{\arraycolsep }{0,3mm}
$$
\begin{array}{r|rrrrrrrrrrrrrrrrrrrrrrrrr}
2&11&11&7&5&14&13&14&12&12&13&11&12&12&12&12&12&12&12&12&12&12&12&11&11\\
3&.&.&2&2&1&1&.&1&1&.&1&.&.&.&.&.&.&.&.&.&.&.&.&.\\
5&.&.&.&.&.&.&.&.&.&.&.&.&.&.&.&.&.&.&.&.&.&.&.&.\\
\hline
\rule[-2pt]{0pt}{7pt}&2w&2x&3a&3b&4_{1}&4_{2}&4_{3}&4_{4}&4_{5}&4_{6}&4_{7}&4_{8}&4_{9}&4_{10}&4_{11}
&4_{12}&4_{13}&4_{14}&4_{15}&4_{16}&4_{17}&4_{18}&4_{19}&4_{20}\\
\hline
\rule{0pt}{5pt}2P&1a&1a&3a&3b&2a&2d&2a&2d&2d&2d&2d&2e&2f&2d&2e&2d&2d&2a&2d&2d&2f&2d&2e&2d\\
3P&2w&2x&1a&1a&4_{1}&4_{2}&4_{3}&4_{4}&4_{5}&4_{6}&4_{7}&4_{8}&4_{9}&4_{10}&4_{11}
&4_{12}&4_{13}&4_{14}&4_{15}&4_{16}&4_{17}&4_{18}&4_{19}&4_{20}\\
\rule[-2pt]{0pt}{5pt}5P&2w&2x&3a&3b&4_{1}&4_{2}&4_{3}&4_{4}&4_{5}&4_{6}&4_{7}&4_{8}&4_{9}&4_{10}&4_{11}
&4_{12}&4_{13}&4_{14}&4_{15}&4_{16}&4_{17}&4_{18}&4_{19}&4_{20}\\
\hline
\rule{0pt}{5pt}X.1&1&1&1&1&1&1&1&1&1&1&1&1&1&1&1&1&1&1&1&1&1&1&1&1\\
\end{array}
$$}}%
\filbreak
\noindent Character table of $D(\Co_2)$ (continued){
\setlength{\arraycolsep}{1mm}
\renewcommand{\baselinestretch}{0.5}
{\tiny \setlength{\arraycolsep }{0,3mm}
$$
\begin{array}{r|rrrrrrrrrrrrrrrrrrrrrrrrrr}
2&11&11&11&11&11&11&9&9&9&9&10&10&10&10&10&10&10&10&10&10&10&10&10&8&9\\
3&.&.&.&.&.&.&1&1&1&1&.&.&.&.&.&.&.&.&.&.&.&.&.&1&.\\
5&.&.&.&.&.&.&.&.&.&.&.&.&.&.&.&.&.&.&.&.&.&.&.&.&.\\
\hline
\rule[-2pt]{0pt}{7pt}&4_{21}&4_{22}&4_{23}&4_{24}&4_{25}&4_{26}&4_{27}&4_{28}&4_{29}&4_{30}&4_{31}&4_{32}&
4_{33}&4_{34}&4_{35}&4_{36}&4_{37}&4_{38}&4_{39}&4_{40}&4_{41}&4_{42}&4_{43}&4_{44}
&4_{45}\\
\hline
\rule{0pt}{5pt}2P&2e&2f&2e&2d&2d&2d&2g&2g&2g&2g&2h&2h&2e&2h&2e&2e&2d&2d&2d&2f&2d&2e&2h&2h&2h\\
3P&4_{21}&4_{22}&4_{23}&4_{24}&4_{25}&4_{26}&4_{27}&4_{28}&4_{29}&4_{30}&4_{31}&4_{32}&
4_{33}&4_{34}&4_{35}&4_{36}&4_{37}&4_{38}&4_{39}&4_{40}&4_{41}&4_{42}&4_{43}&4_{44}
&4_{45}\\
\rule[-2pt]{0pt}{5pt}5P&4_{21}&4_{22}&4_{23}&4_{24}&4_{25}&4_{26}&4_{27}&4_{28}&4_{29}&4_{30}&4_{31}&4_{32}&
4_{33}&4_{34}&4_{35}&4_{36}&4_{37}&4_{38}&4_{39}&4_{40}&4_{41}&4_{42}&4_{43}&4_{44}
&4_{45}\\
\hline
\rule{0pt}{5pt}X.1&1&1&1&1&1&1&1&1&1&1&1&1&1&1&1&1&1&1&1&1&1&1&1&1&1\\
\end{array}
$$}}%
\filbreak
\noindent Character table of $D(\Co_2)$ (continued){
\setlength{\arraycolsep}{1mm}
\renewcommand{\baselinestretch}{0.5}
{\tiny \setlength{\arraycolsep }{0,3mm}
$$
\begin{array}{r|rrrrrrrrrrrrrrrrrrrrrrrrrrrrr}
2&9&8&8&8&8&8&8&8&8&8&8&8&8&8&8&2&7&5&5&5&5&5&6&5&5&5&5&5\\
3&.&.&.&.&.&.&.&.&.&.&.&.&.&.&.&.&2&2&2&2&2&2&1&1&1&1&1&1\\
5&.&.&.&.&.&.&.&.&.&.&.&.&.&.&.&1&.&.&.&.&.&.&.&.&.&.&.&.\\
\hline
\rule[-2pt]{0pt}{7pt}&4_{46}&4_{47}&4_{48}&4_{49}&4_{50}&4_{51}&4_{52}&4_{53}&4_{54}&4_{55}&4_{56}&4_{57}
&4_{58}&4_{59}&4_{60}&5a&6a&6b&6c&6d&6e&6f&6g&6h&6i&6j&6k&6l\\
\hline
\rule{0pt}{5pt}2P&2h&2l&2t&2t&2t&2g&2t&2t&2g&2t&2t&2t&2q&2q&2h&5a&3a&3b&3b&3b&3a&3a&3a&3b&3b&3a&3a&3b\\
\rule{0pt}{5pt}3P&\,4_{46}&4_{47}&4_{48}&4_{49}&4_{50}&4_{51}&4_{52}&4_{53}&4_{54}&4_{55}&4_{56}&4_{57}
&4_{58}&4_{59}&4_{60}&5a&2a&2c&2b&2a&2b&2c&2d&2g&2i&2j&2h&2f\\
\rule[-2pt]{0pt}{5pt}5P&\,4_{46}&4_{47}&4_{48}&4_{49}&4_{50}&4_{51}&4_{52}&4_{53}&4_{54}&4_{55}&4_{56}&4_{57}
&4_{58}&4_{59}&4_{60}&1a&6a&6b&6c&6d&6e&6f&6g&6h&6i&6j&6k&6l\\
\hline
\rule{0pt}{5pt}X.1&1&1&1&1&1&1&1&1&1&1&1&1&1&1&1&1&1&1&1&1&1&1&1&1&1&1&1&1\\
\end{array}
$$}}%
\filbreak
\noindent Character table of $D(\Co_2)$ (continued){
\setlength{\arraycolsep}{1mm}
\renewcommand{\baselinestretch}{0.5}
{\tiny \setlength{\arraycolsep }{0,3mm}
$$
\begin{array}{r|rrrrrrrrrrrrrrrrrrrrrrrrrrrr}
2&5&5&4&4&4&4&4&8&9&9&8&8&8&8&8&8&7&7&7&7&7&7&7&7&7&6&6\\
3&1&1&1&1&1&1&1&1&.&.&.&.&.&.&.&.&.&.&.&.&.&.&.&.&.&.&.\\
5&.&.&.&.&.&.&.&.&.&.&.&.&.&.&.&.&.&.&.&.&.&.&.&.&.&.&.\\
\hline
\rule[-2pt]{0pt}{7pt}&\,6m&6n&6o&6p&6q&6r&6s&8a&8b&8c&8d&8e&8f&8g&8h&8i&8j&8k&8l&8m&8n&8o&8p&8q&8r&8s&
8t\\
\hline
\rule{0pt}{5pt}2P&3b&3a&3b&3b&3b&3a&3b&4_{1}&4_{3}&4_{3}&4_{6}&4_{3}&4_{1}&4_{1}&4_{6}&4_{3}&4
_{19}&4_{20}&4_{6}&4_{20}&4_{20}&4_{22}&4_{22}&4_{20}&4_{19}&4_{8}&4_{9}\\
3P&2e&2k&2m&2n&2o&2s&2p&8a&8b&8c&8d&8e&8f&8g&8h&8i&8j&8k&8l&8m&8n&8o&8p&8q&8r&8
s&8t\\
5P&6m&6n&6o&6p&6q&6r&6s&8a&8b&8c&8d&8e&8f&8g&8h&8i&8j&8k&8l&8m&8n&8o&8p&8q&8r&8
s&8t\\
\hline
\rule{0pt}{5pt}X.1&1&1&1&1&1&1&1&1&1&1&1&1&1&1&1&1&1&1&1&1&1&1&1&1&1&1&1\\
\end{array}
$$}}%
\filbreak
\noindent Character table of $D(\Co_2)$ (continued){
\setlength{\arraycolsep}{1mm}
\renewcommand{\baselinestretch}{0.5}
{\tiny \setlength{\arraycolsep }{0,3mm}
$$
\begin{array}{r|rrrrrrrrrrrrrrrrrrrrrr}
2&6&6&6&6&6&2&2&2&5&5&5&4&4&4&4&4&4&3&5&5&3\\
3&.&.&.&.&.&.&.&.&1&1&1&1&1&1&1&1&1&1&.&.&1\\
5&.&.&.&.&.&1&1&1&.&.&.&.&.&.&.&.&.&.&.&.&.\\
\hline
\rule[-2pt]{0pt}{6pt}&8u&8v&8w&8x&8y&10a&10b&10c&12a&12b&12c&12d&12e&12f&12g&12h&12i&12j&16a&16b&24a
\\
\hline
\rule{0pt}{5pt}2P&\,4_{32}&4_{32}&4_{31}&4_{31}&4_{24}&5a&5a&5a&6a&6g&6g&6h&6h&6h&6g&6g&6h&6k&8b&8c&12a\\
\rule[-2pt]{0pt}{5pt}3P&8u&8v&8w&8x&8y&10a&10b&10c&4_{1}&4_{4}&4_{5}&4_{27}&4_{28}&4_{29}&4_{7}&4_{2}&4_{30}
&4_{44}&16a&16b&8a\\
\rule{0pt}{5pt}5P&8u&8v&8w&8x&8y&2b&2c&2a&12a&12b&12c&12d&12e&12f&12g&12h&12i&12j&16a&16b&24a\\
\hline
\rule{0pt}{5pt}X.1&1&1&1&1&1&1&1&1&1&1&1&1&1&1&1&1&1&1&1&1&1\\
\end{array}
$$}}%
\filbreak
\noindent Character table of $D(\Co_2)$ (continued){
\setlength{\arraycolsep}{1mm}
\renewcommand{\baselinestretch}{0.5}
{\tiny \setlength{\arraycolsep }{0,3mm}
$$
\begin{array}{r|rrrrrrrrrrrrrrrrrrrrrr}
2&18&18&14&14&17&16&16&14&14&14&14&14&15&13&13&13&13&14&13&11&12\\
3&2&2&2&2&1&1&1&1&1&1&1&1&.&1&1&1&1&.&.&1&.\\
5&1&1&1&1&.&.&.&.&.&.&.&.&.&.&.&.&.&.&.&.&.\\
\hline
\rule[-2pt]{0pt}{6pt}&1a&2a&2b&2c&2d&2e&2f&2g&2h&2i&2j&2k&2l&2m&2n&2o&2p&2q&2r&2s&2t\\
\hline
\rule{0pt}{5pt}2P&1a&1a&1a&1a&1a&1a&1a&1a&1a&1a&1a&1a&1a&1a&1a&1a&1a&1a&1a&1a&1a\\
3P&1a&2a&2b&2c&2d&2e&2f&2g&2h&2i&2j&2k&2l&2m&2n&2o&2p&2q&2r&2s&2t\\
5P&1a&2a&2b&2c&2d&2e&2f&2g&2h&2i&2j&2k&2l&2m&2n&2o&2p&2q&2r&2s&2t\\
\hline
\rule{0pt}{5pt}X.85&120&120&.&.&-8&24&24&.&-8&.&24&24&-8&.&.&.&.&-8&-8&8&.\\
\end{array}
$$}}%
\filbreak
\noindent Character table of $D(\Co_2)$ (continued){
\setlength{\arraycolsep}{1mm}
\renewcommand{\baselinestretch}{0.5}
{\tiny \setlength{\arraycolsep }{0,3mm}
$$
\begin{array}{r|rrrrrrrrrrrrrrrrrrrrrrr}
2&12&12&11&11&7&5&14&13&14&12&12&13&11&12&12&12&12&12&12&12&12&12\\
3&.&.&.&.&2&2&1&1&.&1&1&.&1&.&.&.&.&.&.&.&.&.\\
5&.&.&.&.&.&.&.&.&.&.&.&.&.&.&.&.&.&.&.&.&.&.\\
\hline
\rule[-2pt]{0pt}{6pt}&2u&2v&2w&2x&3a&3b&4_{1}&4_{2}&4_{3}&4_{4}&4_{5}&4_{6}&4_{7}&4_{8}&4_{9}&4_{10}
&4_{11}&4_{12}&4_{13}&4_{14}&4_{15}&4_{16}\\
\hline
\rule{0pt}{5pt}2P&1a&1a&1a&1a&3a&3b&2a&2d&2a&2d&2d&2d&2d&2e&2f&2d&2e&2d&2d&2a&2d&2d\\
\rule[-2pt]{0pt}{5pt}3P&2u&2v&2w&2x&1a&1a&4_{1}&4_{2}&4_{3}&4_{4}&4_{5}&4_{6}&4_{7}&4_{8}&4_{9}&4_{10}
&4_{11}&4_{12}&4_{13}&4_{14}&4_{15}&4_{16}\\
\rule[-2pt]{0pt}{5pt}5P&2u&2v&2w&2x&3a&3b&4_{1}&4_{2}&4_{3}&4_{4}&4_{5}&4_{6}&4_{7}&4_{8}&4_{9}&4_{10}
&4_{11}&4_{12}&4_{13}&4_{14}&4_{15}&4_{16}\\
\hline
\rule{0pt}{5pt}X.85&-8&.&8&8&6&.&8&.&8&-8&-8&.&.&4&-4&-8&4&-8&.&8&-8&.\\
\end{array}
$$}}%
\filbreak
\noindent Character table of $D(\Co_2)$ (continued){
\setlength{\arraycolsep}{1mm}
\renewcommand{\baselinestretch}{0.5}
{\tiny \setlength{\arraycolsep }{0,3mm}
$$
\begin{array}{r|rrrrrrrrrrrrrrrrrrrrrrrrrr}
2&12&12&11&11&11&11&11&11&11&11&9&9&9&9&10&10&10&10&10&10&10&10&10&10&10\\
3&.&.&.&.&.&.&.&.&.&.&1&1&1&1&.&.&.&.&.&.&.&.&.&.&.\\
5&.&.&.&.&.&.&.&.&.&.&.&.&.&.&.&.&.&.&.&.&.&.&.&.&.\\
\hline
\rule[-2pt]{0pt}{7pt}&4_{17}&4_{18}&4_{19}&4_{20}&4_{21}&4_{22}&4_{23}&4_{24}&4_{25}&4_{26}&4_{27}&
4_{28}&4_{29}&4_{30}&4_{31}&4_{32}&4_{33}&4_{34}&4_{35}&4_{36}&4_{37}&4_{38}&
4_{39}&4_{40}&4_{41}\\
\hline
\rule{0pt}{5pt}2P&2f&2d&2e&2d&2e&2f&2e&2d&2d&2d&2g&2g&2g&2g&2h&2h&2e&2h&2e&2e&2d&2d&2d&2f&2d\\
3P&4_{17}&4_{18}&4_{19}&4_{20}&4_{21}&4_{22}&4_{23}&4_{24}&4_{25}&4_{26}&4_{27}&
4_{28}&4_{29}&4_{30}&4_{31}&4_{32}&4_{33}&4_{34}&4_{35}&4_{36}&4_{37}&4_{38}&
4_{39}&4_{40}&4_{41}\\
\rule[-3pt]{0pt}{7pt}5P&4_{17}&4_{18}&4_{19}&4_{20}&4_{21}&4_{22}&4_{23}&4_{24}&4_{25}&4_{26}&4_{27}&
4_{28}&4_{29}&4_{30}&4_{31}&4_{32}&4_{33}&4_{34}&4_{35}&4_{36}&4_{37}&4_{38}&
4_{39}&4_{40}&4_{41}\\
\hline
\rule{0pt}{5pt}X.85&-4&-8&12&.&4&4&4&.&.&.&.&.&.&.&.&.&-4&.&-4&-4&.&.&.&4&.\\
\end{array}
$$}}%
\filbreak
\noindent Character table of $D(\Co_2)$ (continued){
\setlength{\arraycolsep}{1mm}
\renewcommand{\baselinestretch}{0.5}
{\tiny \setlength{\arraycolsep }{0,3mm}
$$
\begin{array}{r|rrrrrrrrrrrrrrrrrrrrrrrrrrrrr}
2&10&10&8&9&9&8&8&8&8&8&8&8&8&8&8&8&8&8&8&2&7&5&5&5&5&5&6&5\\
3&.&.&1&.&.&.&.&.&.&.&.&.&.&.&.&.&.&.&.&.&2&2&2&2&2&2&1&1\\
5&.&.&.&.&.&.&.&.&.&.&.&.&.&.&.&.&.&.&.&1&.&.&.&.&.&.&.&.\\
\hline
\rule[-2pt]{0pt}{7pt}&4_{42}&4_{43}&4_{44}&4_{45}&4_{46}&4_{47}&4_{48}&4_{49}&4_{50}&4_{51}&4_{52}&4_{53}
&4_{54}&4_{55}&4_{56}&4_{57}&4_{58}&4_{59}&4_{60}&5a&6a&6b&6c&6d&6e&6f&6g&6h\\
\hline
\rule{0pt}{5pt}2P&2e&2h&2h&2h&2h&2l&2t&2t&2t&2g&2t&2t&2g&2t&2t&2t&2q&2q&2h&5a&3a&3b&3b&3b&3a&3a&3a&3b\\
\rule[-2pt]{0pt}{5pt}3P&\,4_{42}&4_{43}&4_{44}&4_{45}&4_{46}&4_{47}&4_{48}&4_{49}&4_{50}&4_{51}&4_{52}&4_{53}
&4_{54}&4_{55}&4_{56}&4_{57}&4_{58}&4_{59}&4_{60}&5a&2a&2c&2b&2a&2b&2c&2d&2g\\
\rule[-2pt]{0pt}{5pt}5P&\,4_{42}&4_{43}&4_{44}&4_{45}&4_{46}&4_{47}&4_{48}&4_{49}&4_{50}&4_{51}&4_{52}&4_{53}
&4_{54}&4_{55}&4_{56}&4_{57}&4_{58}&4_{59}&4_{60}&1a&6a&6b&6c&6d&6e&6f&6g&6h\\
\hline
\rule{0pt}{5pt}X.85&-4&.&.&.&.&.&.&.&.&.&.&.&.&.&.&.&.&.&.&.&6&.&.&.&.&.&-2&.\\
\end{array}
$$}}%
\filbreak
\noindent Character table of $D(\Co_2)$ (continued){
\setlength{\arraycolsep}{1mm}
\renewcommand{\baselinestretch}{0.5}
{\tiny \setlength{\arraycolsep }{0,3mm}
$$
\begin{array}{r|rrrrrrrrrrrrrrrrrrrrrrrrrrrr}
2&5&5&5&5&5&5&4&4&4&4&4&8&9&9&8&8&8&8&8&8&7&7&7&7&7&7&7\\
3&1&1&1&1&1&1&1&1&1&1&1&1&.&.&.&.&.&.&.&.&.&.&.&.&.&.&.\\
5&.&.&.&.&.&.&.&.&.&.&.&.&.&.&.&.&.&.&.&.&.&.&.&.&.&.&.\\
\hline
\rule[-2pt]{0pt}{7pt}&6i&6j&6k&6l&6m&6n&6o&6p&6q&6r&6s&8a&8b&8c&8d&8e&8f&8g&8h&8i&8j&8k&8l&8m&8n&8o&
8p\\
\hline
\rule{0pt}{5pt}2P&3b&3a&3a&3b&3b&3a&3b&3b&3b&3a&3b&4_{1}&4_{3}&4_{3}&4_{6}&4_{3}&4_{1}&4_{1}&4_{6}
&4_{3}&4_{19}&4_{20}&4_{6}&4_{20}&4_{20}&4_{22}&4_{22}\\
\rule{0pt}{5pt}3P&2i&2j&2h&2f&2e&2k&2m&2n&2o&2s&2p&8a&8b&8c&8d&8e&8f&8g&8h&8i&8j&8k&8l&8m&8n&8o&8p\\
\rule[-2pt]{0pt}{7pt}5P&6i&6j&6k&6l&6m&6n&6o&6p&6q&6r&6s&8a&8b&8c&8d&8e&8f&8g&8h&8i&8j&8k&8l&8m&8n&8o&8p\\
\hline
\rule{0pt}{5pt}X.85&.&.&-2&.&.&.&.&.&.&2&.&.&.&.&.&.&.&.&.&.&2&.&.&.&.&-2&-2\\
\end{array}
$$}}%
\filbreak
\noindent Character table of $D(\Co_2)$ (continued){
\setlength{\arraycolsep}{1mm}
\renewcommand{\baselinestretch}{0.5}
{\tiny \setlength{\arraycolsep }{0,3mm}
$$
\begin{array}{r|rrrrrrrrrrrrrrrrrrrrrrrrrr}
2&7&7&6&6&6&6&6&6&6&2&2&2&5&5&5&4&4&4&4&4&4&3&5&5&3\\
3&.&.&.&.&.&.&.&.&.&.&.&.&1&1&1&1&1&1&1&1&1&1&.&.&1\\
5&.&.&.&.&.&.&.&.&.&1&1&1&.&.&.&.&.&.&.&.&.&.&.&.&.\\
\hline
\rule[-2pt]{0pt}{7pt}&8q&8r&8s&8t&8u&8v&8w&8x&8y&10a&10b&10c&12a&12b&12c&12d&12e&12f&12g&12h&12i&12j&16a&16b&24a\\
\hline
\rule{0pt}{5pt}2P&\,4_{20}&4_{19}&4_{8}&4_{9}&4_{32}&4_{32}&4_{31}&4_{31}&4_{24}&5a&5a&5a&6a&6g&6g&6h
&6h&6h&6g&6g&6h&6k&8b&8c&12a\\
\rule{0pt}{5pt}3P&8q&8r&8s&8t&8u&8v&8w&8x&8y&10a&10b&10c&4_{1}&4_{4}&4_{5}&4_{27}&4_{28}&4_{29}&4_{7}
&4_{2}&4_{30}&4_{44}&16a&16b&8a\\
\rule{0pt}{5pt}5P&8q&8r&8s&8t&8u&8v&8w&8x&8y&2b&2c&2a&12a&12b&12c&12d&12e&12f&12g&12h&12i&12j&16a&16b
&24a\\
\hline
\rule{0pt}{5pt}X.85&.&2&2&-2&.&.&.&.&.&.&.&.&2&-2&-2&.&.&.&.&.&.&.&.&.&.\\
\end{array}
$$}}%
\end{ctab}
\filbreak
\begin{ctab}\label{Co_2ct_E} Character table of $E_3 = E(\Co_2) = \langle x,y,e \rangle$
{ \setlength{\arraycolsep}{1mm}
\renewcommand{\baselinestretch}{0.5}
{\tiny \setlength{\arraycolsep}{0,3mm}
$$
\begin{array}{r|rrrrrrrrrrrrrrrrrrrrrrr}
2&18&18&17&15&14&14&14&13&12&14&7&14&13&14&12&12&13&11&12&11&11&11\\
3&2&2&1&2&1&1&1&1&.&.&2&1&1&.&1&1&.&1&.&.&.&.\\
5&1&1&.&1&.&.&.&.&1&.&.&.&.&.&.&.&.&.&.&.&.&.\\
7&1&.&1&.&1&1&.&.&.&.&.&.&.&.&.&.&.&.&.&.&.&.\\
11&1&.&.&.&.&.&.&.&.&.&.&.&.&.&.&.&.&.&.&.&.&.\\
\hline
\rule[-3pt]{0pt}{8pt}&1a&2a&2b&2c&2d&2e&2f&2g&2h&2i&3a&4a&4b&4c&4d&4e&4f&4g&4h&4i&4j&4k\\
\hline
\rule{0pt}{5pt}2P&1a&1a&1a&1a&1a&1a&1a&1a&1a&1a&3a&2a&2b&2a&2b&2b&2b&2b&2a&2b&2b&2b\\
3P&1a&2a&2b&2c&2d&2e&2f&2g&2h&2i&1a&4a&4b&4c&4d&4e&4f&4g&4h&4i&4j&4k\\
\end{array}
$$}}%
\filbreak
\noindent Character table of $E_3 = E(\Co_2)$ (continued) {
\setlength{\arraycolsep}{1mm}
\renewcommand{\baselinestretch}{0.5}
{\tiny \setlength{\arraycolsep}{0,3mm}
$$
\begin{array}{r|rrrrrrrrrrrrrrrrrrrrrrrrrrrrrrr}
2&8&10&10&10&10&8&9&9&8&8&8&3&7&5&5&6&5&5&5&4&2&2&8&9&9&8&8&8&8&8\\
3&.&.&.&.&.&1&.&.&.&.&.&.&2&2&2&1&1&1&1&1&.&.&1&.&.&.&.&.&.&.\\
5&1&.&.&.&.&.&.&.&.&.&.&1&.&.&.&.&.&.&.&.&.&.&.&.&.&.&.&.&.&.\\
7&.&.&.&.&.&.&.&.&.&.&.&.&.&.&.&.&.&.&.&.&1&1&.&.&.&.&.&.&.&.\\
11&.&.&.&.&.&.&.&.&.&.&.&.&.&.&.&.&.&.&.&.&.&.&.&.&.&.&.&.&.&.\\
\hline
\rule[-2pt]{0pt}{7pt}&4l&4m&4n&4o&4p&4q&4r&4s&4t&4u&4v&5a&6a&6b&6c&6d&6e&6f&6g&6h&7a&7b&8a&8b&8c&8d&
8e&8f&8g&8h\\
\hline
\rule{0pt}{5pt}2P&2c&2f&2f&2f&2f&2f&2f&2f&2i&2f&2i&5a&3a&3a&3a&3a&3a&3a&3a&3a&7a&7b&4a&4c&4c&4
f&4f&4a&4c&4c\\
3P&4l&4m&4n&4o&4p&4q&4r&4s&4t&4u&4v&5a&2a&2c&2a&2b&2d&2f&2e&2g&7b&7a&8a&8b&8c&8
d&8e&8f&8g&8h\\
\end{array}
$$}}%
\filbreak
\noindent Character table of $E_3 = E(\Co_2)$ (continued){
\setlength{\arraycolsep}{1mm}
\renewcommand{\baselinestretch}{0.5}
{\tiny \setlength{\arraycolsep }{0,3mm}
$$
\begin{array}{r|rrrrrrrrrrrrrrrrrrrrrrrrrrrr}
2&8&7&6&6&6&6&6&3&2&2&.&5&5&5&4&4&3&2&2&2&2&2&2&5&5&2&3\\
3&.&.&.&.&.&.&.&.&.&.&.&1&1&1&1&1&1&.&.&.&.&.&.&.&.&.&1\\
5&.&.&.&.&.&.&.&1&1&1&.&.&.&.&.&.&.&.&.&.&.&.&.&.&.&1&.\\
7&.&.&.&.&.&.&.&.&.&.&.&.&.&.&.&.&.&1&1&1&1&1&1&.&.&.&.\\
11&.&.&.&.&.&.&.&.&.&.&1&.&.&.&.&.&.&.&.&.&.&.&.&.&.&.&.\\
\hline
\rule[-2pt]{0pt}{7pt}&8i&8j&8k&8l&\,8m&8n&8o&10a&10b&10c&11a&12a&12b&12c&12d&12e&12f&14a&14b&14c&14d&14e
&14f&16a&16b&20a&24a\\
\hline
\rule{0pt}{5pt}2P&4a&4f&4m&4n&4m&4n&4i&5a&5a&5a&11a&6d&6d&6a&6d&6d&6f&7a&7b&7b&7a&7b&7a&8b&8c&
10a&12c\\
3P&8i&8j&8k&8l&8m&8n&8o&10a&10b&10c&11a&4d&4e&4a&4g&4b&4q&14b&14a&14d&14c&14f&1
4e&16a&16b&20a&8a\\
\end{array}
$$}}%
\noindent where $A = \frac{1}{2}(-1 +i\sqrt{7})$.
\end{ctab}
\filbreak
\begin{ctab} \label{Fi_22ct_H}Character table of $H(\Fi_{22}) = \langle x,y,h \rangle$
{ \setlength{\arraycolsep}{1mm}
\renewcommand{\baselinestretch}{0.5}
{\tiny \setlength{\arraycolsep}{0,3mm}
$$
\begin{array}{r|rrrrrrrrrrrrrrrrrrrrrrr}
2&17&17&16&15&15&16&16&14&10&12&12&13&11&10&7&8&5&12&10&12&10&10\\
3&4&4&4&2&2&1&1&2&2&1&1&.&1&1&4&3&3&3&2&1&1&1\\
5&1&1&1&.&.&.&.&.&1&.&.&.&.&.&.&.&.&.&1&.&.&.\\
\hline
\rule[-2pt]{0pt}{7pt}&1a&2a&2b&2c&2d&2e&2f&2g&2h&2i&2j&2k&2l&2m&3a&3b&3c&4_{1}&4_{2}&4_{3}&4_{4}&4_{5}\\
\hline
\rule{0pt}{5pt}2P&1a&1a&1a&1a&1a&1a&1a&1a&1a&1a&1a&1a&1a&1a&3a&3b&3c&2a&2a&2a&2e&2a\\
3P&1a&2a&2b&2c&2d&2e&2f&2g&2h&2i&2j&2k&2l&2m&1a&1a&1a&4_{1}&4_{2}&4_{3}&4_{4}&4_{5}\\
5P&1a&2a&2b&2c&2d&2e&2f&2g&2h&2i&2j&2k&2l&2m&3a&3b&3c&4_{1}&4_{2}&4_{3}&4_{4}&4_{5}\\
\hline
\rule{0pt}{5pt}X.1&1&1&1&1&1&1&1&1&1&1&1&1&1&1&1&1&1&1&1&1&1&1\\
\end{array}
$$}}%
\filbreak
\noindent Character table of $H(\Fi_{22})$ (continued){
\setlength{\arraycolsep}{1mm}
\renewcommand{\baselinestretch}{0.5}
{\tiny \setlength{\arraycolsep }{0,3mm}
$$
\begin{array}{r|rrrrrrrrrrrrrrrrrrrrrrrrrr}
2&10&10&10&10&10&11&11&9&9&9&9&8&8&9&9&9&9&8&8&8&8&7&3&7&7\\
3&1&1&1&1&1&.&.&1&1&1&1&1&1&.&.&.&.&.&.&.&.&.&.&4&4\\
5&.&.&.&.&.&.&.&.&.&.&.&.&.&.&.&.&.&.&.&.&.&.&1&.&.\\
\hline
\rule[-2pt]{0pt}{7pt}&4_{6}&4_{7}&4_{8}&4_{9}&4_{10}&4_{11}&4_{12}&4_{13}&4_{14}&4_{15}&4_{16}&4_{17}
&4_{18}&4_{19}&4_{20}&4_{21}&4_{22}&4_{23}&4_{24}&4_{25}&4_{26}&4_{27}&5a&6a&6b
\\
\hline
\rule{0pt}{5pt}2P\,&2e&2e&2e&2e&2e&2e&2e&2c&2c&2g&2g&2i&2j&2e&2e&2f&2e&2i&2i&2j&2i&2i&5a&3a&3a\\
3P\,&4_{6}&4_{7}&4_{8}&4_{9}&4_{10}&4_{11}&4_{12}&4_{13}&4_{14}&4_{15}&4_{16}&4_{17}
&4_{18}&4_{19}&4_{20}&4_{21}&4_{22}&4_{23}&4_{24}&4_{25}&4_{26}&4_{27}&5a&2a&
2b\\
\rule[-2pt]{0pt}{5pt}5P\,&4_{6}&4_{7}&4_{8}&4_{9}&4_{10}&4_{11}&4_{12}&4_{13}&4_{14}&4_{15}&4_{16}&4_{17}
&4_{18}&4_{19}&4_{20}&4_{21}&4_{22}&4_{23}&4_{24}&4_{25}&4_{26}&4_{1}&1a&6a&6c\\
\hline
\rule{0pt}{5pt}X.1&1&1&1&1&1&1&1&1&1&1&1&1&1&1&1&1&1&1&1&1&1&1&1&1&1\\
\end{array}
$$}}%
\filbreak
\noindent Character table of $H(\Fi_{22})$ (continued){
\setlength{\arraycolsep}{1mm}
\renewcommand{\baselinestretch}{0.5}
{\tiny \setlength{\arraycolsep }{0,3mm}
$$
\begin{array}{r|rrrrrrrrrrrrrrrrrrrrrrrrrrrrrr}
2&7&8&7&7&7&7&7&5&6&6&4&7&7&5&5&4&4&4&4&5&5&5&3&4&7&7&7&7&7\\
3&4&3&3&2&2&2&2&3&2&2&3&1&1&2&2&2&2&2&2&1&1&1&2&1&1&1&.&.&.\\
5&.&.&.&.&.&.&.&.&.&.&.&.&.&.&.&.&.&.&.&.&.&.&.&.&.&.&.&.&.\\
\hline
\rule[-2pt]{0pt}{7pt}&6c&6d&6e&6f&6g&6h&6i&6j&6k&6l&6m&6n&6o&6p&6q&6r&6s&6t&6u&6v&6w&6x&6y&6z&8a&8b&
8c&8d&8e\\
\hline
\rule{0pt}{5pt}2P\,&3a&3b&3b&3a&3a&3a&3a&3c&3b&3b&3c&3b&3b&3b&3b&3c&3c&3c&3c&3c&3c&3b&3c&3c&4_{1}
&4_{1}&4_{11}&4_{11}&4_{11}\\
3P\,&2b&2a&2b&2c&2g&2g&2d&2a&2c&2d&2b&2e&2f&2h&2g&2g&2g&2d&2c&2i&2j&2m&2h&2l&8a&8b&8c&8d&8e\\
5P\,&6b&6d&6e&6f&6h&6g&6i&6j&6k&6l&6m&6n&6o&6p&6q&6s&6r&6t&6u&6v&6w&6x&6y&6z&8a&8b&8c&8d&8e\\
\hline
\rule{0pt}{5pt}X.1&1&1&1&1&1&1&1&1&1&1&1&1&1&1&1&1&1&1&1&1&1&1&1&1&1&1&1&1&1\\
\end{array}
$$}}%
\filbreak
\noindent Character table of $H(\Fi_{22})$ (continued){
\setlength{\arraycolsep}{1mm}
\renewcommand{\baselinestretch}{0.5}
{\tiny \setlength{\arraycolsep }{0,3mm}
$$
\begin{array}{r|rrrrrrrrrrrrrrrrrrrrrrrrrrr}
2&7&7&6&5&5&5&2&3&2&2&5&6&6&5&5&5&5&5&5&5&5&5&5&5&5&5\\
3&.&.&.&.&.&.&2&.&.&.&3&2&2&2&1&1&1&1&1&1&1&1&1&1&1&1\\
5&.&.&.&.&.&.&.&1&1&1&.&.&.&.&.&.&.&.&.&.&.&.&.&.&.&.\\
\hline
\rule[-2pt]{0pt}{7pt}&8f&8g&8h&8i&8j&8k&9a&10a&10b&10c&12a&12b&12c&12d&12e&12f&12g&12h&12i&12j&12k&12l
&12m&12n&12o&12p\\
\hline
\rule{0pt}{5pt}2P\,&4_{3}&4_{11}&4_{12}&4_{13}&4_{14}&4_{21}&9a&5a&5a&5a&6a&6d&6d&6d&6n&6d&6n&6n&
6n&6g&6h&6a&6g&6h&6n&6n\\
\rule{0pt}{5pt}3P\,&8f&8g&8h&8i&8j&8k&3a&10a&10b&10c&4_{1}&4_{1}&4_{1}&4_{2}&4_{4}&4_{5}&4_{6}&4_{7}
&4_{8}&4_{15}&4_{15}&4_{3}&4_{16}&4_{16}&4_{10}&4_{9}\\
\rule{0pt}{5pt}5P\,&8f&8g&8h&8i&8j&8k&9a&2a&2h&2b&12a&12b&12c&12d&12e&12f&12g&12h&12i&12k&12j&12l
&12n&12m&12o&12p\\
\hline
\rule{0pt}{5pt}X.1&1&1&1&1&1&1&1&1&1&1&1&1&1&1&1&1&1&1&1&1&1&1&1&1&1&1\\
\end{array}
$$}}%
\filbreak
\noindent Character table of $H(\Fi_{22})$ (continued){
\setlength{\arraycolsep}{1mm}
\renewcommand{\baselinestretch}{0.5}
{\tiny \setlength{\arraycolsep }{0,3mm}
$$
\begin{array}{r|rrrrrrrrrrrrrr}
2&3&4&4&3&3&5&5&2&2&2&2&4&4\\
3&2&1&1&1&1&.&.&2&2&2&.&1&1\\
5&.&.&.&.&.&.&.&.&.&.&1&.&.\\
\hline
\rule[-2pt]{0pt}{7pt}&12q&\,12r&\,12s&\,12t&\,12u&\,16a&\,16b&\,18a&\,18b&\,18c&\,20a&\,24a&\,24b\\
\hline
\rule{0pt}{5pt}2P&6j&6f&6f&6v&6w&8f&8f&9a&9a&9a&10a&12b&12b\\
\rule{0pt}{5pt}3P&4_{2}&4_{13}&4_{14}&4_{17}&4_{18}&16a&16b&6a&6b&6c&20a&8a&8b\\
\rule{0pt}{5pt}5P&12q&12r&12s&12t&12u&16b&16a&18a&18c&18b&4_{2}&24a&24b\\
\hline
\rule{0pt}{5pt}X.1&1&1&1&1&1&1&1&1&1&1&1&1&1\\
\end{array}
$$}}%
\filbreak
\noindent Character table of $H(\Fi_{22})$ (continued){
\setlength{\arraycolsep}{1mm}
\renewcommand{\baselinestretch}{0.5}
{\tiny \setlength{\arraycolsep }{0,3mm}
\begin{alignat*}{1}
&\begin{array}{r|rrrrrrrrrrrrrrrrrrrrrrr}
2&17&17&16&15&15&16&16&14&10&12&12&13&11&10&7&8&5&12&10&12&10&10\\
3&4&4&4&2&2&1&1&2&2&1&1&.&1&1&4&3&3&3&2&1&1&1\\
5&1&1&1&.&.&.&.&.&1&.&.&.&.&.&.&.&.&.&1&.&.&.\\
\hline
\rule[-2pt]{0pt}{7pt}&1a&2a&2b&2c&2d&2e&2f&2g&2h&2i&2j&2k&2l&2m&3a&3b&3c&4_{1}&4_{2}&4_{3}&4_{4}&4_{5}\\
\hline
\rule{0pt}{5pt}2P\,&1a&1a&1a&1a&1a&1a&1a&1a&1a&1a&1a&1a&1a&1a&3a&3b&3c&2a&2a&2a&2e&2a\\
\rule{0pt}{5pt}3P\,&1a&2a&2b&2c&2d&2e&2f&2g&2h&2i&2j&2k&2l&2m&1a&1a&1a&4_{1}&4_{2}&4_{3}&4_{4}&4_{5}\\
\rule{0pt}{5pt}5P\,&1a&2a&2b&2c&2d&2e&2f&2g&2h&2i&2j&2k&2l&2m&3a&3b&3c&4_{1}&4_{2}&4_{3}&4_{4}&4_{5}\\
\hline
\rule{0pt}{5pt}X.85&720&720&720&-48&-48&-48&-48&-48&.&.&.&16&.&.&-9&.&.&48&.&-16&.&.\\
\end{array}\\
\intertext{\normalsize\noindent Character table of $H(\Fi_{22})$ (continued)}
&\begin{array}{r|rrrrrrrrrrrrrrrrrrrrrrrrrr}
2&10&10&10&10&10&11&11&9&9&9&9&8&8&9&9&9&9&8&8&8&8&7&3&7&7\\
3&1&1&1&1&1&.&.&1&1&1&1&1&1&.&.&.&.&.&.&.&.&.&.&4&4\\
5&.&.&.&.&.&.&.&.&.&.&.&.&.&.&.&.&.&.&.&.&.&.&1&.&.\\
\hline
\rule[-3pt]{0pt}{8pt}&4_{6}&4_{7}&4_{8}&4_{9}&4_{10}&4_{11}&4_{12}&4_{13}&4_{14}&4_{15}&4_{16}&4_{17}
&4_{18}&4_{19}&4_{20}&4_{21}&4_{22}&4_{23}&4_{24}&4_{25}&4_{26}&4_{27}&5a&6a&6b\\
\hline
\rule{0pt}{5pt}2P\,&2e&2e&2e&2e&2e&2e&2e&2c&2c&2g&2g&2i&2j&2e&2e&2f&2e&2i&2i&2j&2i&2i&5a&3a&3a\\
3P\,&4_{6}&4_{7}&4_{8}&4_{9}&4_{10}&4_{11}&4_{12}&4_{13}&4_{14}&4_{15}&4_{16}&4_{17}
&4_{18}&4_{19}&4_{20}&4_{21}&4_{22}&4_{23}&4_{24}&4_{25}&4_{26}&4_{27}&5a&2a&2b\\
\rule[-2pt]{0pt}{5pt}5P\,&4_{6}&4_{7}&4_{8}&4_{9}&4_{10}&4_{11}&4_{12}&4_{13}&4_{14}&4_{15}&4_{16}&4_{17}
&4_{18}&4_{19}&4_{20}&4_{21}&4_{22}&4_{23}&4_{24}&4_{25}&4_{26}&4_{1}&1a&6a&6c\\
\hline
\rule{0pt}{5pt}X.85&.&.&.&.&.&.&.&-8&-8&8&8&.&.&.&.&.&.&.&.&.&.&.&.&-9&-9\\
\end{array}
\end{alignat*}}}
\filbreak
\noindent Character table of $H(\Fi_{22})$ (continued){
\setlength{\arraycolsep}{1mm}
\renewcommand{\baselinestretch}{0.5}
{\tiny \setlength{\arraycolsep }{0,3mm}
\begin{alignat*}{1}
&%
\\
%
\intertext{\normalsize\noindent Character table of $H(\Fi_{22})$ (continued)}
&%
\end{alignat*}}}
\filbreak
\noindent Character table of $H(\Fi_{22})$ (continued){
\setlength{\arraycolsep}{1mm}
\renewcommand{\baselinestretch}{0.5}
{\tiny \setlength{\arraycolsep }{0,3mm}
$$
%
$$}}%
\noindent where $ A = -6i\sqrt{3}$; $B = 2i\sqrt{3}$; $C = 1 +
i\sqrt{3}$; $D = -i\sqrt{3}$;
$E = -2(\zeta_8^3 + \zeta_8)$; $F = -1 + 2i\sqrt{3}$;
$G = -18i\sqrt{3}$; $H = 12i\sqrt{3}$; and $I =
4i\sqrt{3}$.
\end{ctab}

\filbreak
\begin{ctab}\label{Fi_22ct_D}Character table of $D(\Fi_{22}) = \langle x,y \rangle$
{ \setlength{\arraycolsep}{1mm}
\renewcommand{\baselinestretch}{0.5}
{\tiny \setlength{\arraycolsep}{0,3mm}
$$
%
$$}}%
\filbreak
\noindent Character table of $D(\Fi_{22})$ (continued){
\setlength{\arraycolsep}{1mm}
\renewcommand{\baselinestretch}{0.5}
{\tiny \setlength{\arraycolsep }{0,3mm}
$$
%
$$}}%
\filbreak
\noindent Character table of $D(\Fi_{22})$ (continued){
\setlength{\arraycolsep}{1mm}
\renewcommand{\baselinestretch}{0.5}
{\tiny \setlength{\arraycolsep }{0,3mm}
$$
%
$$}}%
\filbreak
\noindent Character table of $D(\Fi_{22})$ (continued){
\setlength{\arraycolsep}{1mm}
\renewcommand{\baselinestretch}{0.5}
{\tiny \setlength{\arraycolsep }{0,3mm}
$$
%
$$}}%
\filbreak
\noindent Character table of $D(\Fi_{22})$ (continued){
\setlength{\arraycolsep}{1mm}
\renewcommand{\baselinestretch}{0.5}
{\tiny \setlength{\arraycolsep }{0,3mm}
\begin{alignat*}{1}
&%
\\
%
\intertext{\normalsize\noindent Character table of $D(\Fi_{22})$ (continued)}
&%
\\
\intertext{\normalsize\noindent Character table of $D(\Fi_{22})$ (continued)}
&%
\\
\intertext{\normalsize\noindent Character table of $D(\Fi_{22})$ (continued)}
&%
\end{alignat*}}}%
\noindent where $A = 4i$, $B = i\sqrt{5}$, $C = -2i\sqrt{3}$, and
$D = -2(\zeta_8^3 + \zeta_8)$.
\end{ctab}

\filbreak
\begin{ctab}\label{Fi_22ct_E} Character table of $E_2 = E(\Fi_{22}) = \langle x,y,e \rangle$
{ \setlength{\arraycolsep}{1mm}
\renewcommand{\baselinestretch}{0.5}
{\tiny \setlength{\arraycolsep}{0,3mm}
$$
%
$$}}%
\filbreak
\noindent Character table of $E_2 = E(\Fi_{22})$ (continued){
\setlength{\arraycolsep}{1mm}
\renewcommand{\baselinestretch}{0.5}
{\tiny \setlength{\arraycolsep }{0,3mm}
$$
%
$$}}%
\noindent where $ A = \frac{1}{2}(-1 + i\sqrt{7})$, $B =
\frac{1}{2}(-1 + i\sqrt{11})$, $C = -i\sqrt{5}$, and
$D = -2(\zeta_8^3 + \zeta_8)$.
\end{ctab}
\filbreak

\section{Generating matrices of Fischer group $\mathfrak G_2 $}\label{sec:genmatrixFiG2}
\noindent The generating matrices are 
block matrices:
{\scriptsize
$$
\mathcal{H} = \left( \begin{array}{c|c} \mathcal{H}_1 & 0\\\hline
0 & \mathcal{H}_2
\end{array} \right),
\quad \mathcal{X} = \left( \begin{array}{c|c} \mathcal{X}_1 &
0\\\hline 0 & \mathcal{X}_2
\end{array} \right),
\quad \mathcal{Y} = \left( \begin{array}{c|c} \mathcal{Y}_1 &
0\\\hline 0 & \mathcal{Y}_2
\end{array} \right),
\quad \mathcal{E} = \left( \begin{array}{c|c} \mathcal{E}_1 &
\mathcal{E}_2\\ \hline \mathcal{E}_3 & \mathcal{E}_4
\end{array} \right),
$$}%
\noindent where $\mathcal{H}_1, \mathcal{X}_1, \mathcal{Y}_1 \in \GL_{46}
(13)$ and  $\mathcal{H}_2, \mathcal{X}_2, \mathcal{Y}_2 \in
\GL_{32}(13)$; while $\mathcal{E}_1 \in \Mat_{38,38} (13)$,
$\mathcal{E}_4 \in \Mat_{40,40} (13)$, $\mathcal{E}_2 \in
\Mat_{38, 40}(13)$, and $\mathcal{E}_3 \in \Mat_{40, 38}(13)$. Below, the
element $0$ of $\GF(13)$ is replaced by a dot.

\bigskip\bigskip

\noindent Matrix $\mathcal{H}_1$: { \tiny \[
\renewcommand{\arraystretch}{0.8}\tabcolsep .pt
%
\]}
\bigskip\bigskip
\end{appendix}

\filbreak

\end{document}